\documentclass{amsart}
\usepackage{amsmath}
\usepackage[foot]{amsaddr}
\usepackage{amsthm}
\usepackage{amssymb}
\usepackage{color}
\usepackage[dvipsnames]{xcolor}
\usepackage{enumitem}
\usepackage[colorlinks,linkcolor={blue},citecolor={blue},urlcolor={black}]{hyperref}

\theoremstyle{plain}
\newtheorem{theorem}{Theorem}
\newtheorem{lemma}{Lemma}[section]

\newtheorem{proposition}[lemma]{Proposition}

\newtheorem{definition}[lemma]{Definition}

\newtheorem*{definition*}{Definition}

\theoremstyle{remark}
\newtheorem{remark}[lemma]{Remark}

\newtheorem*{claim*}{Claim}
\newtheorem*{remark*}{Remark}
\newtheorem*{example*}{Example}
\newtheorem*{notation*}{Notation}

\numberwithin{equation}{section}

\def\R{{\mathbb R}}

\renewcommand{\d}{\delta}

\renewcommand{\phi}{\varphi}

\newcommand{\TTd}{{\mathbb{T}^d}}

\newcommand{\e}{\epsilon}
\newcommand{\TND}{\mathbb{T}^d_N}
\newcommand{\PP}{\mathbb{P}}

\newcommand{\cQ}{\mathcal{Q}}

\newcommand{\cM}{\mathcal{M}}
\newcommand{\cH}{\mathcal{H}}
\newcommand{\cB}{\mathcal{B}}
\newcommand{\cD}{\mathcal{D}}
\newcommand{\cW}{\mathcal{W}}
\newcommand{\cI}{\mathcal{I}}
\newcommand{\cL}{\mathcal{L}}
\newcommand{\cU}{\mathcal{U}}
\newcommand{\cDD}{\mathbb{D}}
\newcommand{\cS}{\mathcal{S}}

\newcommand{\cG}{\mathcal{G}}
\newcommand{\cX}{\mathcal{X}}
\newcommand{\cE}{\mathcal{E}}
\newcommand{\cA}{\mathcal{A}}

\newcommand{\RRd}{\mathbb{R}^d}

\newcommand{\RR}{\mathbb{R}}

\newcommand{\cF}{\mathcal{F}}

\newcommand{\cV}{\ensuremath{\mathcal{V}}}
\newcommand{\hs}{\ensuremath{\hspace{1cm}}}
\newcommand{\h}{\ensuremath{\hspace{0.1cm}}}
\newcommand{\EE}{\ensuremath{\mathbb{E}}}

\newcommand{\NN}{\ensuremath{\mathbb{N}}}

\newcommand{\indiq}{\hbox{\rm 1}{\hskip -2.8 pt}\hbox{\rm I}}

\newcommand{\cR}{{\mathcal{R}}}

\newcommand{\Law}{{\rm Law}}

\newcommand{\lllangle}{\left\langle \left\langle}
\newcommand{\rrrangle}{\right\rangle\right\rangle}

\newcommand{\tf}{{t_{\mathrm{fin}}}}

\usepackage[colorinlistoftodos,prependcaption,color=yellow,textsize=tiny]{todonotes}

\newcommand\QQ{\mathbb{Q}}

\usepackage{tikz}
\usetikzlibrary{backgrounds}
\usetikzlibrary{decorations.pathmorphing}
\usetikzlibrary{arrows}
\usetikzlibrary{shapes.geometric}
\usetikzlibrary{calc}
\usetikzlibrary {arrows.meta,bending,positioning}

\begin{document}
\title[Porous Medium Equation: LDP and Gradient Flow]{The Porous Medium Equation: Large Deviations and Gradient Flow with Degenerate and Unbounded Diffusion}

\author{Benjamin Gess}
\author{Daniel Heydecker}

\address{B. Gess: Fakult\"at f\"ur Mathematik, Universit\"at Bielefeld, 33615 Bielefeld, Germany \\ Max Planck Institute for Mathematics in the Sciences, 04103 Leipzig, Germany.}
\email{Benjamin.Gess@mis.mpg.de}
\address{D. Heydecker: Department of Mathematics, Imperial College London, London, SW7 2AZ.}
\email{d.heydecker@imperial.ac.uk}

\subjclass[2010]{35Q84,60F10 (primary), 60K35, 82B21, 82B31, 82B35.}

\keywords{Zero-range process, porous medium equation, large deviations, gradient flow.}

\begin{abstract}
The problem of deriving a gradient flow structure for the porous medium equation which is {\em thermodynamic}, in that it arises from the large deviations of some microscopic particle system is studied. To this end, a rescaled zero-range process with jump rate $g(k)=k^\alpha, \alpha>1$ is considered, and its hydrodynamic limit and dynamical large deviations are shown in the presence of both degenerate and unbounded diffusion. The key superexponential estimate is obtained using pathwise discretised regularity estimates in the spirit of the Aubin-Lions-Simons lemma. This allows to exhibit the porous medium equation as the gradient flow of the entropy in a thermodynamic metric via the energy-dissipation inequality. 
\end{abstract}

\newpage \maketitle

\def\jrnl{0} 

\section{Introduction}
\label{sec: intro}
While the derivation of nonlinear but uniformly parabolic equations from microscopic dynamics \cite{kipnis1989hydrodynamics}, fluctuations around these limits \cite{kipnis1990large,kipnis1998scaling,quastel1999large} and the corresponding canonical choice of a gradient flow structure \cite{dirr2016entropic,adams2011large,adams2013large} are now well-understood, much less\footnote{We refer to some recent works \cite{goncalves2009hydrodynamic,blondel2018convergence,goncalves2023exclusion}  and references therein; see the literature review in Section \ref{sec: Lit}.} is known for equations with either degenerate, or unbounded, diffusivity.  Indeed, for the model case of the \emph{porous medium equation} (PME)
\begin{equation}
	\label{eq: PME} \partial_t u_t=\frac12\Delta(u_t^\alpha),\qquad x\in \TTd, \qquad \alpha>1 
\end{equation} 
multiple gradient flow structures
have been known since the works of Br\'ezis and Otto \cite{brezis1971monotonicity,otto2001geometry}, but it is not known which, if any, are {\em thermodynamic}, in that they arise through the large deviations of a microscopic model. A necessary first step to a rigorous answer is to identify the dynamical large deviations of a suitable microscopic model whose hydrodynamic limit is \eqref{eq: PME}. 

The degeneracy and unboundedness of the diffusivity pose serious challenges on the probabilistic side. Indeed, the key property of rapid relaxation into local equilibrium, which
underlies the usual arguments for the so-called superexponential estimate, becomes problematic for two reasons: Firstly, the presence of degenerate diffusivity leads to degenerate mixing rates, for which the property of rapid local equilibration underlying the `one-block
estimate' may break down. {Secondly, the faster-than-linear growth of the local mobility may lead to the explosion of the total mobility\footnote{See the discussion on page 4 below for more details.} in finite time.} For these reasons the established methods for proving the superexponential replacement lemma \cite{kipnis1998scaling,benois1995large,kipnis1989hydrodynamics} do not apply. 
A contribution of this work is to resolve these difficulties for a suitable model, leading to an LDP and to a thermodynamic gradient flow structure for \eqref{eq: PME}. 

In this work, in order to prove that the (formal) geometric picture we obtain is thermodynamic, we consider a rescaling of the zero-range process (ZRP) which converges to \eqref{eq: PME}. The first main contribution then is to resolve both the hydrodynamic limit and the large deviations for this rescaled ZRP and \eqref{eq: PME}, in the presence of both degenerate and unbounded diffusivity. To the best of the authors' knowledge, this represents the first time that large deviations around \eqref{eq: PME}, or any other equation with degenerate and unbounded diffusivity, have been identified for a particle system. A key technical tool developed in this work in order to overcome the difficulties described above, is a new approach to the superexponential estimate via {\em pathwise regularity estimates}.

The second main contribution of this work is the rigorous identification of a gradient flow structure for \eqref{eq: PME} arising from the large deviation picture, which thus completes the programme described above. As a result of the degenerate and unbounded diffusivity, and in contrast to \cite{dirr2016entropic}, one cannot define a global geometric structure induced by the rescaled ZRP. Instead, we rigorously validate the generalisation to the case of degenerate  and unbounded  diffusivity of the connection \cite{dirr2016entropic,adams2011large,adams2013large,mielke2014relation} between dynamic large deviations and the {\em entropy-dissipation inequality} (EDI). Since in non-degenerate cases, the EDI is an equivalent formulation of the gradient flow \cite{ambrosio2005gradient,sandier2004gamma,serfaty2011gamma}, this expresses the PME \eqref{eq: PME} as the gradient flow of a thermodynamic entropy in a degenerate geometry induced by the ZRP. The geometric structure which we derive is thermodynamic, since it is derived from the ZRP, while being more universal, as it describes the fluctuations of any other particle system with the same dynamic large deviations. The proof of the identity of the large deviations rate function and the EDI without restriction on regular paths is novel even for the case of nondegenerate and bounded diffusivity.

We first describe the rescaling of the ZRP with which we work. We start from the standard ZRP $\hat{\eta}^N(t, x)$, defined on the periodic lattice $\{0,1\dots,N-1\}^d$ with periodic boundary conditions and local jump rate $ g(k):= d k^\alpha$, see \cite{kipnis1989hydrodynamics,kipnis1990large}. In order to see macroscopic behaviour, as usual, we parabolically rescale time and space by $(N^{2}t, N^{}x)$. In addition, in order to make visible small masses, and, thus, degenerate diffusion, we rescale particle size and time by setting, for $x\in \TND=\{0, N^{-1},\dots 1-N^{-1}\}^d$ and $t\ge 0$, $\eta^N_t(x):=\chi_N\hat{\eta}^N(N^{2}\chi_N^{\alpha-1}t, Nx)$. In this way, $\chi_N$ plays the role of particle size, and at rate $$ d\eta^N_t(x)^\alpha N^2/\chi_N$$ the value at $x\in \TND$ is decreased by $\chi_N$, and the value at a neighbouring site is simultaneously increased by the same amount. 
 We impose the scaling relation on the parameters \begin{equation}
	\label{eq: scaling hypothesis} N^2\chi^{\min(1,\alpha/2)}_N\text{ is bounded as }N\to \infty.
\end{equation}  We identify the particle configuration with a nonnegative integrable function $\eta^N \in L^1_{\ge 0}(\TTd)$ by assigning the value $\eta^N(x)$ at all points $z\in c^N_x$ in a the box of side length $N^{-1}$ centred at $x\in \TND$. We will always work on a fixed, finite time interval $[0,\tf]$. With these conventions fixed, we can now describe the main results, which correspond to the strategy described below \eqref{eq: PME}.\paragraph{\textbf{Hydrodynamic Limit.}} The first step is to justify the implicit claim that the hydrodynamic limits of $\eta^N_\bullet$ are indeed governed by \eqref{eq: PME}. The following theorem resolves this under the condition only that $u_0\in L^1_{\ge 0}(\TTd)$ has finite entropy, and mild conditions on the initial data $\eta^N_0$.


\begin{theorem}[Hydrodynamic Limit]\label{th: hydrodynamic limit} Assume (\ref{eq: scaling hypothesis}). Fix $u_0\in L^1_{\ge 0}(\TTd)$ with finite entropy $\cH(u_0)<\infty$, and let $u_\bullet=(u_t)_{t\ge 0}$ be the unique weak solution to the PME (\ref{eq: PME}) starting at $u_0$. Let $\PP$ be a probability measure, let $\eta^N_\bullet$ be the ZRP with initial data satisfying \begin{equation}
	\label{eq: entropy UI hyp} \limsup_{M\to \infty} \limsup_{N\to \infty} \PP\left(\cH(\eta^N_0)>M\right)=0;
\end{equation} and, for all $\e>0$,
\begin{equation}
	\label{eq: weak convergence hyp} \langle \varphi, \eta^N_0\rangle \to \langle \varphi, u_0\rangle \text{ in probability, for all }\varphi \in C(\TTd).
\end{equation}  Then \begin{equation}
	\eta^N_\bullet \to u_\bullet
\end{equation} in probability in the topology of the Skorokhod space $\mathbb{D}:=\mathbb{D}\left([0,\tf],(L^1_{\ge 0}(\TTd))_{\rm w}\right)$, where $L^1_{\ge 0}(\TTd)$ is equipped with the weak topology, and in the topology of $L^\alpha([0,\tf]\times\TTd)$.	
\end{theorem} This theorem treats much more general initial data than other models considered in the literature; the works of Ekehaus and Sepp\"al\"ainen \cite{ekhaus1995stochastic} and Feng, Iscoe and Sepp\"al\"ainen \cite{feng1997microscopic} require a more regular initial datum $u_0\in C^{2+\e}(\TTd), \e>0$, and the work of Gon\c{c}alves, Landim and Toninelli \cite{goncalves2009hydrodynamic} require that $\inf_{x\in \TTd}u_0(x)>0$, so that the initial data remains in the region where (\ref{eq: PME}) is nondegenerate parabolic. 

\paragraph{\textbf{Large Deviations.}} Having found the hydrodynamic limit, we consider the dynamical large deviations of the ZRP $\eta^N_\bullet$ on a time interval $[0,\tf]$. This completes the probabilistic elements of the strategy described below \eqref{eq: PME}, in that we have identified the large deviations of a particle system around this limit. If  the initial data $\eta^N_0$ are sampled from an equilibrium distribution or a slowly-varying local equilibrium $\Pi^N_\rho$, see (\ref{eq: eq}), then the large deviations rate functional associated to large deviations of the initial data is the rescaled relative entropy $\alpha \cH_\rho(u_0)$ relative to $\rho$.  The other component is a dynamic cost $\mathcal{J}$ through the {\em skeleton equation} which extends the definitions of Benois, Kipnis and Landim \cite[p.70]{benois1995large}, \cite[Equation (5.1)]{kipnis1998scaling}. The {skeleton equation} around (\ref{eq: PME})  is\begin{equation}\label{eq: Sk}
				\partial_t u_t = \frac12\Delta(u_t^\alpha)-\nabla\cdot(u^{\alpha/2}_t g_t), \qquad g\in {L^2_{t,x}}:= L^2([0, \tf]\times\TTd, \RRd)
			\end{equation} where, since we often work with the space $L^2([0, \tf]\times\TTd, \RRd)$, we introduce the notation ${L^2_{t,x}}$ and $\|\cdot\|_{L^2_{t,x}}$ for the associated norm.  The dynamic cost of a trajectory $u_\bullet=(u_t)_{0\le t\le \tf}$ is then given by \begin{equation}\label{eq: dynamic cost}
		\mathcal{J}(u_\bullet):=\frac{1}{2}\inf\left\{\|g\|^2_{L^2_{t,x}}: \h u_\bullet \text{ solves the skeleton equation (\ref{eq: Sk}) for }g\right\}
	\end{equation} where, as for (\ref{eq: PME}) above, we use the notion of weak solutions in \cite{fehrman2019large}, recalled in Definition \ref{def: solutions}. The total rate function $\cI_\rho(u_\bullet)$ for a path $u_\bullet$ relative to an initial profile $\rho\in C(\TTd, (0,\infty))$ is thus given by \begin{equation}\label{eq: rate function}
		\cI_\rho(u_\bullet):=\alpha\cH_\rho(u_0)+\mathcal{J}(u_\bullet).
	\end{equation} With this notation, the result is as follows. \begin{theorem}[Large Deviations Principle]\label{thrm: LDP} Let $\rho\in C(\TTd, (0,\infty))$, let $\PP$ be a probability measure, and for all $N$ let $\eta^N_\bullet$ be the rescaled ZRP with initial data sampled according to $\Law_{\PP}(\eta^N_0)=\Pi^N_\rho$. Then, with speed $\frac{N^d}{\chi_N}$, the processes $\eta^N_\bullet$ are exponentially tight and satisfy a large deviation principle in $\mathbb{D}$ with rate function $\mathcal{I}_\rho$; that is, for all $\cDD$-closed sets $\mathcal{E}$, \begin{equation}
	\label{eq: UB statement} \limsup_N \frac{\chi_N}{N^d}\log \PP\left(\eta^N_\bullet \in \mathcal{E}\right)\le -\inf\left\{\cI_\rho(u_\bullet): u_\bullet \in \mathcal{E}\right\}
\end{equation} and all $\cDD$-open sets $\cU$, \begin{equation}
	\label{eq: LB statement} \liminf_N \frac{\chi_N}{N^d}\log \PP\left(\eta^N_\bullet \in \cU\right)\ge -\inf\left\{\cI_\rho(u_\bullet): u_\bullet \in \cU\right\}.
\end{equation} Moreover, the function $\mathcal{I}_\rho$ is lower-semicontinuous with respect to $\mathbb{D}$ and has compact sublevel sets. \end{theorem} We emphasise that we obtain a \emph{full} large deviations principle in which the infimum in the lower bound runs over the whole set $\cU$, rather than being restricted to a class of `good' paths $\mathcal{Q}$ as in the works \cite{quastel1999large,landim1995large,landim1997hydrodynamic}. This issue will be discussed in detail in the literature review. \\\paragraph{\textbf{A regularity-based approach to the replacement lemma}} We now turn to the main technical tool which will be needed in the derivation of Theorems \ref{th: hydrodynamic limit} - \ref{thrm: LDP}. As is common in the hydrodynamic limits and large deviations of lattice gasses, we will need to take a limit of terms similar to $\int_0^t\langle (\Delta \varphi), (\eta^N_s)^\alpha\rangle ds$ which are nonlinear in the particle configuration $\eta^N_s$, and where convergence of $\eta^N_\bullet$ in the topology of $\mathbb{D}$ is not sufficient; this is usually called a \emph{replacement lemma} or \emph{superexponential estimate} (e.g. \cite[Theorem 3.1]{kipnis1998scaling}, Theorem 2 of Benois, Kipnis and Landim \cite{benois1995large}, or Theorem 2 in Kipnis, Olla and Varadhan \cite{kipnis1989hydrodynamics}). 

As already remarked, the usual approach to the superexponential estimate in the previously cited works cannot be applied for the rescaled ZRP, and the difficulties correspond exactly to the features of degenerate and unbounded diffusivity that we aim to study. Firstly, the presence of degenerate diffusivity leads to a vanishing minimum rate; if one attempts to carry out the usual argument leading to the one-block estimate, the control provided by the Dirichlet form degenerates as $N\to \infty$, and the requisite compactness is lost. {At the same time, the superlinear growth of the local jump rate $(\eta^N(x))^\alpha$ means that the functionals $\langle \Delta \varphi, (\eta^N_t)^\alpha\rangle$ could diverge in the limit $N\to \infty$ for a nontrivial set of times $t$. Since the local jump rate $(\eta^N(x))^\alpha$ corresponds to the mobility $u^\alpha$ in the limit, this is the possibility of the explosion of the total mobility described above \eqref{eq: PME}. In order to establish the superexponential estimate, we must establish a sufficiently strong integrability estimate to show that such divergences occur at most for a small set of times, and in fact the contributions to the error from any small space-time region are small, with superexponentially high probability. }
\\ The proof developed in this work is inspired by techniques in the theory of stochastic partial differential equations \cite{dirr2020conservative,fehrman2019large,fehrman2021well} and the Aubin-Lions-Simon lemma \cite{aubin1963theoreme,lions1969quelques,simon1986compact}. For the case of a lattice-based model, as we deal with here, the situation is complicated by the fact that the regularity is measured by different functionals $\mathcal{F}_{\alpha,N}$ for each $N$, whose $\Gamma$-convergence in the limit is not immediate. This issue further cannot be remedied by a convenient choice of $\mathcal{F}_{\alpha,N}$, since these are dictated by the fundamental regularity and cannot be freely chosen.\\ \\ Before stating the theorem, let us first introduce the entropy dissipation which is central to the argument.  It is a classical computation that, along solutions to (\ref{eq: PME}) $\partial_t \cH(u_t)=-\cD_\alpha(u_t) \le 0$ for the entropy dissipation given by \begin{equation}\label{eq: entropy dissipation}
			\cD_\alpha(u)=\frac2\alpha \int_{\TTd} |\nabla (u^{\alpha/2})|^2(x) dx. \end{equation} This estimate is also central to the theory of (\ref{eq: Sk}) in \cite{fehrman2019large}, where for some $c=c(d,\alpha)$ \begin{equation}\label{eq: basic entropy estimate} \begin{split} & \sup_{t\le \tf}\cH(u_t) + \int_0^\tf \cD_\alpha(u_t) dt  \le  \cH(u_0) + c\|g\|^2_{L^2_{t,x}}.\end{split} \end{equation}  In particular, solutions to (\ref{eq: Sk}) for any choice of $g$ always take values in the space $\cR \subset L^\alpha([0,\tf]\times \TTd)$ of paths which are continuous in $d$ with respect to time and where $\int_0^\tf \cD_\alpha(u_s)ds<\infty$, see (\ref{eq: energy class}). We will use a similar path-by-path estimate, with lattice discretisations $\cD_{\alpha,N}$ of (\ref{eq: entropy dissipation}), to exclude paths outside $\cR$ from the large deviations, and to obtain the convergence in the norm of $L^\alpha([0,\tf]\times\TTd)$. The first property will be a step towards obtaining a true large deviation principle, as advertised below Theorem \ref{thrm: LDP}, and the latter property plays the same role as the superexponential estimate in taking limits of the nonlinear terms.   \begin{theorem}[Large Deviations Weak-to-Strong Principle]\label{thrm: WtS}
	Suppose the scaling relation (\ref{eq: scaling hypothesis}) holds, and let $\mathbb{P}$ be a probability measure under which $\eta^N_\bullet$ is the rescaled ZRP with initial data satisfying  \begin{equation} \label{eq: entropy finite hyp}
	\limsup_N \frac{\chi_N}{N^d}\log\EE\left[\exp\left(\frac{N^d}{\chi_N}\gamma \cH(\eta^N_0)\right)\right]<\infty
\end{equation} for all $\gamma<\alpha$. Then the following hold. \begin{enumerate}[label=\roman*).]
		\item (Open Sets Version) For any $u_\bullet \in \mathbb{D}\cap \cR \subset L^\alpha([0,\tf]\times\TTd)$, any open neighbourhood $\cV$ of $u_\bullet$ in the (norm) topology of $L^\alpha([0,\tf]\times\TTd)$ and any $z<\infty$, there exists an open neighbourhood $\cU \ni u_\bullet$ in the topology of $\mathbb{D}$ such that \begin{equation}\label{eq: WtS good case}
		\limsup_N \frac{\chi_N}{N^d}\log \PP\left(\eta^N_\bullet \in \cU\setminus \cV\right) \le -z.
	\end{equation} If instead $u_\bullet \not \in \cR$, then for all $z<\infty$, there exists a $\cDD$-open set $\cU\ni u_\bullet$ such that \begin{equation} \label{eq: QY2}
		\limsup_N \frac{\chi_N}{N^d}\log \PP\left(\eta^N_\bullet \in \cU\right) \le -z.
	\end{equation} \item (Almost-Sure-Convergence Version) For $\PP$ as above, let $\mathbb{Q}$ be a probability measure with the bound \begin{equation}\label{eq: entropy bound hypothesis}
		\limsup_N \frac{\chi_N}{N^d}H\left(\Law_{\QQ}[\eta^N_\bullet]\right|\left.\Law_{\PP}[\eta^N_\bullet]\right)<\infty
	\end{equation} where $H(\cdot|\cdot)$ denotes the relative entropy of probability measures on $\mathbb{D}$. Suppose further that, $\QQ$-almost surely, on a subsequence $N_k\to \infty$, $\eta^{N_k}_\bullet$ converges to a random variable $\eta_\bullet$ in the topology of $\mathbb{D}$. Then there exists a further subsequence $N'_k$ such that $\eta^{N'_k}_\bullet \to \eta_\bullet$ almost surely in the topology of $L^\alpha([0,\tf]\times\TTd)$. 
	\end{enumerate}
 \end{theorem}
 
\paragraph{\textbf{From Large Deviations to Gradient Flows}} {
Having achieved the probabilistic parts of the strategy, we turn to the derivation of a gradient flow structure, hence resolving the problem described below \eqref{eq: PME}. The final theorem extends the correspondence between the dynamical large deviation cost and the EDI \cite{fathi2016gradient,dirr2016entropic,adams2011large} corresponding to a suitable formal Riemannian structure for \eqref{eq: PME}; see also the literature review below for further discussion.} {In addition to extending this identity to the case of degenerate and unbounded diffusion, we also prove, for the first time, that the two functionals appearing in Theorem \ref{thrm: gradient flow} below are equal whenever either is finite. Notably, this is a novel result even in the setting of non-degenerate and bounded diffusion.} 

We consider the space of absolutely continuous measures $\mathcal{M}_{ac, \lambda}(\mathbb{T}^d)$ on $\mathbb{T}^d$ with a prescribed total mass $\lambda>0$, and define the \emph{action} of a curve $u_\bullet \in \cDD$ to be \begin{equation}\label{eq: action 2}
		\cA(u_\bullet)=\frac12\inf\left\{\left\|\theta\right\|_{L^2_{t,x}}^2\right\}. 
	\end{equation} In this definition, the infimum runs over all $\theta \in {L^2_{t,x}}$ such that $u_\bullet$ solves the continuity equation \begin{equation} \label{eq: CE}
		\partial_t u_t+\nabla\cdot(\frac{1}{2}u^{\alpha/2}_t\theta_t)=0
	\end{equation} setting $\mathcal{A}(u_\bullet)=\infty$ if no such $\theta$ exists. With this defined, the result is as follows.
 \begin{theorem}\label{thrm: gradient flow}
	Let $u_\bullet \in \mathbb{D}$ with $\cH(u_0)<\infty$ and fix a constant $\rho>0$. Then we have the identity
	\begin{equation}\label{eq: conclusion of gf}
		\mathcal{J}(u_\bullet)=\frac12\left(\alpha\cH_\rho(u_\tf)-\alpha\cH_\rho(u_0)+\frac\alpha2\int_0^\tf \cD_\alpha(u_s)ds +\frac12 \cA(u_\bullet)\right)
	\end{equation} allowing both sides to be infinite. In particular, the functional on the right-hand side is nonnegative, and vanishes if and only if $u_\bullet$ is a solution to (\ref{eq: PME}). Moreover, if $u_\bullet$ is such that $\mathcal{J}(u_\bullet)<\infty$, then for almost all $0\le t\le \tf$,  $u_t^{\alpha/2}\in H^1(\TTd)$ and \begin{equation} \label{eq: tangent space useful def} \nabla u_t^{\alpha/2}\in {T}_{u_t}\mathcal{M}_{\text{ac}, \lambda}(\TTd):=\overline{\left\{u_t^{\alpha/2}\nabla \varphi: \varphi\in C^2(\TTd)\right\}}^{L^2(\TTd)}.\end{equation} We can also choose $g$ such that the skeleton equation (\ref{eq: Sk}) holds, attaining the minimum (\ref{eq: dynamic cost}), and such that $g_t\in T_{u_t}\cM_{\text{ac},\lambda}(\TTd)$ for almost all $t\le \tf$. Finally, in the special case where $u_\bullet$ is a solution to the PME (\ref{eq: PME}) and $\rho>0$, it holds for all $0\le t\le \tf$ that \begin{equation}
		\label{eq: entropy dissipation equality}\cH_\rho(u_t) +\int_0^t \cD_\alpha(u_s)ds=\cH_\rho(u_0).
	\end{equation}\end{theorem} Although the statement (\ref{eq: conclusion of gf}) involves only objects defined at the level of the limiting trajectories, the proof will be probabilistic and exploit Theorem \ref{thrm: LDP}.  
	\subsection{Plan of the Paper} The paper is structured as follows. Section \ref{sec: Lit} discusses related works in the literature and contextualises the results of the paper. Section \ref{sec: prelim} collects some preliminaries: We give formal definitions of all objects used in the paper and give an overview of the tools which we will use, but which are not novel in the present work, including summarising some results of \cite{fehrman2019large}  regarding (\ref{eq: Sk}) in Section \ref{sec: skg}. Since all results in this section are of very little novelty, proofs are ommited in favour of references to similar proofs in the literature. Section \ref{sec: equilibrium} gives some estimates and computations for static large deviations in slowly-varying local equilibria{, which differ from the standard large deviations in equilibrium due to the additional rescaling of particle size by $\chi_N$.} \medskip \\ The main technical novelties of the paper, culminating in Theorem \ref{thrm: WtS}, are in Sections \ref{sec: DTC} - \ref{sec: WTS}. The key ingredients are separated into the functional-analytic aspects in Section \ref{sec: DTC} and the key large deviations estimate in Section \ref{sec: a priori}; the proof of Theorem \ref{thrm: WtS} is assembled in Section \ref{sec: WTS}. \medskip \\ 
 \ifx\jrnl\undefined
undefed
\else
  \if\jrnl1
With these ingredients at hand, Theorems \ref{th: hydrodynamic limit} - \ref{thrm: LDP} follow: the proofs will closely follow those of the works \cite{kipnis1989hydrodynamics,kipnis1998scaling}, with the difference that we use Theorem \ref{thrm: WtS} to play the role of a superexponential estimate.  Theorem \ref{th: hydrodynamic limit} is proven in Section \ref{sec: hydrodynamic}. The proof of the large deviations upper bound (\ref{eq: UB statement}) is given in Section \ref{sec: UB}. The proof of the lower bound is very similar to the standard argument, although a priori complicated by the double-rescaling and the presence of unbounded diffusivity; in Section \ref{sec: LB} we give the details of these technical points, using Theorem \ref{thrm: WtS} for technical support. Finally, Theorem \ref{thrm: gradient flow} is deduced in Section \ref{sec: gradient flow}. 
  \else
With these ingredients at hand, Theorems \ref{th: hydrodynamic limit} - \ref{thrm: LDP} follow: the proofs will closely follow those of the works \cite{kipnis1989hydrodynamics,kipnis1998scaling}, with the difference that we use Theorem \ref{thrm: WtS} to play the role of a superexponential estimate.  Theorem \ref{th: hydrodynamic limit} is proven in Section \ref{sec: hydrodynamic}. The proof of the large deviations upper bound (\ref{eq: UB statement}) is given in Section \ref{sec: UB}, and the lower bound is proven in Section \ref{sec: LB}. Finally, Theorem \ref{thrm: gradient flow} is deduced in Section \ref{sec: gradient flow}. 
  \fi
\fi 
	
	\section{Discussion \& Literature Review}\label{sec: Lit}
\paragraph{\textbf{1. PME from Particle Systems}}  Suzuki and Ushiyama \cite{suzuki1993hydrodynamic} introduce a model, called `stick-breaking' in later works, where each site $x\in \TND$ is assigned  a `stick-length' in $[0,\infty)$, and show the convergence in probability to the PME with homogeneity $\alpha=2$. The same convergence was also proven under slightly different assumptions in \cite{ekhaus1995stochastic}, who also introduced a version of the model where the stick lengths take discrete values, and the model was generalised to relax the restriction to $\alpha=2$ in \cite{feng1997microscopic}. More recently, the work \cite{goncalves2009hydrodynamic} and Blondel, Canc{\`e}s, Sasada and Simon\footnote{At the time of writing, this preprint contains a disclaimer that the work is incomplete.} \cite{blondel2018convergence} considered the hydrodynamic limit of an exclusion-type particle system with kinetic constraints, which can be constructed to lead to (\ref{eq: PME}) for any choice of the exponent $\alpha\in \mathbb{N}$. In the special case $\alpha=1$, this model is the simple exclusion process, where the limit was proven in \cite{kipnis1989hydrodynamics}. The restriction to integer $\alpha$ for similar processes was lifted by \cite{goncalves2023exclusion}, who derive both fast and slow diffusion $\partial_t u= \Delta u^\alpha, \alpha\in (0,2]$ from an exclusion-type particle system. Let us also mention the works by Seo \cite{seo2017large} and Dembo, Shkolnikov, Varadhan and Zeitouni \cite{dembo2016large}, who proved matching upper and lower large deviation bounds for interacting particle systems driven by Brownian motions. The work \cite{seo2017large} proves a large deviation principle around a system of nonlinear, nondegenerate parabolic PDEs for the interaction of `colours' of particles, and \cite{dembo2016large} characterises the large deviations from a non-degenerate PME in dimension $d=1$.   \bigskip \\ The works \cite{ekhaus1995stochastic,feng1997microscopic} are based on an entropy method first introduced by Varadhan \cite{varadhan1991scaling}, while the works \cite{goncalves2009hydrodynamic,blondel2018convergence} on the exclusion-type model use a relative entropy method due to Yau \cite{yau1991relative}. All of these works deal exclusively with the hydrodynamic limit equivalent to Theorem \ref{th: hydrodynamic limit}, and the method cannot reach large deviations. Indeed, as discussed above Theorem \ref{thrm: WtS}, the key step is to replace the nonlinearity $(\eta^N(x))^\alpha$ by $(\overline{\eta}^{N,r}(x))^\alpha$, which is the same nonlinearity applied to the spatial averages $\overline{\eta}^{N,r}$ on a macroscopic spatial scale $r$, with only a small error. In the case of large deviations, the error probability must be superexponentially small, whereas the methods in the cited papers can at most show a $L^1(\PP)$-convergence of these errors, and are unable to say anything at the large deviation level. 
\bigskip \\  The hypotheses (\ref{eq: entropy UI hyp} - \ref{eq: weak convergence hyp}) of Theorem \ref{th: hydrodynamic limit} allow significant freedom in choosing the initial data. For example, given $u_0\in C(\TTd,[0,\infty))$, we will see in Section \ref{sec: equilibrium} that both conditions hold with $\Law_{\PP}[\eta^N_0]=\Pi^N_{u_0}$ given by the slowly-varying local equilibrium (\ref{eq: eq}), which is a typical construction for hydrodynamic limits. On the other hand, we can also take $\eta^N_0$ to be deterministic, which is usually out of the reach of (relative) entropy methods.  We are also able to treat much more general initial data than previous works: the work \cite{goncalves2009hydrodynamic} required that the initial profile $u_0$ be bounded away from $0$ and $1$; in \cite{ekhaus1995stochastic,feng1997microscopic} some more regularity of the initial data $u_0\in C^{2+\e}(\TTd), \e>0$ is required. In contrast, we can treat (for example) deterministic $\eta^N_0$ converging to an irregular $u_0$ with support only on a subset of $\TTd$, provided only that $\cH(u_0)<\infty$.  Similarly to the works \cite{ekhaus1995stochastic,feng1997microscopic}, we will use some integrability estimates, i.e. estimates on $\|\eta^N_t\|_{L^\beta(\TND)}$-norms with $\beta>\alpha$, in a tightness argument. In the present context, we need estimates at the large-deviations level, which are rather more subtle than the estimates in expectation.\bigskip \\ Finally, we note that the large deviations are much more sensitive to the fine details of the particle model than the hydrodynamic limit, and different ideas would be needed to treat the large deviations of the other models mentioned above. For example, the equilibrium distribution of the `stick-breaking' model investigated by  
\cite{ekhaus1995stochastic,feng1997microscopic} leads to large deviations with the same speed, but where the finite-rate large deviations may be measures singular with respect to the Lebesgue measure. Correspondingly, it appears to be much harder to establish a large deviations principle for this model. \bigskip \\  \paragraph{\textbf{2. Relation to the Zero-Range Process}   The limits of the zero-range process for general jump rates, assuming bounded diffusivity, as well as fluctuation theorems about these limits, have been widely studied. The hydrodynamic limit, see \cite{kipnis1998scaling}, is a nonlinear parabolic equation $\partial_t u=\Delta \Phi(u)$, with globally bounded and locally nondegenerate diffusivity $\sup_\rho \Phi'(\rho)<\infty, \inf_{\rho\le M}\Phi'(\rho)>0$. Let us cite the works by Menegaki \cite{menegaki2021quantitative} and Menegaki and Mouhot \cite{menegaki2022consistence} for a more recent hydrodynamic limit with an explicit rate of convergence. Equilibrium large deviations for a variant of the zero-range process have been studied by Bernardin, Gon{\c{c}}alves, Jim\'enez-Oviedo and Scotta in \cite{bernardin2022non}. In infinite volume, which we do not consider in the present work, the large deviations and hydrodynamic limit have been studied by Landim and Yau \cite{landim1995large} and Landim and Mourragui \cite{landim1997hydrodynamic}. Quastel, Rezakhanlou and Varadhan \cite{quastel1999large} found the analagous rate function for the simple symmetric exclusion process.  \bigskip \\As already mentioned below Theorem \ref{thrm: LDP}, an important difference to previous works is that we obtain a full large deviation principle, with matching upper and lower bounds and the infimum of (\ref{eq: LB statement}) running over the whole open set $\cU$, rather than being restricted to $\cU \cap \cQ$ for a class of `good' paths $\cQ$. This is achieved through two steps, namely the exclusion of paths outside $\cR$ by Theorem \ref{thrm: WtS}i), and the characterisation of $\cI_\rho|_\cR$ as the lower semicontinuous envelope of a restriction $\cI_\rho|_\cX$(Proposition \ref{prop: lsc envelope}), which we recall from \cite{fehrman2019large}. This property is known for relatively few of the models whose large deviations have been studied. In the case \cite{kipnis1989hydrodynamics}, the rate function is globally convex and such approximations can be obtained by convolution, and in works on reaction-diffusion systems \cite{jona-lasino1993large,bodineau2012,landim2018large,farfan2019}, the rate is a pertubation of a convex functional by a lower-order term. In \cite{quastel1999large,bertini2009non}, the required property is proven for exclusion-type processes, using {\em a priori} $L^\infty_{t,x}$ bounds and the boundedness and nondegeneracy of the diffusion.} In other works, for example Quastel and Yau \cite{quastel1998lattice}, the rate function is defined as a lower semicontinuous envelope, so that this property holds by definition, but lack a corresponding explicit characterisation. A counterexample in a different setting, where the na\"ive rate function does not coincide with the lower semicontinuous envelope of its restriction, has been found by the second author \cite{heydecker2021large}. \bigskip \paragraph{\textbf{3. Macroscopic Fluctuation Theory and Large Deviations of Conservative SPDE}} The same rate function $\mathcal{I}_\rho$ studied in this work appears naturally  as the informal large deviation rate function for the stochastic partial differential equation (SPDE) with conservative noise \begin{equation}
	\label{eq: SPDE} \partial_t u^\e_t = \frac12 \Delta \left((u^\epsilon_t)^\alpha\right)-\sqrt{\e} \nabla\cdot\left((u^\e_t)^{\alpha/2}\xi^K\right)
\end{equation} where $\xi^K$ is the convolution of a space-time white noise $\xi$ with a spatial mollifier on a scale $K^{-1}$, see the works of Dirr, Stamatakis and Zimmer \cite{dirr2016entropic} and Giacomin, Lebowitz and Presutti \cite{giacomin1999deterministic}. The well-posedness of SPDEs with a conservative noise term, including (\ref{eq: SPDE}) as a special case, was established by  \cite{fehrman2021well} as soon as $K<\infty$, so the noise has nontrivial spatial correlation. In general, the regularisation of the noise is necessary, since the case $\alpha=1$ is the Dean-Kawasaki equation, corresponding to non-interacting particles \cite{dean1996langevin,donev2014reversible,donev2014dynamic}, which was shown by Konarovskyi, Lehmann and von Renesse \cite{konarovskyi2019dean} to admit only trivial solutions if $K=\infty$. Large deviations for singular SPDEs have been studied in the works \cite{cerrai2011approximation,faris1982large,jona-lasinio1990large,hairer2015large}; in general, renormalisation constants may enter the rate function, as shown by Hairer and Weber \cite{hairer2015large}. In the work \cite{fehrman2019large}, a large deviation principle \cite[Theorem 6.8]{fehrman2019large} with rate function $\mathcal{J}$ for fixed initial data is proven under a scaling relation between $\e$ and $K$ which prevents renormalisation constants from appearing in the rate function. The cited theorem proves the large deviation principle for a much broader class of SPDEs, replacing $u^\alpha, u^{\alpha/2}$ by $\Phi(u), \Phi^{1/2}(u)$ for a class of $\Phi$ which includes the porous-medium nonlinearity considered here. Dirr, Fehrman and the first author \cite{dirr2020conservative} proved the analagous large deviation principle, now with $u^\alpha$ replaced by $\Phi(u)=u(1-u)$, for the symmetric exclusion process. This current work therefore extends the ideas of \cite{dirr2020conservative,fehrman2019large} to a particle setting, and Theorems \ref{th: hydrodynamic limit}-\ref{thrm: LDP}, together with \cite[Theorem 6.8]{fehrman2019large}, show that the SPDE (\ref{eq: SPDE}) has the same small-noise limit and large deviations as the ZRP we consider. We may therefore consider (\ref{eq: SPDE}) as a phenomenological model for the ZRP.  \bigskip \\ The large deviation rate functional is also closely related to macroscopic fluctuation theory (MFT) \cite{bertini2009non,bertini2015macroscopic,derrida2007non}, which postulates an Ansatz for the large-deviations rate function, for example in Bertini \cite[Equation 1.3]{bertini2015macroscopic}. We also refer to \cite{hohenberg1977theory,spohn2012large} for a rigorous justification. In this context, the time-reversal argument in Section \ref{sec: gradient flow} is due to Onsager \cite{onsager1931reciprocali,onsager1931reciprocalii}. \bigskip \\  \paragraph{\textbf{4. The scaling hypothesis (\ref{eq: scaling hypothesis}).}} We conclude with a discussion of the scaling hypothesis (\ref{eq: scaling hypothesis}). This hypothesis will play a key role in this paper, as it allows the path-by-path estimates on $\int_0^\tf \cD_{\alpha,N}(\eta^N_s)ds$ for a discrete entropy dissipation $\cD_{\alpha,N}$ defined in (\ref{eq: DNA}).  \bigskip \\ We observe that we could also take the limits of shrinking particle size $\chi\to 0$ and growing lattice $N\to \infty$ with decoupled parameters, and in either order, to obtain meaningful limits, which we summarise in Figure \ref{fig: all limits} below. The six relevant limits are described below: This paper is concerned with the curved arrow (e), and we leave the problems of making other limits rigorous to future work. 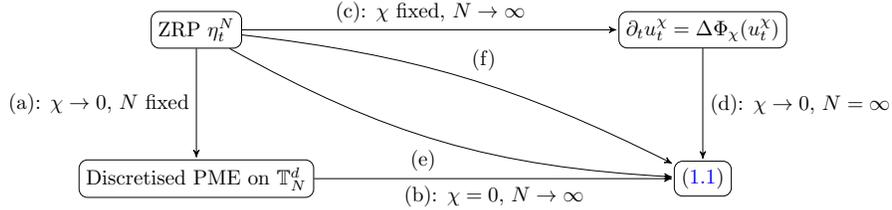
\begin{figure}[htb]
\centering
\resizebox{12cm}{!}{ \begin{tikzpicture}[->,>=stealth',shorten >=1pt,node distance=2.5cm,auto,main node/.style={rectangle,rounded corners,draw,align=center}]

\node[main node] (1) {ZRP $\eta^N_t$};
\node[main node] (2) [below of=1] {Discretised PME on $\TND$};
\node[main node] (3) [right of=1, xshift = 6cm] {$\partial_t u^\chi_t = \Delta \Phi_\chi(u^\chi_t)$};
\node[main node] (4) [below of=3] {PME (\ref{eq: PME})};

\path
(1) edge node [swap] {(a): $\chi\to 0$, $N$ fixed} (2)

(2) edge node [swap] {(b): $\chi=0$, $N\to \infty$} (4)
(1) edge node [above] {(c): $\chi$ fixed, $N\to \infty$} (3)
(3) edge node [right] {(d): $\chi\to 0$, $N=\infty$} (4)

(1) edge [bend left=10] node  {(f)} (4)
(1) edge [bend right=13] node[swap]  {(e)} (4);

\end{tikzpicture}}  \caption{All possible double-limits. The horizontal arrows represent $N\to \infty$ with the hydrodynamic rescaling of space and time, and the vertical arrows represent $\chi\to 0$ with the value-time rescaling. The two curved arrows are the two double limits, depending on whether $\chi_N\to 0$ sufficiently quickly (e) or sufficiently slowly (f) with $N$. } \label{fig: all limits}\end{figure} 
 \begin{enumerate}[label=(\alph*)]
	\item {\bf $\chi\to 0$, $N$ fixed} gives a discrete version of the PME (\ref{eq: PME}). Fluctuations are described by the Freidlin-Wentzell theory in a $N$-dependent space. \item {\bf Subsequently taking $ N\to \infty$}  should recover the continuum PME (\ref{eq: PME}) from the discretised version, and $\Gamma$-convergence of the rate functionals. \item {\bf $\chi$ fixed, $N\to \infty$} is the usuals scaling limit of the zero-range process, in the spirit of \cite{kipnis1998scaling}. In this case, we expect to find the superexponential estimate through the same mecahnism as \cite{kipnis1989hydrodynamics} (rapid equilibration on macroscopic boxes), in contrast to the current work.  \item  {\bf Subsequently taking $\chi\to 0$} will recover the nonlinearity $\Phi(u)=u^\alpha$ from the nonlinearities $\Phi_\chi(u)$.  As in (b), a natural approach would be to seek the  $\Gamma$-convergence of the associated rate functions.  \item  {\bf Double-limit $N\to 0, \chi_N \to 0$ sufficiently fast}  is the content of the current paper, with the scaling relation imposed by (\ref{eq: scaling hypothesis}). In this case we find the superexponential estimate via regularity estimates, in contrast to (c). \item {\bf Double-limit $N\to 0, \chi_N \to 0$ sufficiently slowly} is the opposite type of relation from the ideas of this paper. As in (c) above, different ideas will be needed for the superexponential estimate. 

\end{enumerate}

\paragraph{\textbf{5. PME as a gradient flow}} {As discussed in the introduction, the problem of identifying a thermodynamic gradient flow structure for \eqref{eq: PME} has stood open since the works of Br\'ezis and Otto \cite{brezis1971monotonicity,otto2001geometry}. Theorem \ref{thrm: gradient flow} is one resolution of this problem by extending, to the case of degenerate and unbounded diffusivity, the connection between large deviations and gradient flow \cite{adams2011large,mielke2014relation,fathi2016gradient}, and more specifically the connection via the EDI \cite{adams2013large,dirr2016entropic}.\\
We first start by recalling the state of the art in the nondegenerate and boundedly diffusive setting. In \cite{dirr2016entropic}, for nondegenerate and sublinear nonlinearity $\Phi(u)$ replacing $u^\alpha$ in \eqref{eq: PME}, it is shown how the identity analogous to \eqref{eq: conclusion of gf} arises from large deviations of the ZRP with Lipschitz jump rates, and how this identity corresponds to the EDI. The corresponding gradient flow is in a Riemannian structure determined by the underlying ZRP, which is defined by modifying the Otto calculus \cite{otto2001geometry} by replacing the weight $u$ by $\Phi(u)$ in the Riemannian tensor. This Riemannian structure leads to a global metric which may be defined in the same way as the Benamou-Brenier formulation for the Wasserstein distance: One defines the action $\mathcal{A}_\Phi$ of a curve in the same way as \eqref{eq: action 2}, with $\Phi(u)$ in place of $u^\alpha$, and defines the distance between two profiles $u_0, u_1$ as the infimal action $\mathcal{A}_\Phi(u_\bullet)$ of all curves starting at $u_0$ and ending at $u_1$. The modified Otto calculus also permits to identify the action as $\int_0^\tf g_{u_t}(\dot{u}_t, \dot{u}_t)dt$, for a Riemannian metric $g_u$ for which the tangent space may be identified in the same way as (\ref{eq: tangent space useful def}). For these cases, where a global metric structure may be defined, the identity corresponding to \eqref{eq: conclusion of gf} is the entropy dissipation identity, which gives an equivalent formulation of the gradient flow property \cite{ambrosio2005gradient,sandier2004gamma,serfaty2011gamma, fathi2016gradient}. In addition, in these cases further connections between the large deviations rate function and gradient flow structure are possible, for instance the convergence of a JKO scheme \cite{jordan1998variational, adams2011large,mielke2014relation}.} \\\\ {Let us now contrast this {nondegenerate and boundedly diffusive case} with the {case of degenerate and unbounded diffusivity of interest here.}   Replacing the nonlinearity $\Phi(u)$ in the arguments of Dirr et al. \cite{dirr2016entropic} by the porous medium nonlinearity $u^\alpha, \alpha>1$, introduces both degenerate diffusivity at $u=0$, and unbounded diffusivity as $u\to \infty$. As a result, the formal Riemann structure no longer gives a well-defined global geometry: The distance between two profiles may become zero, as a result of unbounded diffusivity, or infinite, as a result of degeneracy.  \ifx\jrnl\undefined
\else
  \if\jrnl1
  \else
{For completeness, the formal Riemannian metric and gradient flow structure are recalled in Appendix \ref{sec: formal GF}.}
  \fi
\fi As a result, characterisations of the gradient flow involving the metric, such as the JKO scheme, are no longer meaningful. \\ The approach in Theorem \ref{thrm: gradient flow} therefore consists of rigorously validating the EDI via the connection to the rate function, extending the connection \cite{dirr2016entropic,adams2011large,adams2013large,mielke2014relation}, and in this way we resolve the problem of selecting a gradient flow corresponding to a microscopic model for the PME. \\ 
The second part of Theorem \ref{thrm: gradient flow} shows that the gradient of the entropy in the formal Otto calculus is indeed an element of the tangent space according to (\ref{eq: tangent space useful def}). The final part of the theorem makes the formal computation above (\ref{eq: entropy dissipation}) precise for solutions to the PME (\ref{eq: PME}) no weaker than the solution concept of \cite{fehrman2019large}.\\
Notably, Theorem \ref{thrm: gradient flow} even extends what is known in the setting of nondegenerate and bounded diffusivity, since the formal calculations of \cite{adams2011large,dirr2016entropic} are implicitly restricted to fluctuations where each $u_t$ is regular enough to satisfy a `chain rule for the entropy', see the display above \cite[Equation (2.3)]{dirr2016entropic} and Erbar \cite[Proposition 4.1]{erbar2016gradient} for a similar calculation in a different context. 
The argument in Section \ref{sec: gradient flow} circumvents these difficulties by using the large deviations principle and the machinery of the skeleton equation \eqref{eq: Sk} from \cite{fehrman2019large}}, thereby removing any additional assumptions on the path $u$ aside from the finiteness of the objects in \eqref{eq: conclusion of gf}.  \\\\   Gradient flow structures for the PME (\ref{eq: PME}) are well-known: Br\'ezis \cite{brezis1971monotonicity} formulated it as the gradient flow of the functional $\cE_\alpha(u)=\frac{1}{\alpha+1}\int u^{\alpha+1}$ with respect to a flat $H^{-1}$-metric, and Otto \cite{otto2001geometry} showed that is the gradient flow of $\cE'_\alpha(u)=\frac1{\alpha-1}\int u^\alpha$ with respect to a geometry inducing the Wasserstein$_2$ distance $\cW_2$. As already remarked, the microscopic model studied here is not the only model which leads to the PME (\ref{eq: PME}) as a macroscopic limit, and other choices of the underlying particle system may lead to different large deviation rate functions, and hence other gradient flow formulations. In particular, the energy functional $\alpha \cH$ and the geometry defined by $\cA$ should be thought of as intrinsic {\em to the class of particle systems whose large deviations are given by Theorem \ref{thrm: LDP}}, and not to the PME (\ref{eq: PME}). 

 \ifx\jrnl\undefined
undefed
\else
  \if\jrnl1
\input{main_text_short.tex}
  \else
\section{Preliminaries}\label{sec: prelim}
We will now introduce some non-novel preliminary concepts before addressing the main proofs of the Theorems. Subsection \ref{sec: formal} will fix definitions and notation of constant use, including some deferred from the statements in the introduction, and a formal definition of the ZRP $\eta^N$ is given. In subsection \ref{sec: eq def} we give an explicit form for the equilibrium, respectively slowly-varying local equilibrium, distributions $\Pi^N_\rho$ and introduce some associated notation. Subsection \ref{sec: computation of generators} recalls a basic computation for the linear generator $\cL_N$ and exponentially transformed generator $\cG_N$ applied to linear functionals $F^\varphi_N$ of the particle configuration. Subsection \ref{sec: skg} recalls the theory of the skeleton equation (\ref{eq: Sk}) from \cite{fehrman2019large} and recalls results which will be relevant for the present work. Finally, subsection \ref{sec: variational} will give an alternative formulation of the functional $\mathcal{J}$, which is well-known to specialists in large deviation theory, and some properties of the functional $\cA$. Since all results in this section are of limited novelty, the proofs are deferred to Appendix \ref{sec: misc proofs}.

\subsection{Formal Definitions.}\label{sec: formal} We now give some technical definitions, which allow us to give precise statements of the main results, as well as fixing some frequently used notation. \paragraph{\textbf{1. Lebesgue Spaces on the Discrete Lattice and Torus; Particle State Space.}} We will frequently work with Lebesgue spaces on the discrete lattice, which we define so as to be included in $L^p(\TTd)$. For each $N$, let $(c^N_x, x\in \TND)$ be a partition of $\TTd$ into translations of $[0, N^{-1})^d$, with centres $x\in \TND$, and let $\cB_N$ be the $\sigma$-algebra generated by $\{c^N_x, x\in \TND\}$. We then take $L^p(\TND):=L^p(\TTd, \cB_N)$ to be the $p$-integrable functions on the whole torus $\TTd$ which are measurable with this $\sigma$-algebra, and hence constant on each $c^N_x$. This is trivially isometrically isomorphic to the usual meaning of $L^p(\TND)$, and this viewpoint is helpful in that we can view $\eta^N \in L^p(\TND)$ as \emph{either} a function on the lattice (for example, using Sobolev inequalities on the lattice in Lemma \ref{prop: interpolation}), or as a function on the whole torus $\TTd$. We will indicate $\cdot_{\ge 0}$ the subset of a Lebesgue space consisting of nonnegative functions. \medskip \\ We now define the state space for the particle configurations as \begin{equation*}
	X_N:=\left\{\eta^N\in L^1_{\ge 0}(\TND): \eta^N(x)\in \chi_N\NN \text{ for all }x\in \TTd\right\}
\end{equation*} which corresponds to the informal description of the ZRP by setting $\eta^N=k\chi_N$ on $c^N_x$ to represent $k$ particles of height $\chi_N$ at the lattice site $x\in \TND$. The identification of $L^1(\TND)\subset L^1(\TTd)$ above has the immediate benefit that all particle state spaces are embedded into a common space, namely $L^1_{\ge 0}(\TTd)$. For $a\in [0,\infty)$, we write $X_{N,a}$ for the subset \begin{equation}\label{eq: XNA} X_{N,a}:= \left\{\eta^N\in X_N: \int_{\TTd} \eta^N dx \le a\right\}.\end{equation} 
 \paragraph{\textbf{2. Generator and Exponentially Transformed Generator.}} We now make precise the definition of the particle dynamics by specifying the generator, for $F_N:X_N\to \R$, by \begin{equation}
	\label{eq: linear generator} \cL_N F_N(\eta^N):=\frac{dN^2}{\chi_N}\sum_{x,y\in \TND} p^N(x,y)(\eta^N(x))^\alpha\left(F_N(\eta^{N,x,y})-F_N(\eta^N)\right)
\end{equation}
where $\eta^{N,x,y}$ is the particle configuration obtained by moving one particle from site $x$ to site $y$: \begin{equation}
	\label{eq: def change at jump}
\eta^{N,x,y}=\eta^N-\chi_N e^N_x+ \chi_N e^N_y \end{equation} where $e^N_x$ is the indicator function for the cube $c^N_x$.  It is immediate to check that the rates of $\cL_N$ are bounded on each space $X_{N,a}$, and these spaces are invariant because $\cL_N F_N \equiv 0$ for any functional of the form $F_N(\eta^N):=f_N(\langle 1, \eta^N\rangle)$. It follows that $\cL_N$ generates a Markovian semigroup, that for each initial measure the generator corresponds to a unique-in-law process $\eta^N_\bullet$, and that the process is globally defined. \bigskip \\ 
We will frequently work with exponential martingales constructed by the (nonlinear) exponentially transformed  generator, see the book of Feng and Kurtz \cite{feng2006large}. For a functional $F_N: X_N\to \RR$, the action of the exponentially transformed generator is defined by \begin{equation}\label{eq: nonlinear generator}
	\cG_N F_N(\eta^N):=\frac{\chi_N}{N^d}e^{-N^dF_N(\eta^N)/\chi_N} \left(\cL_N e^{N^d F_N/\chi_N}\right)(\eta^N)
\end{equation} which is defined such that \begin{equation}
	\label{eq: exponential entropy martingale} \exp\left(\frac{N^d}{\chi_N}\left(F_N(\eta^N_t)-\int_0^t \cG_N F_N(\eta^N_s)ds\right)\right)
\end{equation} is a nonnegative local martingale, and hence a supermartingale.

  \paragraph{\textbf{3. Skorokhod Space $\cDD$.}}We now specify the topology underlying underlying Theorems \ref{th: hydrodynamic limit} - \ref{thrm: LDP}. We let $d$ be a metric on $L^1_{\ge 0}(\TTd)$ inducing the weak-$\star$ topology, and write $\langle\cdot, \cdot\rangle$ for the duality pairing of $L^\infty(\TTd)$ and $L^1(\TTd)$. We consider a forever fixed time horizon $\tf<\infty$ and write $\mathbb{D}$ for the Skorokhod space  \begin{equation*}
	\mathbb{D}:=\mathbb{D}\left([0,\tf], (L^1_{\ge 0}(\TTd,[0,\infty)), d)\right)
\end{equation*} and let $\mathbb{D}_N\subset \mathbb{D}$ be the subset of processes taking values in $X_N$, so that the realisations $\eta^N_\bullet$ of the ZRP take values in $\mathbb{D}_N$. We will occasionally use the fact that the topology of $\cDD$ admits a metric; see the work of \cite{jakubowski1986skorokhod}. We will find that all the limiting fluctuations take values in the space $\mathcal{C}\subset\mathbb{D}$ for the paths $u_\bullet$ which are continuous in the weak-$\star$ topology.

\paragraph{\textbf{4. Entropy and Entropy Dissipation.}} For $u, \rho\in L^1_{\ge 0}(\TTd)$, we define the Boltzmann entropy \begin{equation}\label{eq: entropy def} \cH_\rho(u):=\int_{\TTd}\left(\frac{u(x)}{\rho(x)}\log \frac{u(x)}{\rho(x)} - \frac{u(x)}{\rho(x)}+1\right)dx \end{equation} which we allow to take the value $\infty$ if the integral fails to converge or if $u(x)dx\not \ll \rho(x)dx$. If the argument $\rho$ is omitted, it is taken to be the constant function $1$; note that for any $u_0\in L^1_{\ge 0}(\TTd)$, $\cH(u_0)<\infty$ if and only if $\cH_\rho(u_0)<\infty$ for all $\rho \in C(\TTd, (0,\infty))$. Recalling the definition of the entropy dissipation $\cD_\alpha$ in (\ref{eq: entropy dissipation}), we extend the definition to a functional on $L^1_{\ge 0}(\TTd)$, taking the value $\infty$ if $u^{\alpha/2}\not \in \dot{H}^1(\TTd)$. The energy class $\cR$ which arises from the entropy/entropy-dissipation estimate is \begin{equation}
			\label{eq: energy class}\begin{split} \cR:&=\left\{u_\bullet \in \mathcal{C}: u^{\alpha/2}\in L^2([0,\tf],\dot{H}^1(\TTd))\right\}\\ & =\left\{u_\bullet \in \mathcal{C}: \int_0^\tf \cD_\alpha(u_s)ds<\infty\right\} \end{split} \end{equation} where $\mathcal{C}$ is defined above and we have used the definition (\ref{eq: entropy dissipation}) of $\cD_\alpha$. Let us remark, for future reference, that we have the equality \begin{equation} \label{eq: energy class'}
		\begin{split} \cR & =\left\{u_\bullet \in \mathcal{C}: u^{\alpha/2}\in L^2([0,\tf],H^1(\TTd))\right\} \\ & =\left\{u_\bullet \in \mathcal{C}\cap L^\alpha([0,\tf]\times\TTd): \int_0^\tf \cD_\alpha(u_s)ds<\infty \right\} \end{split}	\end{equation} by using the fact that $\mathcal{C}\subset L^\infty([0,\tf],L^1_{\ge 0}(\TTd))$, the Sobolev embedding and an interpolation; see \cite[Lemma 5.2]{fehrman2019large} or Remark \ref{rmk: continnum nonstandard convolution} below. This will also be the class for the regularity which is natural for the skeleton equation (\ref{eq: Sk}) with $g\in L^2([0,\tf]\times\TTd,\RRd)$, see the discussion in \cite{fehrman2019large} and Definition \ref{def: solutions} below. Let us emphasise that the functionals $\mathcal{J}, \cA$ are defined on the whole space $\mathbb{D}$, understanding the equations (\ref{eq: Sk}, \ref{eq: CE}) in the sense of Definition \ref{def: solutions} below. In particular, since the solution concept for these equation is defined only in $\cR$, the definition (\ref{eq: dynamic cost}) implies that $\mathcal{J}(u_\bullet)=\infty$ whenever $u_\bullet \not \in \cR$.

\subsection{Equilibrium \& Slowly Varying Local Equilibrium Distributions} \label{sec: eq def} As is common in the theory of the zero-range process, a straightforward computation shows that the process $\eta^N_\bullet$ admit a one-parameter family of reversible invariant measures, which we write as \begin{equation*}\label{eq: eq}
	\Pi^N_\rho(\eta^N):=\prod_{x\in \TND} \pi^N_\rho(\eta^N(x)), \qquad  \rho>0, \eta^N\in X_N
\end{equation*} Here, $\pi^N_\rho$ is the single-site measure: \begin{equation*}
		\pi^N_\rho(\eta):=\frac{(\rho/\chi_N)^{\alpha(\eta/\chi_N)}}{((\eta/\chi_N)!)^\alpha Z_\alpha(\rho/\chi_N)}, \qquad \rho\ge 0, \hs \eta \in \{0, \chi_N, 2\chi_N,\dots\}  \end{equation*} and where $Z_\alpha$ is the normalisation factor\footnote{This is linked to the partition function $Y(\varphi)$ by $Y(\varphi)=Z_\alpha(\varphi^{1/\alpha})$. This change of variables is convenient for the calculations in Section \ref{sec: equilibrium}.} \begin{equation} \label{eq: single site'} Z_\alpha(\varphi):=\sum_{m\ge 0}\frac{\varphi^{\alpha m}}{(m!)^\alpha}, \qquad \varphi\ge 0.
	\end{equation} The fact that the distributions $\Pi^N_\rho$ are reversible is easily verified by checking the detailed balance conditions, and implies that the law of the ZRP started from $\Pi^N_\rho$ is invariant under a time-reversal, a fact which will be important in the derivation of Theorem \ref{thrm: gradient flow}. In order to encode inhomogeneous initial data $u_0\in C([0,\infty), \TTd)$ for the limiting equation (\ref{eq: PME}), we replace the constant $\rho$ by a function $\rho \in C(\TTd,[0,\infty))$ and consider the \emph{slowly-varying local equilibrium}, in which the parameter of the local equilibria varies only on the macroscopic spatial scale: \begin{equation*}
		\label{eq: slowly varying local equilibrium} {\Pi}^N_{\rho}(\eta^N):=\prod_{x\in \TTd} \pi^N_{\rho(x)}\left(\eta^N(x)\right)
	\end{equation*} see, for instance, \cite[Definition 2.1]{kipnis1998scaling}.
 \subsection{A Basic Computation for the Linear \& Exponential Generators} \label{sec: computation of generators} The following computation is largely standard for the zero-range process; see, for example \cite[equation 1.7] {kipnis1998scaling} for the first claim and \cite[equation (2.6)]{kipnis1998scaling} for a claim similar to the second. The results are adapted here to the present situation, where the superlinear growth of the jump rates lead to a different form of the error. \begin{lemma}\label{lemma: computation of generators}
	Fix a test function $\varphi\in C^2(\TTd)$ and let $F^\varphi_N$ be the evaluation \begin{equation} \label{eq: evaluation} F^\varphi_N(\eta^N):=\langle \varphi, \eta^N\rangle=\int_{\TTd} \eta^N(x)\varphi(x)dx. \end{equation} Then the linear and nonlinear generators satisfy, for $\eta^N\in X_N$, \begin{equation}
		\left|\cL_N F^\varphi_N(\eta^N)-\frac{1}{2}\int_{\TTd}(\eta^N(x))^\alpha \Delta \varphi(x)dx\right|\le \theta_N\left\|\eta^N\right\|_{L^\alpha(\TTd)}^\alpha
	\end{equation} and \begin{equation}
		\left|\cG_N F^\varphi_N(\eta^N)-\frac{1}{2}\int_{\TTd} (\eta^N(x))^\alpha\left(\Delta \varphi(x)+|\nabla \varphi(x)|^2\right)dx\right|\le \theta_N \|\eta^N\|_{L^\alpha(\TTd)}^\alpha
	\end{equation} for some sequence $\theta_N$, depending only on $\|\varphi\|_{C^2(\TTd)}$ and converging to $0$ as $N\to \infty$. 
\end{lemma} 

\subsection{The Skeleton and Fokker-Planck Equations} \label{sec: skg} 
				
				This section will focus on the properties of the skeleton equation (\ref{eq: Sk}) and the Fokker-Planck-type equation\begin{equation}\label{eq: FP} \partial_t u_t=\frac12\Delta(u_t^\alpha)-\nabla\cdot(u_t^\alpha \nabla h_t) \end{equation} with $h\in C^{1,3}([0,\tf]\times \TTd)$. The theory of the skeleton equation (\ref{eq: Sk}) was developed in \cite{fehrman2019large}, and is recalled here for convenience, with mild reformulations which are convenient for the applications in this paper. We will also need a new result on the well-posedness of (\ref{eq: FP}), which we obtain by adapting the corresponding proof from \cite{fehrman2019large}. 
				
				Let us begin with a brief discussion on the skeleton equation (\ref{eq: Sk}). This equation is of parabolic-hyperbolic type, and since we work in the full generality $g\in {L^2_{t,x}}:=L^2([0,\tf]\times\TTd,\RRd)$, the drift term is highly irregular. As is argued in \cite[Section 2]{fehrman2019large}, the equation (\ref{eq: Sk}) is critical for values $u\in L^1(\TTd)$ and supercritical for values $u\in L^r(\TTd)$ for all $r>1$ for scalings which preserve the ${L^2_{t,x}}$-norm of the control $g$; we should therefore not hope for an \emph{a priori} estimate better than (\ref{eq: basic entropy estimate}). This corresponds exactly to the fact that we can find large-deviations estimates for the entropy $\cH(\eta^N_t)$ in Section \ref{sec: a priori}, but not for any functional $\cH'$ which dominates it\footnote{We would say that $\cH'$ dominates $\cH$ if $\sup_{u_0}\frac{\cH'}{\cH}=\infty$. It follows from the calculations in Section \ref{sec: equilibrium} that, if $\cH'$ is also lower semicontinuous, we would have $\limsup_N \frac{\chi_N}{N^d}\log \EE_{\Pi^N_\rho}[\exp(N^d \cH'(\eta^N_0)/\chi_N)]=\infty$, and there is no hope of repeating the arguments leading to Lemma \ref{lemma: QY1}.}. Moreover, in studying large deviations, we are forced on probabilistic grounds to work with scalings of (\ref{eq: Sk}) on which the two terms are of the same order. Indeed, we will see in Section \ref{sec: gradient flow} that the microscopic reversibility at the level of the ZRP has the effect of changing $g$ to absorb an \emph{anti}-diffusive term $-\frac12\Delta u_t^{\alpha}$  into the hyperbolic term. It also follows that, for a given $t>0$, we should not expect $u_t$ obtained from the large deviations of $\eta^N_t$ to have better regularity than $u_0$.
				
			\begin{definition}\label{def: solutions} 
				We say that a function $u_\bullet\in \mathcal{R}$ is a \emph{weak solution} to the skeleton equation (\ref{eq: Sk}) starting at $u_0$ if, for all $t\in [0,\tf]$ and all $\varphi\in C^\infty(\TTd)$, \begin{equation}\label{eq: weak form}\begin{split}  
						\int_{\TTd} \varphi(x)(u_t(x)-u_0(x))dx & = -\int_0^t\int_{\TTd} u^{\alpha/2}_s(x)\nabla u^{\alpha/2}_s(x)\cdot \nabla \varphi(x) dxds.\\& \hspace{2cm} + \int_0^t\int_{\TTd} g_s(x)u^{\alpha/2}_s(x)\cdot \nabla \varphi(x) dxds. 
					\end{split} \end{equation}
			The hypothesis $u_\bullet \in \cR$ ensures that all terms are meaningful thanks to (\ref{eq: energy class'}).  We write the set of weak solutions $u_\bullet$ starting at a given $u_0$ as $\mathcal{S}_g(u_0)$, or $\mathcal{S}(u_0)$ for the case $g=0$, in which case the skeleton equation (\ref{eq: Sk}) is the PME (\ref{eq: PME}).    \end{definition} 
				This definition has the same content as the work \cite{fehrman2019large}, but has been repackaged for  convenience. Provided that $g\in {L^2_{t,x}}$ and $\cH(u_0)<\infty$, \cite{fehrman2019large} proves existence and the entropy estimate (\ref{eq: basic entropy estimate})  (\cite[Proposition 5.7]{fehrman2019large}) and uniqueness (\cite[Theorem 3.3, Theorem 4.3]{fehrman2019large}).	 The definition of solution given here is more stringent than the cited work, in that we impose $u_\bullet \in \mathcal{C}$, but this does not lead to differences in the theory, since every solution in the sense of \cite{fehrman2019large} can be modified on a $dt$-measure 0 set to lie in 	$\mathcal{C}$ \cite[Proposition 3.4]{fehrman2019large}.		
				The next result concerning the skeleton equation (\ref{eq: Sk}) we will need is as follows, given by \cite[Theorem 8.6]{fehrman2019large}. \begin{proposition}\label{prop: lsc envelope} Fix $\rho\in C(\TTd, (0,\infty))$ as in Theorem \ref{thrm: LDP}. Let $\mathcal{X}$ consist of all $u_\bullet \in \mathbb{D}$ such that $u_t$ is uniformly (pointwise) bounded and uniformly bounded away from $0$, such that $u_0\in C^\infty(\TTd)$, and solving the skeleton equation (\ref{eq: Sk})  for some $g$ of the form $$g_t(x)=u^{\alpha/2}_t(x)\nabla h_t(x), \qquad h\in C^{1,3}([0,\tf]\times\TTd). $$ Then $\cI_\rho$ is the lower semicontinuous envelope of its restriction to $\cX$ in the topology of $\cDD$, that is, $\cI_\rho$ is lower semicontinuous in this topology and for any $u_\bullet \in \mathbb{D}$, there exists a sequence $u^{(n)}_\bullet \in \mathcal{X}$ with $u^{(n)}_\bullet \to u_\bullet$ in the topology of $\mathbb{D}$ and \begin{equation}
					\mathcal{I}_\rho(u^{(n)}_\bullet)\to \mathcal{I}_\rho(u_\bullet). 
				\end{equation}					
				\end{proposition} The cited statement \cite[Theorem 8.6]{fehrman2019large} does not include the uniform positivity of the approximants $u^{(n)}_\bullet$, and proves only convergence of\footnote{In the notation of the current work.} the dynamic cost $\mathcal{J}(u^{(n)}_\bullet)$, rather than the whole rate function $\cI_\rho$. The strict positivity is already in the proof in \cite{fehrman2019large}, and examining the proof shows the convergence  $\cH_\rho(u^{(n)}_0)\to \cH_\rho(u_0)$ using that $\log \rho \in C(\TTd)$ thanks to the assumption $\rho \in C(\TTd, (0,\infty))$. Let us emphasise that this property is, in general, highly non trivial, and it is this result that allows us to find a true large deviation principle; see the discussion in Section \ref{sec: Lit}. 

\bigskip We will also need a uniqueness theory for equations of the form (\ref{eq: FP}), which appear when we perform a change of measures for the lower bound (c.f. Step 2 of Lemma \ref{lemma: local lower bound}). For $h\in C^{1,3}([0,\tf]\times\TTd)$, we define a weak solution to (\ref{eq: FP}) to be $u_\bullet \in \cR$ which satisfies the skeleton equation (\ref{eq: Sk}) in the sense of Definition \ref{def: solutions} with $g_t(x):=u^{\alpha/2}_t(x)\nabla h_t(x)$; the hypothesis $u_\bullet \in \cR$ implies that $g\in {L^2_{t,x}}$. In keeping with the notation above, we write $\cS^\mathrm{FP}_h(u_0)$ for the set of weak solutions with prescribed $h$ and $u_0$.  The result is as follows. \begin{proposition}\label{prop: uniqueness FP}[Well-Posedness of (\ref{eq: FP})] Fix $h\in C^{0,1}([0,\tf]\times\TTd)$ and $u_0\in L^1(\TTd)$ with finite entropy $\cH(u_0)<\infty$. Then there exists at most one solution to (\ref{eq: FP}) with this initial data, so $\cS^\mathrm{FP}_h(u_0)$ is either a singleton or empty.\end{proposition} Let us note that we make no claim about existence, although this can be established from step 2 of Lemma \ref{lemma: local lower bound}. As soon as $h\in C^{1,3}_{t,x}$ is fixed, the equation may be treated as a parabolic perturbation of the PME \eqref{eq: PME}, and one could prove integrability, and then uniqueness in a way similar to \cite[Theorem 5.3]{vazquez2007porous}, see also \cite{bris2008existence} for uniqueness arguments for Fokker-Planck type equations. The most efficient proof which we have found for the hypotheses of Proposition \ref{prop: uniqueness FP} in the literature is a simple modification of the arguments leading to \cite[Theorem 3.3, Theorem 4.3]{fehrman2019large}, given in Appendix \ref{sec: misc proofs}.

\subsection{Properties of $\mathcal{J}$ and $\cA$.}\label{sec: variational} We next give a dual formulation of $\mathcal{J}$, which will be important both for the upper bound of Theorem \ref{thrm: LDP} in Section \ref{sec: UB} and the lower bound in Section \ref{sec: LB}. The proof also leads to an efficient proof of the existence, uniqueness and characterisation of an optimal $g$, which will be important in Section \ref{sec: gradient flow}. The following is a mild extension of \cite[Therorem 8.6]{fehrman2019large}. We recall the notation ${L^2_{t,x}}, \|\cdot\|_{L^2_{t,x}}$ for $L^2([0,\tf]\times\TTd,\RRd)$ and the associated norm from the introduction. \begin{lemma}\label{lemma: variational form} For $u_\bullet \in \cR$, the dynamic cost (\ref{eq: dynamic cost}) admits the variational representation \begin{equation}
	\begin{split}
			\label{eq: variational} \mathcal{J}(u_\bullet)=& \frac{1}{2}\inf\left\{\left\|g\right\|_{L^2_{t,x}}^2: u_\bullet \in \cS_g(u_0)\right\}=\sup_{\varphi \in C^{1,2}([0,\tf]\times \TTd)}\Xi_1(\varphi, u_\bullet)\end{split} \end{equation} where $\Xi_1(\varphi, u_\bullet)$ is the functional \begin{equation} \begin{split}\label{eq: def xi1} \Xi_1(\varphi, u_\bullet) = \langle \varphi_{\tf}, u_{\tf}\rangle & - \langle \varphi_0, u_0\rangle - \int_0^\tf \langle \partial_t \varphi_t, u_t\rangle \\& - \frac12\int_0^\tf \int_{\TTd}\left( \Delta \varphi_t(x) + |\nabla \varphi_t(x)|^2\right) u_t(x)^\alpha dt dx. \end{split}\end{equation} Moreover, the infimum over $g$ is uniquely attained, and the optimiser is the unique such $g$ which makes (\ref{eq: Sk}) hold and which belongs to \begin{equation}\label{eq: characterisation}
				\Lambda_{u_\bullet}:=\overline{\left\{u^{\alpha/2}\nabla \psi: \psi\in C^{1,2}([0,\tf]\times\TTd)\right\}}^{L^2_{t,x}}
			\end{equation} where the closure is in  the norm topology of ${L^2_{t,x}}$. Finally, the total rate function has the representation \begin{equation}\label{eq: whole rate function variational}
				\mathcal{I}_\rho(u_\bullet):=\sup\left\{\Xi_0(\psi,u_0,\rho)+\Xi_1(\varphi, u_\bullet): \psi \in C(\TTd), \varphi\in C^{1,2}([0,\tf]\times\TTd)\right\}
			\end{equation} where $\Xi_0$ is defined as \begin{equation}
				\Xi_0(\psi,u_0, \rho):=\langle \psi, u_0\rangle - \alpha \int_{\TTd} \rho(x)\left(e^{\psi(x)/\alpha}-1\right) dx. 
			\end{equation} \end{lemma} \begin{remark}\label{rmk: g tangent space}
				If $\mathcal{J}(u_\bullet)<\infty$, then testing against the function $1$ implies that $\lambda=\langle 1, u_t\rangle$ is constant in time. The characterisation of $g$ by (\ref{eq: characterisation}) implies that, under this assumption, the optimal choice of $g$ is such that $g_t$ belongs to the space $T_{u_t}\cM_{\text{ac},\lambda}$ defined in (\ref{eq: tangent space useful def}) for almost all $0\le t\le \tf$, which is part of the content of Theorem \ref{thrm: gradient flow}.		\end{remark} For completeness, a self-contained proof is given in Appendix \ref{sec: misc proofs}. The equality (\ref{eq: variational}) is part of the content of 
				 \cite[Equation (8.6)]{fehrman2019large}, and examining the standard proof of the equality one in fact constructs a control $g \in \Lambda_{u_\bullet}$ attaining the minimum. The characterisation (\ref{eq: whole rate function variational}) follows from (\ref{eq: variational}), a similar representation of the entropy $\cH_\rho$ and Lemma \ref{lemma: cgf} below. The uniqueness of the minimiser and the observation that there exists at most one $g$ with $u_\bullet \in \cS_g(u_0)$ and $g\in \Lambda_{u_\bullet}$ follow from elementary arguments.  For similar ideas in large deviations problems, let us cite \cite[Lemma 5.1]{kipnis1989hydrodynamics}, \cite[Section 10]{kipnis1998scaling} among many other references \cite{basile2015donsker,dawson1987large,kipnis1990large}. In Lemma \ref{lemma: variational form}, as in \cite{fehrman2019large}, we only the equality of the Laplace transform on the restricted set $\cR$, since limit trajectories outside $\cR$ are excluded separately by Theorem \ref{thrm: WtS}.\bigskip \\ 
			We will use the corresponding properties of the action functional $\cA(u_\bullet)$ defined in (\ref{eq: action 2}), which follow from a similar dual representation. \begin{lemma} \label{lemma: variational A} For any curve $u_\bullet \in \cDD$ with finite action $\cA$, the infimum over $\theta$ in the definition (\ref{eq: action 2}) is uniquely attained by a $\theta:[0,\tf]\times\TTd\to \RRd$, which is the unique such $\theta$ belonging to the same set $\Lambda_{u_\bullet}$ as in (\ref{eq: characterisation}). Moreover, if $u^{(n)}_\bullet \to u_\bullet$ in the topology of $\mathbb{D}$ with $\sup_n \cI_\rho(u^{(n)}_\bullet)<\infty$ for a constant $\rho \in (0,\infty)$, then \begin{equation}
					\label{eq: LSC of A} \cA(u_\bullet)\le \liminf_n \cA(u^{(n)}_\bullet). 
				\end{equation}\end{lemma} We will see in Appendix \ref{sec: misc proofs} that this follows from essentially the same ideas as Lemma \ref{lemma: variational form}.

				\section{Static Large Deviations at Equilibrium}\label{sec: equilibrium}In this section, we will prove some lemmata regarding the  equilibrium, respectively slowly varying local equilibrium, distributions  which we introduced in (\ref{eq: eq}) as \begin{equation} \label{eq: suitable equilibrium}
	\Pi^N_\rho(\eta^N):=\prod_{x\in \TND} \pi^N_{\rho}\left({\eta^N_x}\right), \qquad \eta^N\in X_N \end{equation} depending on whether $\rho$ is a constant or may vary on the macroscopic scale. In Lemma \ref{lemma: analysis of single-site}, we will control the mean and variance of the single-site distributions $\pi^N_\rho$, which produces a natural choice of initial data satisfying the hypotheses of Theorem \ref{th: hydrodynamic limit}. In Lemma \ref{lemma: exponential integrability of entropy}, we verify that the condition (\ref{eq: entropy finite hyp}) holds under the initial measure $\eta^N_0\sim \Pi^N_\rho$, which allows us to use Theorem \ref{thrm: WtS} in the proof of Theorem \ref{thrm: LDP}. Finally, Lemma \ref{lemma: cgf} finds the cumulant generating function of linear functionals $\langle \psi, \eta^N_0\rangle$ of the initial data under measures $\Pi^N_\rho$, which will allow us to identify the contribution to $\cI_\rho$ from the large deviations of the initial data.  \bigskip \\ In the course of proving these, we will frequently use an analysis of the normalisation functions $Z_\lambda$ defined to be \begin{equation} \label{eq: single site} Z_\lambda(\varphi):=\sum_{m\ge 0}\frac{\varphi^{\lambda m}}{(m!)^\lambda}, \qquad \varphi\ge 0
	\end{equation} which appears with $\lambda=\alpha$ as the normalisation constants in the single-site distributions $\pi^N_\rho$. The sum (\ref{eq: single site}) converges for any $\varphi\in [0,\infty)$ for any $\lambda>0$, so these are well-defined sums. \begin{lemma}\label{lemma: useful asymptotics}  \begin{enumerate}[label=\roman*).] \item For any $\lambda>0$,   \begin{equation}\label{eq: claim 2' single site} \frac{1}{\varphi}\log Z_\lambda(\varphi)\to \lambda \end{equation} as $\varphi\to \infty$. As a result, for any $a,b\in (0,\infty)$, \begin{equation}
	\label{eq: claim 2 single site} \chi_N\left(\log Z_\lambda\left(\frac{b}{\chi_N}\right)-\log Z_\lambda\left(\frac{a}{\chi_N}\right)\right) \to \lambda\left(b-a\right).
\end{equation} and the convergence is uniform in $a,b$ belonging to compact subsets of $(0,\infty)$. \item For $\lambda\ge 1$ and $k=1,2$, the functions \begin{equation} S_{k,\lambda}(\varphi):=\sum_{m\ge 0} \frac{m^k\varphi^{m\lambda-k}}{(m!)^\lambda} \end{equation} satisfy \begin{equation}\label{eq: claim 1 single site}
	\frac{S_{k,\lambda}(\varphi)}{Z_\lambda(\varphi)} \to 1 \text{ as }\varphi\to\infty. \end{equation}  \end{enumerate}
\end{lemma}
\begin{proof}
	We deal with the two claims separately. In the course of proving (\ref{eq: claim 2' single site}), we will find some non-asymptotic bounds of $Z_\lambda(\varphi)$, which we also use in the proof of the second claim (\ref{eq: claim 1 single site}). \paragraph{\textbf{Step 1. Upper and Lower bounds for $Z_\lambda$, and proof of (\ref{eq: claim 2 single site}).}} The ratio of successive terms in the sum is $(\frac{\varphi}{m+1})^\lambda$, and so the summand of $Z_\lambda(\varphi)$ is maximised at $m=\lfloor \varphi\rfloor$. Stirling's formula shows that\begin{equation}\label{eq: Y lower bound 1}  ce^{\lambda \varphi}\varphi^{-\lambda/2} \le \left(\frac{\varphi^{\lfloor \varphi\rfloor}}{\lfloor \varphi\rfloor!}\right)^\lambda \le  c^{-1}e^{\lambda \varphi}\varphi^{-\lambda/2} \end{equation} for some absolute constant $c$ and all sufficiently large $\varphi$. We will now use these observations to prove upper and lower bounds. \paragraph{\textbf{Step 1a. Upper Bound.}} We split the sum into $m\le \lfloor 2\varphi\rfloor, m>2\varphi$. In  $m\le \lfloor 2\varphi\rfloor$, there are at most $1+2\varphi$ terms, each of which is bounded above by (\ref{eq: Y lower bound 1}). For $m>2\varphi$, the ratio of successive terms is at most $2^{-\lambda}$, which implies that we can bound the tail above by a geometric series to obtain \begin{equation}
		\sum_{m> \lfloor 2\varphi\rfloor}\left(\frac{\varphi^m}{m!}\right)^\lambda \le C \left(\frac{\varphi^{\lfloor 2\varphi\rfloor}}{\lfloor 2\varphi\rfloor!}\right)^\lambda \le C e^{\lambda \varphi}\varphi^{-\lambda/2}
	\end{equation} for some constant $C=C(\lambda)$. In total, there are therefore at most $2\varphi+2$ such terms, all with the same upper bound as in (\ref{eq: Y lower bound 1}), and we conclude\footnote{Let us remark that, for the case $\lambda\ge 1$, an upper bound $Z_\lambda(\varphi)\le e^{\lambda \varphi}$ is readily obtained by a comparison of $l^p(\NN)$ norms. However, if we restricted this proof to $\lambda \ge 1$, then Lemma \ref{lemma: exponential integrability of entropy} would only hold for $\alpha\ge 2$ and we would have to restrict the main results to this range.} that, for all sufficiently large $\varphi$, \begin{equation}\label{eq: Y upper bound}
		Z_\lambda(\varphi)\le C\varphi^{1-\lambda/2}e^{\lambda \varphi}.
	\end{equation} \paragraph{\textbf{Step 1b. Lower Bound.}} We consider the contributions from the range $m\in I_\varphi=[\varphi-\sqrt{\varphi}, \varphi+\sqrt{\varphi}]\cap \NN$. For $m\ge \lfloor \varphi\rfloor+1, m\in I_\varphi$, \begin{equation}\begin{split}
		\left|\log \left(\frac{\varphi^{\lambda m}}{(m!)^\lambda}\right)-\log \left(\frac{\varphi^{\lambda \lfloor \varphi\rfloor}}{(\lfloor \varphi\rfloor!)^\lambda}\right)\right| & =\lambda\sum_{k=\lfloor \varphi\rfloor}^{m} \log \left(1+\frac{k+1-\varphi}{\varphi}\right) \\& \le 2\lambda\sum_{k=\lfloor \varphi\rfloor+1}^{m-1}  \frac{k+1-\varphi}{\varphi}\end{split} 
	\end{equation} as soon as $\varphi$ is large enough that $\{\frac{k+1-\varphi}{\varphi}: k\in I_\varphi\} \subset (-\frac12,\frac12)$. This sum has at most $\sim \sqrt{\varphi}$ terms, each of which is of the order $\varphi^{-1/2}$, and so is bounded uniformly in $m\in I_\varphi$ and all sufficiently large $\varphi$. The same argument applies to $m\in I_\varphi, m\le \lfloor \varphi\rfloor$, so for all sufficiently large $\varphi$, \begin{equation} \sup_{m\in I_\varphi}\left|\log \left(\frac{\varphi^{\lambda m}}{(m!)^\lambda}\right)-\log \left(\frac{\varphi^{\lambda \lfloor \varphi\rfloor}}{(\lfloor \varphi\rfloor!)^\lambda}\right)\right|  \le M \end{equation} for $M=M(\lambda)$. Up to a new value of $c$, (\ref{eq: Y lower bound 1}) gives \begin{equation}
		\label{eq: Y lower bound 2} \inf_{m\in I_\varphi} \left(\frac{\varphi^m}{m!}\right)^\lambda \ge c e^{\lambda \varphi}\varphi^{-\lambda/2}.
	\end{equation} Since there are $|I_\varphi|\sim 2\sqrt{\varphi}$ such terms, we conclude that, changing $c$ again, \begin{equation}\label{eq: Y lower bound}
		Z_\lambda(\varphi)\ge ce^{\lambda \varphi}\varphi^{(1-\lambda)/2}.
	\end{equation}Gathering (\ref{eq: Y upper bound}, \ref{eq: Y lower bound}), we conclude that $\varphi^{-1}\log Z_\lambda(\varphi)\to \lambda$, which is the claim (\ref{eq: claim 2' single site}).  The second claim (\ref{eq: claim 2 single site}) follows immediately. \paragraph{\textbf{Step 2. Proof of (\ref{eq: claim 1 single site}).}} We recall that, for the second claim, we specialise to $\lambda\ge 1$. Let us start by writing \begin{equation}S_{k,\lambda}(\varphi)-Z_\lambda(\varphi)=\sum_{m\ge 0} (m^k-\varphi^k)\frac{\varphi^{m-k}}{m!}\left(\frac{\varphi^m}{m!}\right)^{\lambda-1}
	\end{equation} whence, using (\ref{eq: Y lower bound}), there exists a constant $C$ such that\begin{equation}\begin{split} \left|\frac{S_{k,\lambda}(\varphi)}{Z_\lambda(\varphi)}-1\right| & \le Ce^{-\lambda \varphi}\varphi^{(\lambda-1)/2}|S_{k,\lambda}(\varphi)- Z_\lambda(\varphi)| \\ &=C\sum_{m\ge 0} \left|\frac{m^k}{\varphi^k}-1\right|\frac{\varphi^m e^{-\varphi}}{m!}\left(\frac{\varphi^{m+1/2} e^{-\varphi}}{m!}\right)^{\lambda-1}. \end{split} 
	\end{equation} Since $\lambda\ge 1$, the argument of step 1 shows that the final parenthesis is bounded, uniformly in $m, \varphi$, and so can be absorbed into the constant $C$. We now use Cauchy-Schwarz to see that, for $k\in \{1,2\}$, we can bound the remaining sum above by \begin{equation} \begin{split}
		\max_{k=1,2}\left|\frac{S_{k,\lambda}(\varphi)}{Z_\lambda(\varphi)}-1\right| & \le C\left(\sum_{m\ge 0} \left(\frac{m}{\varphi}-1\right)^2 \frac{\varphi^m e^{-\varphi}}{m!}\right)^{1/2}\left(\sum_{\ell \ge 0} \left(\frac{\ell}{\varphi}+1\right)^2 \frac{\varphi^\ell e^{-\varphi}}{\ell!}\right)^{1/2} \\[0.5ex]&\hspace{-2cm} = C\text{Var}\left[\left(\frac{X}{\varphi}\right): X\sim \text{Poisson}(\varphi)\right]^{1/2}\EE\left[\left(\frac{X}{\varphi}+1\right)^2: X\sim \text{Poisson}(\varphi)\right]^{1/2} \\[0.5ex] &\le C\varphi^{-1/2}.\end{split} \end{equation} and (\ref{eq: claim 1 single site}) is proven.  \end{proof}We now proceed with the lemmata. The first result concerns the single-site marginals $\pi^N_\rho$. \begin{lemma} \label{lemma: analysis of single-site} As $N\to \infty$, the rescaled mean of the single-site measure converges\begin{equation}\label{eq: single site mean}
\EE_{\pi^N_\rho}\left[\eta^N(0)\right]=\sum_{\eta \in \chi_N \mathbb{N}} \eta \h \pi^N_{\rho}(\eta) \to \rho
\end{equation} and \begin{equation} \label{eq: single site l2 bound}
	\EE_{\pi^N_\rho}\left[|\eta^N(0)-\rho|^2\right] \to 0 
\end{equation} where both convergences are uniform in $\rho$ belonging to compact subsets of $(0,\infty)$.  \end{lemma} \begin{remark}\label{rmk: slowly varying local equilibrium} 
	 This produces two natural ways of constructing initial data satisfying the hypotheses of Theorem \ref{th: hydrodynamic limit} - \ref{thrm: LDP}.  Fix $\rho\in C^\infty(\TTd, (0,\infty))$ and consider either the slowly-varying local equilibrium $\Pi^N_\rho$ or \begin{equation}
		\label{eq: truncated svle} {\Pi}^N_{\rho_0, M}(\eta^N):=\prod_{x\in \TND} \pi^N_{\rho_0(x),M}\left({\eta^N(x)}\right) 
	\end{equation} with $M>\sup_{x\in \TTd} \rho(x)$ and where $\pi^N_{\rho, M}$ is given by conditioning $\pi^N_\rho$ to $[0,M]$. It follows from (\ref{eq: single site l2 bound}) that $\EE_{\pi^N_{\rho,M}}[|\eta^N(0)-\rho|^2]\to 0$, uniformly in $\rho$ in compact subsets of $[0,M)$, and applying either conclusion to each lattice site, \begin{equation}
		\sup_{x\in \TND} \EE_{\Pi^N_{\rho}}\left[\left|\eta^N(x)-\rho(x)\right|^2\right]\to 0,\qquad  \sup_{x\in \TND} \EE_{\Pi^N_{\rho,M}}\left[\left|\eta^N(x)-\rho(x)\right|^2\right]\to 0. 
	\end{equation} 
	\end{remark} \begin{proof}[Proof of Lemma \ref{lemma: analysis of single-site}]  Both the mean and the variance of the single-site distribution may be expressed in terms of the sums $S_{k,\alpha}, Z_\alpha$ analysed in Lemma \ref{lemma: analysis of single-site}, now with $\lambda=\alpha$. For $k=1,2$, \begin{equation}
	\begin{split}\label{eq: rescaled mean 1} 
		\EE_{\pi^N_\rho}\left[\eta^N(0)^k\right] & =\sum_{m\ge 0} \chi_N^k \h  m^k\h  \frac{(\rho/\chi_N)^{\alpha m}}{(m!)^\alpha Z_\alpha(\rho/\chi_N)} \\ & =\rho^k \sum_{m\ge 0} \frac{m^k(\rho/\chi_N)^{\alpha m-k}}{(m!)^\alpha Z_\alpha(\rho/\chi_N)} =\frac{\rho^k S_{k,\alpha}(\rho/\chi_N)}{Z_\alpha(\rho/\chi_N)}.
	\end{split}
\end{equation} Thanks to (\ref{eq: claim 1 single site}), the quotient $(S_{k,\alpha}/Z_\alpha)(\rho/\chi_N)\to 1$ as $N\to \infty$, uniformly in $\rho$ bounded away from $0$, and hence the whole expression converges to $\rho^k$, uniformly in compact subsets of $(0,\infty)$. Both claims (\ref{eq: single site mean}, \ref{eq: single site l2 bound}) now follow.  \end{proof}We next check that the tails of the single-site marginals $\pi^N_{\rho}$ decay sufficiently rapidly for the estimate (\ref{eq: entropy finite hyp}) to hold. In particular, the hypothesis (\ref{eq: entropy finite hyp}) is natural, since it will hold for either true equilibrium or slowly-varying local equilibrium distributions. \begin{lemma}\label{lemma: exponential integrability of entropy} 
Let $\rho\in C(\TTd,[0,\infty))$. Then for all $\gamma<\alpha$  \begin{equation}\label{eq: entropy finite at equilibrium}
		\limsup_N \frac{\chi_N}{N^d}\log \EE_{\Pi^N_\rho}\left[\exp\left(\frac{N^d}{\chi_N}\gamma\cH(\eta^N_0)\right)\right] < \infty.
	\end{equation}
\end{lemma}  \begin{remark}\label{rmk: no better than entropy}
	  Let us note that this would be false if we replaced $u\log u$ by any more rapidly growing function in (\ref{eq: entropy finite at equilibrium}); indeed, by inspecting the proof, it can be seen that  \begin{equation}
		\EE\left[\exp\left(\frac{N^d}{\chi_N}\gamma\cH(\eta^N_0)\right)\right]=\infty
	\end{equation} as soon as\footnote{The sum in (\ref{eq: exponential expectation}) is readily seen to diverge if $\gamma>\alpha$ by using Stirling's formula and collecting powers of $m^m$ in the summand. In the edge case $\gamma=\alpha$, the factors of $m^m$ cancel, and the dominant term is a growing exponential, so the sum diverges.} $\gamma\ge \alpha$. This is consistent with the discussion of the role of entropy above Definition \ref{def: solutions}. \end{remark} \begin{proof}[Proof of Lemma \ref{lemma: exponential integrability of entropy}] First, let us remark that it is sufficient to prove the estimate with $\cH(\eta^N)$ replaced by $$\cH_+(\eta^N_0)=\int_{\TTd} \Psi(\eta^N_0(x))dx =\frac{1}{N^d}\sum_{x\in \TND} \Psi(\eta^N_0(x))$$ for $\Psi(u)=u\log u$, since $\cH-\cH_+ \le 1$ is bounded above, uniformly in $N$. For $\gamma<\alpha$, we write \begin{equation}\begin{split}
	\exp\left(\frac{N^d}{\chi_N}\gamma \cH_+(\eta^N_0)\right)&=\exp\left(\sum_{x\in \TND} \frac{\gamma \eta^N_{0}(x)}{\chi_N}\log \eta^N_{0}(x)\right)\\& = \prod_{x\in \TND} \exp\left(\frac{\gamma \eta^N_{0}(x)}{\chi_N}\log \eta^N_{0}(x)\right). \end{split}
\end{equation} Since $\eta^N_{0}(x), x\in \TND$ are independent under $\Pi^N_\rho$, the expectation of the last expression factorises into $N^d$ terms  $$ \EE_{\pi^N_{\rho(x)}}\left[\exp\left(\frac{\gamma \eta^N_{0}(x)}{\chi_N}\log \eta^N_{0}(x)\right)\right].$$ Putting everything together, we find that \begin{equation}\label{eq: reduction to 1 site} 
	\frac{\chi_N}{N^d}\log \EE_{\Pi^N_\rho}\left[\exp\left(\frac{N^d}{\chi_N}\gamma \cH_+(\eta^N_0)\right)\right] = \frac1{N^d}\sum_{x\in \TND} {\chi_N}\log \EE_{\pi^N_{\rho(x)}}\left[\exp\left(\frac{\gamma \eta^N_0(x)}{\chi_N}\log \eta^N_0(x)\right)\right].
\end{equation} It is now sufficient to find an upper bound on the summand, uniform in $x, N$. We have \begin{equation} \label{eq: exponential expectation}\begin{split}
	\EE_{\pi^N_{\rho(x)}}\left[\exp\left(\frac{\gamma \eta^N_0(x)}{\chi_N}\log \eta^N_0(x)\right)\right] &= \sum_{m\ge 0} \frac{e^{\gamma m\log(m\chi_N)}\rho(x)^{\alpha m}}{\chi_N^{\alpha m}(m!)^\alpha Z_\alpha(\rho(x)/\chi_N)} \\& = \sum_{m\ge 0}\frac{m^{\gamma m} \rho(x)^{\alpha m}}{\chi_N^{(\alpha-\gamma)m} (m!)^{\alpha} Z_\alpha(\rho(x)/\chi_N)}.\end{split}
\end{equation} We now sacrifice $(m!)^\gamma$ from the factor of $(m!)^\alpha$ in the denominator and use $$ (m!)^\gamma \ge C m^{\gamma m}e^{-\gamma m}$$ for all $m$ and some constant $C=C(\gamma)$ to find   \begin{equation}\begin{split}
	\frac{m^{\gamma m} \rho(x)^{\alpha m}}{\chi_N^{(\alpha-\gamma)m} (m!)^{\alpha} Z_\alpha(\rho(x)/\chi_N)}&\le C \frac{\rho(x)^{\alpha m}e^{\gamma m}}{\chi_N^{(\alpha-\gamma)m}(m!)^{\alpha-\gamma}Z_\alpha(\rho(x)/\chi_N)}\\& = \frac{C}{Z_\alpha(\rho(x)/\chi_N)}\left(\frac{(\rho(x)^{\alpha/(\alpha-\gamma)}e^{\gamma/(\alpha-\gamma)}))^m}{\chi_N^m m!}\right)^{\alpha-\gamma}.\end{split}\end{equation} Since $\alpha-\gamma>0$, the sum in (\ref{eq: exponential expectation}) now produces the function $Z_{\alpha-\gamma}$, and we find the upper bound  \begin{equation}
	\begin{split} \EE_{\pi^N_{\rho(x)}}\left[\exp\left(\frac{\gamma \eta^N_0(x)}{\chi_N}\log \eta^N_0(x)\right)\right] \le C \frac{Z_{\alpha-\gamma}(\rho(x)^{\alpha/(\alpha-\gamma)}e^{\gamma/(\alpha-\gamma)}/\chi_N)}{Z_\alpha(\rho(x)/\chi_N)}. \end{split} 
\end{equation} Taking the logarithm and multiplying by $\chi_N$, we find the upper bound  \begin{equation}\begin{split}
	& \chi_N\log \EE_{\pi^N_{\rho(x)}}\left[\exp\left(\frac{\eta^N_0(x)}{\chi_N}\log \eta^N_0(x)\right)\right] \\ & \hspace{3cm}  \le (1+\|\rho\|_\infty^{\alpha/(\alpha-\gamma)}e^{\gamma/(\alpha-\gamma)})\sup_{t\ge 1}\left[ t^{-1} \log Z_{\alpha-\gamma}(t)\right]  \\ & \hspace{3cm}  + (1+\|\rho\|_\infty)\sup_{s\ge 1}\left[ s^{-1} \log Z_{\alpha}(s)\right] + \chi_N \log C
\end{split} \end{equation}  where both of the suprema are finite, thanks to the first point of Lemma \ref{lemma: useful asymptotics}, and where $C$ is still allowed to depend on $\gamma<\alpha$.  This bound is now independent of $N$ and of $x\in \TND$, so returning to (\ref{eq: reduction to 1 site}), the claim is proven.  \end{proof}   Finally, we will use the following cumulant generating function: \begin{lemma}\label{lemma: cgf} For $\rho\in C(\TTd,(0,\infty))$ and any $\psi\in C(\TTd)$,\begin{equation}\label{eq: cgf}
	\frac{\chi_N}{N^d}\log \EE_{\Pi^N_\rho}\left[\exp\left(\frac{N^d}{\chi_N}\langle \psi, \eta^N_0\rangle \right)\right] \to \alpha \int_{\TTd}\rho(x)(e^{\psi(x)/\alpha}-1)dx. \end{equation} \end{lemma}  \begin{proof} Fix $\psi\in C(\TTd)$, and let us write $\overline{\psi}^N(x)$ for the averages over the cubes $c^N_x$, so that \begin{equation} \label{eq: representation of evaluation'} \langle \psi, \eta^N_0\rangle=\frac{1}{N^d}\sum_{x\in \TND} \overline{\psi}^N(x) \eta^N(x). \end{equation} As in the previous lemma, we factorise the variable in the integrand and use independence to see that \begin{equation}
	\begin{split} \frac{\chi_N}{N^d} \log \EE_{\Pi^N_\rho}\left[\exp\left(\frac{N^d}{\chi_N} \langle \psi, \eta^N_0\rangle\right) \right]& = \frac{\chi_N}{N^d}\log\left(\prod_{x\in \TND} \EE_{\Pi^N_\rho}\left[\exp\left(\frac{\overline{\psi}^N(x)\eta^N_0(x)}{\chi_N} \right) \right] \right) \\ & =\frac{1}{N^d}\sum_{x\in \TND} \chi_N\log \EE_{\pi^N_{\rho(x)}}\left[\exp\left(\frac{\overline{\psi}^N(x)\eta^N_0(x)}{\chi_N} \right) \right]. \end{split}
\end{equation}  We now compute the contribution from each site as \begin{equation}
	\EE_{\pi^N_{\rho(x)}}\left[\exp\left(\frac{\overline{\psi}^N(x)\eta^N_0(x)}{\chi_N} \right) \right]  = \sum_{m\ge 0}\frac{\rho(x)^{\alpha m}e^{m\overline{\psi}^N(x)}}{\chi_N^{\alpha m}Z_\alpha(\rho(x)/\chi_N)} =\frac{Z_\alpha(\rho(x) e^{\overline{\psi}^N(x)/\alpha}/\chi_N)}{Z_\alpha(\rho(x)/\chi_N)} \end{equation} so thanks to (\ref{eq: claim 2 single site}), \begin{equation}
	\chi_N \h \sup_{x\in \TND} \left|	\log \EE_{\pi^N_{\rho(x)}}\left[\exp\left(\frac{\overline{\psi}^N(x)\eta^N_0(x)}{\chi_N} \right) \right] - \alpha \rho(x) (e^{\overline{\psi}^N(x)/\alpha}-1)\right| \to 0\end{equation} where the uniformity follows from the fact that $\rho(x) e^{\overline{\varphi}^N(x)/\alpha}, \rho(x)$ are bounded and bounded away from $0$, uniformly in $x\in \TND$ and $N<\infty$. Summing, \begin{equation}
	\left|\frac{\chi_N}{N^d} \log \EE_{\Pi^N_\rho}\left[\exp\left(\frac{N^d}{\chi_N} \langle \psi, \eta^N_0\rangle\right) \right] - \frac{1}{N^d}\sum_{x\in \TND} \alpha \rho(x)(e^{\overline{\psi}^N(x)/\alpha}-1)\right| \to 0
\end{equation} while using continuity of $\rho,\psi$, \begin{equation}
	\frac{1}{N^d}\sum_{x\in \TND} \alpha \rho (x)(e^{\overline{\psi}^N(x)/\alpha}-1) \to \alpha \int_{\TTd}\rho(x)(e^{\psi(x)/\alpha}-1)dx
\end{equation} as desired.  \end{proof} 

\section{Discrete to Continuous Analysis}\label{sec: DTC} In this section, we will prove some facts about functions on the discrete lattice. We study the functionals $\cF_{\alpha,N}:\mathbb{D}_N\to [0,\infty)$, defined on the discrete path space $\mathbb{D}_N$:\begin{equation}\label{eq: FNA}
	\cF_{\alpha,N}(\eta^N_\bullet):=\sup_{t\le \tf} \cH(\eta^N_t) + \int_0^\tf \cD_{\alpha,N}(\eta^N_s)ds;
\end{equation} \begin{equation} \label{eq: DNA}\cD_{\alpha,N}(\eta^N):=\frac{1}{\alpha N^{d-2}}\sum_{x,y\in \TND} ((\eta^N(x)^{\alpha/2}-\eta^N(y)^{\alpha/2})^2 dp^N(x,y)\end{equation} and the continuum counterpart $\cF_\alpha:\cDD\to [0,\infty]$ \begin{equation} \label{eq: cF} \cF_\alpha(u_\bullet):=\sup_{t\le \tf} \cH(u_t)+\int_0^\tf \cD_\alpha(u_s)\end{equation} for $\cD_\alpha$ given by (\ref{eq: entropy dissipation}) and extended to $L^1_{\ge 0}(\TTd)$ as discussed in Section \ref{sec: formal}. The most important result of this section is the following lemma, which contains the functional analytic aspects of Theorem \ref{thrm: WtS}. 

\begin{lemma}[Deterministic Weak-to-Strong Principle]\label{lemma: DWtS}  Suppose $\eta^N_\bullet \in \mathbb{D}_N$ enjoy a uniform bound  
	 $\cF_{\alpha, N}(\eta^N_\bullet)\le z$ and converge in the topology of $\mathbb{D}$ to $u_\bullet \in \mathcal{C}$. Then, for some $\beta>\alpha$ given in Lemma \ref{prop: interpolation}, $\eta^N_\bullet, N\ge 1, u_\bullet$ are bounded in $L^\beta([0,\tf]\times\TTd)$ and $\eta^N_\bullet \to u_\bullet$ strongly in $L^\alpha([0,\tf]\times\TTd)$. Moreover, it also holds that $\cF_{\alpha}(u_\bullet)\le z$, so for any $u_\bullet \in \mathcal{C}$ we have the $\Gamma$-lower bound  \begin{equation}\label{eq: Gamma lower bound}\inf\left\{\liminf_N \cF_{\alpha, N}(\eta^N_\bullet): \eta^N_\bullet \in \mathbb{D}_N, (\eta^N_\bullet)_{N\ge 1} \to u_\bullet \right\} \ge \cF_\alpha(u_\bullet). \end{equation} Moreover, \begin{enumerate}[label=\roman*).] 
	 	\item If $u_\bullet \in \cR$ and $\lambda<\infty$, then for all open neighbourhoods $\cV\ni u_\bullet$ in the topology of $L^\alpha([0,\tf]\times\TTd)$, there exists $\cU\ni u_\bullet$ open in the topology of $\cDD$ such that, for all sufficiently large $N$, \begin{equation}\label{eq: ALS claim 1}
		\left\{\eta^N_\bullet\in \mathbb{D}_N: \cF_{\alpha, N}(\eta^N_\bullet)\le \lambda\right\} \cap \cU \subset \cV.
	\end{equation} \item If instead $\cF_\alpha(u_\bullet)=\infty$, then for any finite $\lambda>0$, there exists an open neighbourhood $\cU\ni u_\bullet$ in the topology of $\cDD$ such that, for all sufficiently large $N$, \begin{equation}\label{eq: ALS claim 2} \cU\cap \cDD_N\subset \{\eta^N_\bullet \in \cDD_N: \cF_{\alpha, N}(\eta^N_\bullet)>\lambda\}. \end{equation} 
	 \end{enumerate}\end{lemma} 
	 Together with the probabilistic estimates which we will obtain in Section \ref{sec: a priori}, the proof of Theorem \ref{thrm: WtS} will follow rapidly in Section \ref{sec: WTS}. \bigskip \\  Although we have not found a version of the Aubin-Lions-Simons lemma \cite{aubin1963theoreme,lions1969quelques,simon1986compact} from which the first assertion of Lemma \ref{lemma: DWtS} would follow, the same fundamental ingredients are used in the proof and the reader may find it helpful to keep this parallel in mind.  To illustrate the parallels, let us give without proof the following result, which follows from Lemma \ref{lemma: DWtS} and standard facts about the Skorokhod space $\mathbb{D}$, see Ethier and Kurtz \cite[Theorem 6.3]{ethier1987markov}, and is comparable to Theorem 1 of the work by Moussa \cite{moussa2016variants}. \begin{proposition}\label{prop: ALS parallel}
	 	For each $N$, let $\mathcal{K}_N \subset \mathbb{D}_N$ satisfy\begin{equation}\label{eq: prototype ALS entropy}\sup_N \sup_{\eta^N_\bullet \in \mathcal{K}_N} \sup_{t\le \tf} \cH(\eta^N_t)<\infty\end{equation} with the spatial regularity \begin{equation}\label{eq: prototype ALS spatial}\sup_N\sup_{\eta^N_\bullet \in \mathcal{K}_N} \int_0^\tf \cD_{\alpha,N}(\eta^N_s)ds<\infty \end{equation} and the temporal regularity \begin{equation}
	 		\label{eq: prototype ALS no jump} \sup_{\eta^N_\bullet \in \mathcal{K}_N} \sup\left\{  d(\eta^N_t, \eta^N_{t-}): t\le \tf\right\} \to 0; 
	 	\end{equation} \begin{equation}
	 		\label{eq: prototype ALS temporal} \lim_{\delta\downarrow 0}\h \sup_N \sup_{\eta^N_\bullet \in \mathcal{K}_N} \inf_{\{t_i\}} \max_i \sup_{s,t \in [t_i, t_{i+1})} d(\eta^N_s, \eta^N_t) \to 0
	 	\end{equation} where the infimum runs over all partitions $\{t_0=0<t_1<\dots<t_m=\tf\}$ of minimal spacing $\min_i (t_i-t_{i-1})>\delta$. Then  $\cup_N \mathcal{K}_N$ is precompact in the norm topology of $L^\alpha([0,\tf]\times\TTd)$. 
	 \end{proposition} As in \cite[Theorem 1]{moussa2016variants}, the bound on $\cF_{\alpha,N}$  or (\ref{eq: prototype ALS spatial}) is a sort of $L^2_tH^1_x$ regularity for the nonlinear transformation $\Phi(\eta^N(x))=(\eta^N(x))^{\alpha/2}$, and the time regularity (\ref{eq: prototype ALS temporal}) is slightly stronger than what is assumed in \cite[Theorem 1]{moussa2016variants} and produces compactness in $\mathbb{D}$ by \cite[Theorem 6.3]{ethier1987markov}. This is the ingredient that is deferred to probabilistic considerations in Lemma \ref{lemma: ET}. In Lemma \ref{lemma: DWtS} above, this and (\ref{eq: prototype ALS no jump}) are eliminated by assuming convergence to $u_\bullet \in \mathcal{C}$ relative to the topology of $\mathbb{D}$. \bigskip \\  We now turn to the proof of Lemma \ref{lemma: DWtS} by the following sequence of lemmata. We first prove how the estimate on $\cF_{\alpha,N}$ implies some additional integrability. \begin{lemma}[Integrability via discrete $\alpha$-entropy dissipation] \label{prop: interpolation} For some $\beta>\alpha$ and some $C=C(d,\alpha)$, \begin{equation}
		\left\|\eta^N\right\|_{L^\beta(\TND)}^\beta\le C\left(1+\langle 1, \eta^N\rangle\right)\left(1+\cD_{\alpha,N}(\eta^N)\right)
	\end{equation} for all $\eta^N\in L^1_{\ge 0}(\TND)$. \end{lemma}
	\begin{proof} Let us define functions by $\xi^N:\TND\to [0,\infty)$ by $\xi^N(x):=\eta^N(x)^{\alpha/2}$ so we have the equality $ \cD_{\alpha,N}(\eta^N_t)=\frac{2}{\alpha}\cD_{2,N}(\xi^N)$.  By discrete versions of the Sobolev-Garaglido-Nirenberg inequality, se Porretta \cite[Section 4]{porretta2020note} and the Poincar\'e inequality, there exists $p>2$ and an absolute constant $C$ such that \begin{equation}\label{eq: SNG on xi}
	\|\xi^N\|_{L^p(\TND)}^2 \le C\left(\|\xi^N\|_{L^1(\TND)}^2+ \cD_{2,N}(\xi^N)\right).
\end{equation} Let us now choose $0<\gamma<\frac{2}{\alpha}$ such that $\frac{2(1-\gamma)}{2-\alpha \gamma}<p$, and observe that by H\"older's inequality,\begin{equation}\begin{split}
	\|\xi^N\|_{L^1(\TND)} & \le \left\|\xi^N\right\|_{L^p(\TND)}^{(1-\gamma)} \|\eta^N\|_{L^1(\TND)}^{\alpha \gamma/2} \\[1ex] & \le \frac{1}{2\sqrt{C}}\|\xi^N\|_{L^p(\TND)}+C' \|\eta^N_t\|_{L^1(\TND)}^{\alpha/2} \end{split}
\end{equation} where $C$ is the constant in (\ref{eq: SNG on xi}) and $C'$ is a new constant, still depending only on $\alpha, d$. Substituting everything back, we have proven that for a new constant $C$, \begin{equation}\begin{split}
	\left\|\eta^N\right\|_{L^{\alpha p/2}(\TND)}^{\alpha}=\|\xi^N\|_{L^p(\TND)}^2&\le   C\left(\|\eta^N\|_{L^1(\TND)}^{\alpha}+ \cD_{2,N}(\xi^N)\right)\\& = C\left(\|\eta^N\|_{L^1(\TND)}^{\alpha}+ \cD_{\alpha,N}(\eta^N)\right). \end{split}
\end{equation} We finally use the interpolation of $L^p$ norms \begin{equation}
	\left\|\eta^N\right\|_{L^{\alpha+1-(2/p)}(\TND)}^{\alpha+1-(2/p)} \le \left\|\eta^N\right\|_{L^1(\TND)}^{1-(2/p)}\left\|\eta^N\right\|^{\alpha}_{L^{\alpha p/2}(\TND)}
\end{equation} to conclude that, for $\beta=\alpha+1-2/p$, that \begin{equation}
	\frac1{N^d}\sum_{x\in \TND}(\eta^N(x))^\beta \le C\left(1+\langle 1, \eta^N\rangle^\beta\right)\left(1+\cD_{\alpha,N}(\eta^N)\right) \end{equation} as claimed. \end{proof}
	
	The next two lemmata concern local averages of discrete and continuous functions. Let us fix a smooth, even approximation of the unit $\varrho:\RRd\to [0,\infty)$ with support in $[-\frac{1}{2},\frac{1}{2}]^d$, and for $r<1$, we define the associated mollifiers $\varrho_r(x):=r^{-d}\varrho(x/r)$ on $x\in \TTd$.  For $\eta^N \in L^1(\TND)$, we now define the local averages $\overline{\eta}^{N,r}\in L^1(\TND)$ by specifying, for $x\in \TND$,\begin{equation}\label{eq: convolution} 
		\overline{\eta}^{N,r}(x):=\frac{1}{\kappa_{r,N} N^d} \sum_{y\in \TND} \eta^N(y)\varrho_r(x-y)=\frac{1}{\kappa_{r,N} N^d r^d}\sum_{y\in \TND} \varrho((x-y)/r)\eta^N(y)
	\end{equation} where $\kappa_{r,N}$ is the normalising factor $$ \kappa_{r,N}:=\frac{1}{N^d}\sum_{y\in \TND} \varrho_r(y) \to 1.$$ We choose this normalisation in order to avoid the appearance of $r$-dependent constants when exploiting Jensens' inequality to transfer estimates from $\eta^N$ to $\overline{\eta}^{N,r}$. \bigskip \\   We still consider these as elements of $L^1(\TND)$, so that the values at points in $\TTd\setminus \TND$ are given by the extension of (\ref{eq: convolution}) which is given by the value (\ref{eq: convolution}) for all $z$ in the small box $z\in c^N_x$. For continuum functions $u\in L^1(\TTd)$, we define analagously $\overline{u}^r(x):=\varrho_r\star u(x)$. 	
	\begin{lemma}\label{lemma: nonstandard block convolution via dirichlet} Fix $r>0$ and $a\in (0,\infty)$. Then for any  $\eta^N\in L^1_{\ge 0}(\TND)$ with $\langle 1, \eta^N\rangle \le a$,  \begin{equation}\label{eq: nonstandard block convolution via dirichlet 1}
	\left\|\eta^N-\overline{\eta}^{N,r}\right\|_{L^\alpha(\TND)}^\alpha \le C(1+a)r^\gamma(1+\cD_{\alpha,N}(\eta^N))
	\end{equation} for some absolute constant $C$ and an exponent $\gamma=\gamma(\alpha)>0$. 
		
	\end{lemma}\begin{proof}In the case $\alpha=2$, the $\alpha$-entropy dissipation $\cD_{\alpha,N}$ coincides with a discrete Dirichlet form, and the conclusion (\ref{eq: nonstandard block convolution via dirichlet 1}) is a piece of folklore, which we will assume. We will now show how to extend this to any $\alpha\ge 1$. \bigskip \\ Let us fix $\eta^N \in L^1_{\ge 0}(\TND)$ and define, for $\delta\in(0,1)$, the function\begin{equation}
		\eta^{N, \delta}(x):=\delta\lor \eta^N(x)\land \delta^{-1} 
	\end{equation} so that $\eta^{N,\delta}\in L^1_{\ge 0}(\TND)$ with $\langle 1, \eta^{N,\delta}\rangle \le a+\delta \le a+1$.We define $\overline{\eta}^{N,\delta,r}$ in the same way as $\overline{\eta}^{N,r}$,  with $\eta^{N,\delta}$ in place of $\eta^N$. By applying Lemma \ref{prop: interpolation}, we have \begin{equation} \label{eq: beta norm} \|\eta^{N,\delta}\|_{L^\beta(\TTd)}^\beta \le C(1+a)(1+\cD_{\alpha,N}(\eta^N)) \end{equation} for some $\beta>\alpha$, and the same extends to $\|\overline{\eta}^{N,\delta,r}\|_{L^\beta(\TND)}^\beta$ by Jensen's inequality.  \paragraph{\textbf{Step 1. From $\eta^{N,\delta}$ to $\overline{\eta}^{N,\delta,r}$ in $L^2(\TTd)$.}}  We begin with the observation that, for any $\alpha\ge 1$ and for all $w,z\in [\delta, \delta^{-1}]$, \begin{equation}
		(w^{\alpha/2}-z^{\alpha/2})^2 \ge \left(\frac{\alpha}{2}\right)^2 \delta^{|\alpha-2|}(z-w)^2	\end{equation} which implies that \begin{equation}
			\cD_{2,N}(\eta^{N,\delta})\le C\d^{-|\alpha-2|}\cD_{\alpha,N}(\eta^N)
		\end{equation} for some $C=C(\alpha)$ independent of $N, \delta,r$. We now apply the case $\alpha=2$ to  $\eta^{N,\delta}$ to find \begin{equation}
			\label{eq: standard dirichlet to block} \left\|\eta^{N,\delta}-\overline{\eta}^{N,\delta,r}\right\|_{L^2(\TTd)}^2 \le C\delta^{-|\alpha-2|}(1+a)r^{\gamma_0} \left(1+\cD_{\alpha,N}(\eta^N)\right)
		\end{equation} where $\gamma_0$ is the exponent in the case $\alpha=2$. \paragraph{\textbf{Step 2. From $\eta^{N,\delta}$ to $\overline{\eta}^{N,\delta,r}$ in $L^\alpha(\TTd)$.}} We now convert (\ref{eq: standard dirichlet to block}) to an estimate in $L^\alpha(\TTd)$, differentiating between the cases $\alpha\le 2, \alpha>2$. In either case, we will prove that \begin{equation} \label{eq: alpha norm claim} \begin{split}
			\left\|\eta^{N,\delta}-\overline{\eta}^{N,\delta,r}\right\|_{L^\alpha(\TTd)}^\alpha &\le C(1+a)\delta^{-\lambda_1}r^{\lambda_2}(1+\cD_{\alpha,N}(\eta^N))
		\end{split}\end{equation} for some exponents $\lambda_1, \lambda_2>0$ depending only on $\alpha$. \paragraph{\textbf{2a. Case 1: $\alpha\le 2$.}} In the case $\alpha\le 2$, the monotonicity of $L^p$ norms implies that (\ref{eq: standard dirichlet to block}) still holds with the $\|\cdot\|_{L^\alpha(\TTd)}$ norm in place of $\|\cdot\|_{L^2(\TTd)}$. The claim (\ref{eq: alpha norm claim}) follows by raising both sides to the power of  $\frac\alpha 2 \le 1$. \paragraph{\textbf{2b. Case 2: $\alpha> 2$.}} In this case, we recall that $\eta^{N,\delta}$ and $\overline{\eta}^{N,\delta,r}$ are bounded in $L^\beta(\TTd)$ by (\ref{eq: beta norm}). The claim now follows by recalling that $\alpha \in (2, \beta)$ and using the interpolation \begin{equation} \label{eq: alpha norm after interpolation} \begin{split}
			\left\|f\right\|_{L^\alpha(\TTd)}^\alpha &\le \left\|f\right\|_{L^2(\TTd)}^\frac{2(\beta-\alpha)}{\beta-2}\left\|f\right\|_{L^\beta(\TTd)}^\frac{\beta(\alpha-2)}{\beta-2} 
		\end{split}\end{equation} for $f\in L^\beta(\TTd)$.  \paragraph{\textbf{Step 3. From $\eta^N$ to $\eta^{N,\delta}$ in $L^\alpha(\TTd)$.}} Finally, let us estimate the error induced by the truncation. Using a Chebychev estimate for the truncation at $\delta^{-1}$, we estimate \begin{equation} \label{eq: small truncation error}
			\|\eta^{N,\delta}-\eta^N\|^\alpha_{L^\alpha(\TTd)} \le \delta^\alpha+ \delta^{\beta-\alpha}\left\|\eta^N\right\|_{L^\beta(\TTd)}^\beta
		\end{equation} and by Jensen's inequality the same extends to $\|\overline{\eta}^{N,\delta,r}-\overline{\eta}^{N,r}\|_{L^\alpha(\TTd)}$. Recalling (\ref{eq: beta norm}) again, we find that \begin{equation} \label{eq: small truncation error'}\begin{split}& \left\|\eta^{N,\delta}-\eta^N\right\|_{L^\alpha(\TTd)}^\alpha + \left\|\overline{\eta}^{N,\delta,r}-\overline{\eta}^{N,r}\right\|_{L^\alpha(\TTd)}^\alpha \\& \hspace{4cm} \le C(1+a)(\delta^\alpha+\delta^{\beta-\alpha})(1+\cD_{\alpha,N}(\eta^N)). \end{split}\end{equation} \paragraph{\textbf{Step 4. Reassembly and Conclusion.}} Gathering (\ref{eq: alpha norm claim}) and (\ref{eq: small truncation error'}), we find that for any $\delta<1$, \begin{equation}
			\left\|\eta^N-\overline{\eta}^{N,r}\right\|_{L^\alpha(\TTd)}^\alpha \le C(1+a)(1+\cD_{\alpha,N}(\eta^N))\left(\delta^\alpha +\delta^{\beta-\alpha}+ \delta^{-\lambda_1}r^{\lambda_2}\right).
		\end{equation} We now take $\delta=r^\frac{\lambda_2}{\min(\alpha, \beta-\alpha)+\lambda_1}\in (0,1)$ to conclude the claim for a new exponent $\gamma$. \end{proof}  The final lemma we will need before turning to the proof of Lemma \ref{lemma: DWtS} establishes a consistency between the averaging $\eta^N\mapsto \overline{\eta}^{N,r}$ on the torus $\TND$ and the continuum averaging $u\mapsto \overline{u}^r$. \begin{lemma}[Convergence of local averages] \label{lemma: convergence of convolution} Let $\eta^N_\bullet \in \mathbb{D}_N$ be a sequence converging, in the topology of $\mathbb{D}$, to $u_\bullet \in \mathcal{C}$. Then for any $r>0$, \begin{equation}\label{eq: convergence of convolution}
	\sup_{t\le \tf}\sup_{x\in \TTd}
\left|\overline{\eta}^{N,r}_{t}(x)- \overline{u}^r_t(x)\right| \to 0. \end{equation}
	
\end{lemma} \begin{proof}
	Firstly, the assumption $u_\bullet \in \mathcal{C}$ and the convergence $\eta^N_\bullet\to u_\bullet$ in $\mathbb{D}$,  it follows from definition of the topology on $\mathbb{D}$ that \begin{equation} \label{eq: unf} \sup_{t\le \tf}\left|\langle \varphi, \eta^N_t-u_t\rangle\right| \to 0 \end{equation} for any $\varphi \in C(\TTd)$. It follows that the masses $\langle 1, \eta^N_t\rangle$, $\langle 1, u_t\rangle$ are bounded uniformly in $N$ and $t$. Since $\overline{\eta}^{N,r}$ is defined to be constant on the boxes $c^N_x, x\in \TND$ and $\overline{u}^r_t$ is uniformly continuous in $x$, it is now sufficient to prove (\ref{eq: convergence of convolution}) with the supremum restricted to $x\in \TND$. We next observe that, for $x\in \TND$, \begin{equation}\label{eq: representation of convolution} \overline{\eta}^{N,r}(x)=\frac{1}{\kappa_{r,N}}\left\langle \varrho_r(\cdot-x),\eta^N\right \rangle = \frac{1}{\kappa_{r,N}}\int_{\TTd}\varrho_r(y-x)\eta^N(y)dy. \end{equation}Let us clarify that the third expression is the definition of the second, but this is not the definition of $\overline{\eta}^{N,r}(z)$ for $z\in \TTd\setminus \TND$. Since the functions $\varrho_r(\cdot-x)$ are continuous, we apply (\ref{eq: unf}) for each $x$ to see that\begin{equation} \label{eq: lower bound conv of conv 1}
		\sup_{t\le \tf}\left|\langle \varrho_r(\cdot-x), \eta^N_t -u_t\rangle \right|\to 0.
	\end{equation}  Next, observe that the family $(\varrho_r(\cdot-x))_{x\in \TTd}$ is compact in the uniform topology by Arzel\`a-Ascoli, and the aforementioned bound on the mass implies the continuity of $x\mapsto \langle \varrho_r(\cdot-x), \eta^N_t\rangle$, uniformly in $N$ and $t$. It therefore follows that the convergence in (\ref{eq: lower bound conv of conv 1}) is further uniform in $t\le \tf$ and $x\in \TTd$, that is,  \begin{equation}
		\sup_{t\le \tf}\sup_{x\in \TTd}\left|\langle \varrho_r(\cdot-x),\eta^N_t-u_t\rangle \right|\to 0. 
	\end{equation} Finally, the normalisation $\kappa_{r,N}\to 1$ is independent of $x$, and returning to (\ref{eq: representation of convolution}) we conclude that \begin{equation}
		\sup_{t\le \tf} \sup_{x\in \TND} \left|\overline{\eta}^{N,r}_t(x)-\overline{u}^r_t(x)\right|\to 0
	\end{equation} and the proof is complete.  \end{proof}
 We may now assemble the various parts to prove Lemma \ref{lemma: DWtS}.
\begin{proof}[Proof of Lemma \ref{lemma: DWtS}] First of all, we observe that $u_\bullet \in \mathcal{C}$ and the convergence in $\mathbb{D}$ imply that $a=\sup_N \sup_t \langle 1, \eta^N_t+u_t\rangle<\infty$, and the boundedness of $\eta^N_\bullet$ in $L^\beta([0,\tf]\times\TTd)$ follows immediately from Proposition \ref{prop: interpolation}. It is readily checked that this norm is lower semicontinuous for the topology of $\mathbb{D}$, which also implies that $u_\bullet\in L^\beta([0,\tf]\times\TTd)$.  \paragraph{\textbf{Step 1. Strong Convergence in $L^\alpha([0,\tf]\times\TTd)$.}} We next show how to improve the convergence to norm convergence in $L^\alpha([0,\tf]\times\TTd)$. Fix $\e>0$, and observe that for all $t\le \tf$ such that $u_t\in L^\alpha(\TTd)$,  \begin{equation}
	\left\|\overline{u}^r_t-u_t\right\|^\alpha _{L^\alpha(\TTd)}\to 0
\end{equation} by standard properties of convolutions. Thanks to the gain of integrability, this holds for almost all $t\le \tf$, and the left-hand side is dominated by $2^\alpha\|u_t\|_{L^\alpha(\TTd)}^\alpha \in L^1([0,\tf])$, so by dominated convergence it holds for all $r>0$ small enough that \begin{equation}\label{eq: u to local avg}
	\int_0^\tf \left\|\overline{u}^r_t-u_t\right\|_{L^\alpha(\TTd)}^\alpha dt<\e^\alpha. 
\end{equation} We now apply Lemma \ref{lemma: nonstandard block convolution via dirichlet} to see that, for all $t$, \begin{equation}
	\|\overline{\eta}^{N,r}_t-\eta^{N}_t\|_{L^\alpha(\TTd)}^\alpha \le C(1+a)r^\gamma(1+\cD_{\alpha,N}(\eta^{N}_t)) \end{equation} for some $\gamma>0$. Integrating in time, we conclude that \begin{equation} \label{eq: etan to local avg} \int_0^\tf \left\|\overline{\eta}^{N,r}_t-\eta^N_t\right\|_{L^\alpha(\TTd)}^\alpha dt \le C(1+a)r^\gamma\left(\tf+\int_0^\tf \cD_{\alpha,N}(\eta^N_s)ds\right) \end{equation} and the last factor is at most $\tf+\cF_{\alpha,N}(\eta^N_\bullet)$, which is bounded in $N$. In particular, we may choose $r>0$ small enough that (\ref{eq: u to local avg}) holds and that the right-hand side of (\ref{eq: etan to local avg}) is at most $\e^\alpha$ for all $N$. Since $u_\bullet$ was assumed to be in $\mathcal{C}$, it follows from \ref{lemma: convergence of convolution} that  \begin{equation} \label{eq: discrete local avg to continuous local avg 2} \|\overline{\eta}^{N,r}_\bullet - \overline{u}^r_\bullet\|_{L^\alpha([0,\tf]\times\TTd)}\to 0.\end{equation} Combining with (\ref{eq: u to local avg}, \ref{eq: etan to local avg}), we finally conclude that \begin{equation}
	\limsup_N \left\|\eta^{N}_\bullet-u_\bullet\right\|_{L^\alpha([0,\tf]\times\TTd)} \le 2\e.
\end{equation} and since $\e>0$ was arbitrary, the step is complete. \paragraph{\textbf{Step 2. $\Gamma$-lower bound.}} We next prove the $\Gamma$-lower bound (\ref{eq: Gamma lower bound}). First, because $u_\bullet \in \mathcal{C}$, $\eta^N_t\to u_t$ in the weak-$\star$ topology of $L^1(\TTd)$ for all $t$, and by lower semicontinuity of $\cH$ in this topology, \begin{equation} \label{eq: gamma conv 1}
	\sup_t \cH(u_t)\le \liminf_N \sup_t \cH(\eta^N_t).
\end{equation} For the other term, set $\xi^N_t(x):=(\eta^N_t(x))^{\alpha/2}$ and $v_t(x):=u_t(x)^{\alpha/2}$. From step 1, $\xi^N \to v$ strongly in $L^1([0,\tf]\times\TTd)$, and hence strongly in $L^1(\TTd)$ for almost all $t\in [0,\tf]$. By the well-known $\Gamma$-convergence $\Gamma-\lim \cD_{2,N}=\frac12\left\|\h\cdot\h\right\|_{\dot{H}^1(\TTd)}^2$ in this topology, it follows that for almost all $t$, \begin{equation} \label{eq: pw liminf} \cD_{\alpha}(u_t)=\frac{2}{\alpha}\cD_2(v_t)\le \frac2\alpha\h\liminf_N \cD_{2,N}(\xi^N_t) = \liminf_N \cD_{\alpha,N}(\eta^N_t). \end{equation} Integrating in time and using Fatou's lemma, it follows that \begin{equation}\label{eq: integraged liminf}
	\int_0^\tf \cD_\alpha(u_t) dt \le \liminf_N \int_0^\tf \cD_{\alpha,N}(\eta^N_t)dt
\end{equation} and combining with (\ref{eq: gamma conv 1}), the claim (\ref{eq: Gamma lower bound}) follows. \paragraph{\textbf{Step 3. Construction of $\cDD$-Open Neighbourhoods.}}  We now deduce (\ref{eq: ALS claim 1} - \ref{eq: ALS claim 2}) by contradiction from steps 1-2.  \paragraph{\textbf{3a. Proof of (\ref{eq: ALS claim 1}).}} Let us fix $u_\bullet, \lambda, \cV$ as in the claim, and suppose for a contradiction that no such $\mathcal{U}$ exists. Recalling that the topology of $\mathbb{D}$ is induced by a metric \cite{jakubowski1986skorokhod}, there exists a decreasing sequence $\cU_k$ of $\cDD$-open sets whose intersection is the singleton $\{u_\bullet\}$. By the contradiction assumption applied to $\cU_k$, we can choose an increasing sequence $N_k\to \infty$ and $\eta^{N_k}_\bullet \in (\cU_k\setminus \cV)\cap \cDD_{N_k}$ enjoying the uniform bound $\cF_{\alpha, N_k}(\eta^{N_k}_\bullet)\le \lambda$. By the choice of $\cU_k$, we have $\eta^{N_k}_\bullet \to u_\bullet$ in the topology of $\cDD$, and it follows by step 1  that $\eta^{N_k}_\bullet \to u_\bullet$ in the topology of $L^\alpha([0,\tf]\times\TTd)$. On the other hand, this is absurd, since $\eta^{N_k}_\bullet \not \in \cV$ for any $k$, and we conclude by contradiction that (\ref{eq: ALS claim 1}) holds.  \\ \paragraph{\textbf{3b. Proof of (\ref{eq: ALS claim 2}).}} The other case is similar. If we assume for a contradiction that (\ref{eq: ALS claim 2}) is false, the same choice of $\cU_k$ produces $N_k\to \infty$, $\eta^{N_k}_\bullet \in \mathbb{D}_{N_k}$ with $\eta^{N_k}_\bullet \to u_\bullet$ in the topology of $\mathbb{D}$, while $\cF_{\alpha,N_k}(\eta^{N_k}_\bullet)\le \lambda$. This is the setting of step 2, and we conclude that $\cF_\alpha(u_\bullet)<\infty$, in contradiction to the hypothesis of the claim. \end{proof} \begin{remark}\label{rmk: continnum nonstandard convolution} Exactly the same arguments also establish a continuum version of all the same results. Replacing Lemma \ref{lemma: nonstandard block convolution via dirichlet}, for $u\in L^1_{\ge 0}(\TTd)$ with $u^{\alpha/2}\in \dot{H}^1(\TTd)$ and bounded mean $\langle 1, u\rangle \le a$,  \begin{equation}
			\|\overline{u}^r-u\|_{L^\alpha(\TTd)}^\alpha \le C(1+a)r^\gamma(1+\cD_\alpha(u)) 
		\end{equation} for the same $\gamma>0$. Overall, if $u^{(n)}_\bullet \to u_\bullet \in \mathcal{C}$ relative to the topology of $\mathbb{D}$ and $\limsup_n \int_0^\tf \cD_\alpha(u^{(n)}_s)ds<\infty$, then $u^{(n)}_\bullet, u_\bullet$ are bounded in $L^\beta([0,\tf]\times\TTd)$, for the same $\beta>\alpha$, that $\|u^{(n)}_\bullet-u_\bullet\|_{L^\alpha([0,\tf]\times\TTd)}\to 0$ and that \begin{equation}\label{eq: lsc of D}
			\int_0^\tf \cD_\alpha(u_s)ds \le \liminf_n \int_0^\tf \cD_\alpha(u^{(n)}_s)ds.
		\end{equation}\end{remark} 
\section{Fundamental \emph{a priori} Large Deviations Estimate} \label{sec: a priori}  In this section we will prove an \emph{a priori} estimate for the particle process, which we will extensively use in the sequel.  We will see below that the scaling hypothesis (\ref{eq: scaling hypothesis}) allows us to find a genuinely probabilistic estimate on the discrete entropy dissipation at the level of realisations $\eta^N_\bullet(\omega)$, rather than the analytic cotravellers given in terms of $\text{Law}(\eta^N_t)$. Together with Lemma \ref{lemma: DWtS}, we will then be able to prove Theorem \ref{thrm: WtS}. \bigskip \\  The strategy will be to consider the exponential martingale $\mathcal{Z}^N_t$ associated to $\cH(\eta^N_t)$ according to the general construction (\ref{eq: exponential entropy martingale}). The key computation, which we make precise in Lemma \ref{lemma: computation of entropy dissipation} below, relates $\cG_N(\frac\alpha 2\cH)$, for $\cG_N$ the nonlinear generator (\ref{eq: nonlinear generator}), to the discrete entropy dissipation $\cD_{\alpha,N}$ defined at (\ref{eq: DNA}). This will allow us to find large-deviation estimates on $\cF_{\alpha,N}(\eta^N_\bullet)$ in Lemma \ref{lemma: QY1}, which we use to show a time-continuity property and exponential tightness in Lemma \ref{lemma: ET}. \bigskip \\ We work under the standing assumption (\ref{eq: entropy finite hyp}), which immediately implies that \begin{equation} \label{eq: bound mass hyp}\limsup_{a\to \infty} \limsup_N \frac{\chi_N}{N^d}\log \PP\left(\eta^N_0\not \in X_{N,a}\right)=-\infty\end{equation} where $X_{N,a}$ is the space of configurations with mass bounded by $a$, defined in (\ref{eq: XNA}).\bigskip \\  The key estimate is as follows: \begin{lemma}\label{lemma: computation of entropy dissipation}There exists an absolute constant $C$ such that, for all $N<\infty$ and all $\eta^N\in X_N$, \begin{equation}\label{eq: central estimate}\cG_N\left(\frac{\alpha}{2}\cH\right)(\eta^N)\le - \frac{\alpha}{2} \cD_{\alpha,N}(\eta^N)+ C\left(N^2 \chi_N^{\min(\alpha/2,1)}\right)\left(1+\left\|\eta^N\right\|^\alpha_{L^\alpha(\TND)}\right).
\end{equation} In particular, if (\ref{eq: scaling hypothesis}) holds, then for any $\vartheta\in (0, \alpha/2)$ and $a\ge 1$, we can find $C=C(\vartheta,a,\alpha)$ such that, whenever $\eta^N\in X_{N,a}$,   \begin{equation}\label{eq: central estimate 2}\cG_N\left(\frac{\alpha}{2}\cH\right)(\eta^N)\le  -\left(\frac{\alpha}{2}-\vartheta\right) \cD_{\alpha,N}(\eta^N)+ C.
\end{equation}  \end{lemma}
\begin{proof} We start by applying the definition (\ref{eq: nonlinear generator}) to find \begin{equation} \label{eq: start GN}\begin{split}& \cG\left(\frac\alpha2 \cH\right)(\eta^N)\\& \hspace{1cm}=\frac{\chi_N}{N^d}\cdot \frac{dN^2}{\chi_N}\sum_{x,y\in \TND} p^N(x,y)(\eta^N(x))^\alpha\left(e^{N^d \alpha (\cH(\eta^{N,x,y})-\cH(\eta^N))/2\chi_N}-1\right)\end{split}\end{equation} where we understand the summand to be $0$ if $\eta^N(x)=0$. Recalling the notation $\Psi(u):=u\log u$, the difference appearing in the exponent is \begin{equation}\begin{split}&\frac{N^d}{\chi_N}\left(\cH(\eta^{N,x,y})-\cH(\eta^N)\right)\\&\hspace{1cm} =\frac{\Psi(\eta^N(x)-\chi_N)-\Psi(\eta^N(x))+\Psi(\eta^N(y)+\chi_N)-\Psi(\eta^N(y))}{\chi_N}. \end{split}\end{equation}Substituting everything back into (\ref{eq: start GN}) and simplifying,  \begin{equation}\label{eq: sum gnxy}
	\cG_N\left(\frac\alpha 2\cH\right)(\eta^N)=\frac{d}{N^{d-2}}\sum_{x, y\in \TND} p^N(x,y)g^N_{x,y}[\eta^N]\end{equation} where we define $g^N_{x,y}[\eta^N]$ to be the contribution from jumps from $x$ to $y$:\begin{equation} \label{eq: gnxy} g^N_{x,y}:=g^N_{x,-}g^N_{y,+}-(\eta^N(x))^\alpha \end{equation} where we factorise the gain term as \begin{equation}
		g^N_{x,-}:=(\eta^N(x))^\alpha \exp\left(\frac{\alpha}{2}\left(\frac{\Psi(\eta^N(x)-\chi_N)-\Psi(\eta^N(x))}{\chi_N}\right)\right)
	\end{equation} which we understand to be $0$ whenever $\eta^N(x)=0$, and \begin{equation}
		g^N_{y,+}:= \exp\left(\frac{\alpha}{2}\left(\frac{\Psi(\eta^N(y)+\chi_N)-\Psi(\eta^N(y))}{\chi_N}\right)\right).
	\end{equation}  The first goal is to prove that, for some absolute constant $C$, \begin{equation} \label{eq: claim for QY}\begin{split} g^N_{x,y}\le &  (\eta^N(x))^{\alpha/2}\left((\eta^N(y))^{\alpha/2}-(\eta^N(x))^{\alpha/2}\right) \\& \hspace{2cm}+ C\chi_N^{\min(\alpha/2,1)}(1+(\eta^N(x))^\alpha+(\eta^N(y))^\alpha).\end{split}\end{equation} This will be the goal of steps 1-3. Once we have obtained this, we will post-process the inequality into the claim (\ref{eq: central estimate}) in step 4. \paragraph{\textbf{Step 1. Analysis of $g^N_{x,-}$.}} Recalling that $\eta^N(x)$ is always a multiple of $\chi_N$, we split the argument into cases. If $\eta^N(x)=0$, then (\ref{eq: claim for QY}) is immediate and there is nothing to prove. In the case $\eta^N(x)\ge 2\chi_N$, we write \begin{equation} \begin{split} \label{eq: exponential change x}
		\exp\left(\frac\alpha 2\left(\frac{\Psi(\eta^N(x)-\chi_N)-\Psi(\eta^N(x))}{\chi_N}\right)\right)&=\frac{(\eta^N(x)-\chi_N)^{\alpha(\eta^N(x)-\chi_N)/2\chi_N}}{(\eta^N(x))^{\alpha\eta^N(x)/2\chi_N}}\\& =(\eta^N(x))^{-\alpha/2}\left(1-\frac{\chi_N}{\eta^N(x)}\right)^{\frac\alpha 2\left(\frac{\eta^N(x)}{\chi_N}-1\right)}. \end{split}
	\end{equation} Still assuming $\eta^N(x)\ge 2\chi_N$, we rewrite the second factor as \begin{equation}\begin{split}
		\left(1-\frac{\chi_N}{\eta^N(x)}\right)^{\frac\alpha 2\left(\frac{\eta^N(x)}{\chi_N}-1\right)} & =\exp\left(\frac\alpha 2\left(\frac{\eta^N(x)}{\chi_N}-1\right)\ln\left(1-\frac{\chi_N}{\eta^N(x)}\right)\right). \end{split}
	\end{equation} By Taylor expanding $z\mapsto \ln (1-z), z\in [0,\frac{1}{2}]$ about $z=0$, we write \begin{equation} \begin{split}
		\left(\frac{\eta^N(x)}{\chi_N}-1\right)\ln\left(1-\frac{\chi_N}{\eta^N(x)}\right) &=-1+\frac{\lambda^N_{x,-}\chi_N}{\eta^N(x)}\end{split} 
	\end{equation} for some $\lambda^N_{x,-}$ enjoying a bound $0\le \lambda^N_{x,-}\le C$ uniform in $N, x, \eta^N$. It follows that, for $\eta^N(x)\ge 2\chi_N$, \begin{equation}\label{eq: bound gnxminus below} \begin{split} 
		g^N_{x,-}-e^{-\alpha/2}(\eta^N(x))^{\alpha/2}&= (\eta^N(x))^{\alpha/2}e^{-\alpha/2}\left(\exp\left(\frac{\chi_N \lambda^N_{x,-}}{\eta^N(x)}\right)-1\right) \\ & =: r^N_{x,-}e^{-\alpha/2}.\end{split}
	\end{equation} The same conclusion (\ref{eq: bound gnxminus below}) also holds for the case $\eta^N(x)=\chi_N$, now taking $\lambda^N_{x,-}=\frac\alpha2$. Finally, for suitable constants $C$, we can bound the remainder above in the case $\alpha\in [1,2)$ by\begin{equation}\label{eq: rnx1}
		r^N_{x,-} \le \chi_N^{\alpha/2}C\sup_{t\ge 1}\left(t^{\alpha/2}(e^{t^{-1}}-1)\right)=C\chi_N^{\alpha/2} 
	\end{equation} and if $\alpha\ge 2$ then \begin{equation} \label{eq: rnx2}
		r^N_{x,-}\le C\chi_N(\eta^N(x))^{\frac\alpha2-1}	.\end{equation}\paragraph{\textbf{Step 2. Analysis of $g^N_{y,+}$.}} We follow the same argument for $g^N_{y,+}$, distinguishing between the cases $\eta^N(y)\ge \chi_N, \eta^N(y)=0$. For the case $\eta^N(y)\ge \chi_N$, the estimate $\ln(1+x)\le x$ produces \begin{equation} \label{eq: exponential change y}\begin{split} g^N_{y,+}&=(\eta^N(y))^{\alpha/2}\exp\left(\frac{\alpha}{2}\left(\frac{\eta^N(y)}{\chi_N}+1\right)\log\left(1+\frac{\chi_N}{\eta^N(y)}\right)\right) \\ & \le (\eta^N(y))^{\alpha/2} e^{\alpha/2}\exp\left(\frac{\alpha}{2}\frac{\chi_N}{\eta^N(y)}\right) =: e^{\alpha/2}\left((\eta^N(y))^{\alpha/2}+r^N_{y,+}\right).
		\end{split}	\end{equation} As in step 1, if $\alpha\in [1,2)$ we bound the remainder by \begin{equation}\label{eq: rny1} r^N_{y,+}\le C\chi_N^{\alpha/2}\sup_{t\ge 1}\left(t^{\alpha/2}(e^{t^{-1}}-1)\right)\le C\chi_N^{\alpha/2} \le C\chi_N^{\alpha/2}\end{equation} while for $\alpha\ge 2$\begin{equation}\label{eq: rny2}
		r^N_{y,+}\le C\chi_N(1+(\eta^N(y))^{\frac\alpha2-1}).
	\end{equation} For the excluded case $\eta^N(y)=0$, we have the direct computation \begin{equation}
	r^N_{y,+}=e^{-\alpha/2}g^N_{y,+}=e^{-\alpha/2}\chi_N^{\alpha/2}	\end{equation} so that  (\ref{eq: rny1}, \ref{eq: rny2}) apply unconditionally, up to increasing the constant $C$. \paragraph{\textbf{Step 3. Reassembly}} We now multiply the bounds found in the previous steps. In either case $\{1\le \alpha<2\}$ or $\{\alpha\ge 2\}$, we find \begin{equation}
		\begin{split}
			g^N_{x,y}&=e^{-\alpha/2}((\eta^N(x))^{\alpha/2}+r^N_{x,-}))\times e^{\alpha/2}((\eta^N(y))^{\alpha/2}+r^N_{y_+}))-(\eta^N(x))^\alpha \\ & =:(\eta^N(x))^{\alpha/2}\left((\eta^N(y))^{\alpha/2}-(\eta^N(x))^{\alpha/2}\right)+r^N_{x,y}
		\end{split} 
	\end{equation} where we consolidate all terms involving at least one of $r^N_{x,-}, r^N_{y,+}$ into a single error $$r^N_{x,y}=(\eta^N(y))^{\alpha/2}r^N_{x,-}+(\eta^N(x))^{\alpha/2}r^N_{y-+}+r^N_{x,-}r^N_{y,+}$$ and the claim (\ref{eq: claim for QY}) now reduces to finding a suitable bound for $r^N_{x,y}$. For $1\le \alpha<2$, we gather (\ref{eq: rnx1}, \ref{eq: rny1}) to find \begin{equation}\begin{split}
		r^N_{x,y}&\le C\chi_N^{\alpha/2}\left((\eta^N(y))^{\alpha/2}+(\eta^N(x))^{\alpha/2}+C\chi_N^{\alpha/2}\right) \\ & \le C\chi_N^{\alpha/2}\left(1+(\eta^N(x))^{\alpha}+(\eta^N(y))^\alpha\right). \end{split}
	\end{equation} In the case $\alpha\ge 2$, we combine (\ref{eq: rnx2}, \ref{eq: rny2}) to obtain \begin{equation} \begin{split} r^N_{x,y}&\le C\chi_N\left((\eta^N(x)^{\frac\alpha2-1}(\eta^N(y))^\frac\alpha2+\left(1+(\eta^N(y))^{\frac\alpha2-1}\right) \eta^N(x))^\frac\alpha2\right) \\ & \hspace{1.5cm}+C^2\chi_N^2\left((\eta^N(x)^{\frac\alpha2-1}\left(1+(\eta^N(y))^{\frac{\alpha}{2}-1}\right)\right)\\ & \le C\chi_N\left((1+((\eta^N(x))^{\alpha}+(\eta^N(y))^\alpha\right)\end{split} \end{equation} possibly up to a new choice of $C$. In either case the claim (\ref{eq: claim for QY}) is proven. \paragraph{\textbf{Step 4. Summing over $x,y$.}} We now substitute (\ref{eq: claim for QY}) into (\ref{eq: sum gnxy}) and perform the sum over $x,y$. In the main term, we recall that $p^N(x,y)=p^N(y,x)$ and symmetrise to obtain \begin{equation} \begin{split}
		&  \frac{d}{N^{d-2}}\sum_{x,y\in \TND} p^N(x,y)(\eta^N(x))^{\alpha/2}((\eta^N(y))^{\alpha/2}-\eta^N(x))^{\alpha/2})  \\ &\hspace{3cm} =-\frac{d}{2N^{d-2}}\sum_{x,y\in \TND} p^N(x,y)\left((\eta^N(x))^{\alpha/2}-(\eta^N(y))^{\alpha/2}\right)^2 \\[1ex]&\hspace{3cm} =:-\frac{\alpha}{2}\cD_{\alpha,N}(\eta^N). \end{split} \end{equation} Meanwhile, the sum in the error term produces \begin{equation}\begin{split}
			&\frac{dC\chi_N^{\min(1,\alpha/2)}}{N^{d-2}}\sum_{x,y\in \TND} p^N(x,y)(1+(\eta^N(x))^\alpha+ (\eta^N(y))^\alpha) \\ &\hspace{3.5cm} \le C\chi_N^{\min(1,\alpha/2)}N^2\left(1+\left\|\eta^N\right\|_{L^\alpha(\TND)}^\alpha\right)\end{split}
		\end{equation} and combining everything produces (\ref{eq: central estimate}). \paragraph{\textbf{Step 5. Interpolation.}} We now post-process (\ref{eq: central estimate}) into the more useful form (\ref{eq: central estimate 2}) under the scaling relation (\ref{eq: scaling hypothesis}) on $N, \chi_N$. We first absorb the factor $N^2\chi_N^{\max(\alpha/2,1)}$, which we assume to be bounded in $N$, into the factor $C$. We next use Lemma \ref{prop: interpolation} to write, for some $\beta>\alpha$, $\eta^N\in X_{N,a}$ and absolute constants $C$ varying from line to line, \begin{equation}\begin{split}
			 C\left(1+\left\|\eta^N\right\|_{L^\alpha(\TND)}^\alpha\right) &\le C\left(1+\left\|\eta^N\right\|_{L^\beta(\TND)}^\beta\right)^{\alpha/\beta} \\& \le C\left((1+a)\left[1+\cD_{\alpha,N}(\eta^N)\right]\right)^{\alpha/\beta}.
	\end{split}	\end{equation} For any $\vartheta\in(0,\frac\alpha2)$, we use the fact that $0<\frac\alpha\beta<1$ to choose $C=C(\vartheta, a,\alpha)$  such that the previous expression is at most $$ \vartheta\cD_{\alpha,N}(\eta^N)+C $$ and the proof is complete.\end{proof}

We thus find the following estimate. \begin{lemma}
			\label{lemma: QY1} Suppose that (\ref{eq: scaling hypothesis}) holds, and suppose the initial data $\eta^N_0$ satisfy (\ref{eq: entropy finite hyp}). Then \begin{equation}
				\label{eq: QY1.1} \limsup_{M\to \infty} \limsup_N \frac{\chi_N}{N^d}\log \PP\left(\cF_{\alpha,N}(\eta^N_\bullet) > M \right)=-\infty
			\end{equation} where $\cF_{\alpha,N}$ is the functional (\ref{eq: FNA}). Furthermore, for $\beta>\alpha$ as in Lemma \ref{prop: interpolation}, \begin{equation}\label{eq: LD integrability}
				\limsup_{M\to \infty}\limsup_N\frac{\chi_N}{N^d}\log \PP\left(\int_0^\tf \left\|\eta^N_t\right\|_{L^\beta(\TTd)}^\beta dt > M\right)=-\infty 
			\end{equation}and for all $\lambda>0$ and $a<\infty$, \begin{equation} \label{eq: exponential gaussian moments} \limsup_N \frac{\chi_N}{N^d}\log \EE\left[ \indiq[\eta^N_0\in X_{N,a}]\exp\left(\frac{N^d\lambda}{\chi_N}\int_0^\tf \left\|\eta^N_t\right\|_{L^\alpha(\TTd)}^\alpha dt \right)\right]<\infty.\end{equation}  
		\end{lemma}  \begin{proof} We deal with the two items separately. \paragraph{\textbf{Step 1. Entropy and Entropy Dissipation. }}Let us fix $z<\infty$, and fix $M>0$ to be chosen later. Thanks to (\ref{eq: bound mass hyp}), we can find $a<\infty$ such that, for all but finitely many $N$, \begin{equation}\label{eq: choice of a}
			\PP\left(\eta^N_0\not \in X^a_N\right)\le \exp\left(-\frac{N^d}{\chi_N} z\right). 
		\end{equation} By the definition \eqref{eq: nonlinear generator}, \begin{equation}\label{eq: entropy mg}
			 \mathcal{Z}^N_t:= \indiq\left[\eta^N_0\not \in X_{N,a} \right]\exp\left(\frac{N^d}{\chi_N}\left(\frac\alpha2\cH(\eta^N_t)-\int_0^t\cG_N\left(\frac\alpha2\cH\right)(\eta^N_s)ds\right)\right)
		\end{equation} is a nonnegative local martingale, and hence a supermartingale. Using the final part of Lemma \ref{lemma: computation of entropy dissipation}, we can find $C$ depending on $a$ such that \begin{equation} \label{eq: lower bound a priori mg}
			\mathcal{Z}^N_t\ge \exp\left(\frac{N^d}{\chi_N}\left(\frac\alpha2\cH(\eta^N_t)+\frac\alpha4\int_0^t \cD_{\alpha,N}(\eta^N_s)ds-Ct\right)\right)\indiq[\eta^N_0\in X_{N,a}]
		\end{equation} where we have used the fact that if $\eta^N_0 \in X_{N,a}$ then $\eta^N_t\in X_{N,a}$ almost surely for all $t\ge 0$ by conservation of mass, and so \begin{equation} \label{eq: break up prob} \begin{split} & \cF_{\alpha, N}(\eta^N_\bullet)\indiq[\eta^N_0\in X_{N,a}] \le \frac{8}{\alpha}\left(\frac{\chi_N}{N^d}\log \sup_{t\le \tf} \mathcal{Z}^N_t + C\tf\right).		\end{split}\end{equation} Using Doob's inequality, for any $M$, \begin{equation} \begin{split} \label{eq: chebychev est}
		& \PP\left(\cF_{\alpha,N}(\eta^N_\bullet)>M, \eta^N_0\in X_{N,a}\right) \\&\hspace{2cm}  \le \PP\left(\sup_{t\le \tf} \mathcal{Z}^N_t > \exp\left\{\frac{N^d}{\chi_N}\left(\frac{\alpha M}{8}-C\tf\right)\right\}\right) \\ & \hspace{2cm} \le \exp\left(-\frac{N^d}{\chi_N}\left(\frac{\alpha M}{8}-C\tf \right)\right)\mathbb{E}\left[\mathcal{Z}^N_0\right]. 
		\end{split} \end{equation}Thanks to the hypothesis (\ref{eq: entropy finite hyp}), it follows that \begin{equation}\label{eq: ZN0 controlled} \limsup_N \frac{\chi_N}{N^d}\log \EE\left[\mathcal{Z}^N_0\right]<\infty. \end{equation} and we may now choose $M<\infty$ large enough to make the right-hand side of (\ref{eq: chebychev est}) at most $e^{-N^d z/\chi_N}$ for all $N$. Using the choice of $a$, it follows that $\PP(\cF_{\alpha, N}(\eta^N_\bullet)>M)\le 2e^{-N^d z/\chi_N}$ for all but finitely many $N$, and since $z$ was arbitrary, the proof of the first item is complete. \paragraph{\textbf{Step 2. Integrability Estimates}} We will now deduce the integrability estimates from the previous part. We obtain both conclusions by returning to the nonnegative supermartingale $\mathcal{Z}^N_t$ defined in (\ref{eq: entropy mg}). On the event $\{\eta^N_0\in X_{N,a}\}$, Lemma \ref{prop: interpolation} implies that \begin{equation}
			\int_0^\tf \cD_{N,\alpha}(\eta^N_s)ds \ge C\left(\int_0^\tf \left\|\eta^N_t\right\|_{L^\beta(\TTd)}^\beta dt -1\right)
		\end{equation} for some constant $C=C(a)$ allowed to depend on $a$. It follows from (\ref{eq: lower bound a priori mg}) that, for some $c=c(a)>0$ and $C=C(a)<\infty$, \begin{equation}\label{eq: put integration into mg}
			\indiq[\eta^N_0\in X_{N,a}]\exp\left(\frac{cN^d}{\chi_N}\int_0^\tf \left\|\eta^N_t\right\|_{L^\beta(\TTd)}^\beta dt\right) \le \mathcal{Z}^N_\tf e^{CN^d(1+\tf)/\chi_N}.
		\end{equation}The first point (\ref{eq: LD integrability}) follows immediately by breaking up the probability as above, using the same Chebychev estimate (\ref{eq: chebychev est}) and (\ref{eq: ZN0 controlled}) to estimate $$\limsup_N \frac{\chi_N}{N^d}\log \PP\left(\int_0^\tf \left\|\eta^N_t\right\|_{L^\beta(\TTd)}^\beta dt >M, \eta^N_0\in X_{N,a}\right) $$ and finally choosing $a$ suitably large. For (\ref{eq: exponential gaussian moments}), fix $a<\infty$ and observe that by Young's inequality, for any $\lambda<\infty$, there exists $M=M(a, \lambda)<\infty$ such that \begin{equation}\lambda \int_0^\tf \|\eta^N_t\|_{L^\alpha(\TTd)}^\alpha dt \le  M+c\int_0^\tf \left\|\eta^N_t\right\|_{L^\beta(\TTd)}^\beta dt \end{equation} for the same constant $c$ as in (\ref{eq: put integration into mg}), whence \begin{equation}
			\indiq[\eta^N_0\in X_{N,a}]\exp\left(\frac{N^d\lambda}{\chi_N}\int_0^\tf \left\|\eta^N_t\right\|_{L^\alpha(\TTd)}^\alpha dt\right) \le e^{(M+C(1+\tf))N^d/\chi_N}\mathcal{Z}^N_{\tf}
		\end{equation} and the conclusion follows from (\ref{eq: ZN0 controlled}). \end{proof}
		
	We conclude this section by using the \emph{a priori} estimate  to prove exponential tightness, and in particular an equicontinuity estimate with exponentially small failure probability. \begin{lemma} \label{lemma: ET} Under the assumptions of the previous lemma: \begin{enumerate}[label=\roman*).] 
			\item The processes $\eta^N_\bullet$ are exponentially tight in $\mathbb{D}$: for any $z<\infty$, there exists a compact set $\mathcal{K} \subset \mathbb{D}$ such that \begin{equation}
				\limsup_N \frac{\chi_N}{N^d}\log \PP\left(\eta^N_\bullet \not \in \mathcal{K}\right)\le -z. 
			\end{equation} \item For any $u_\bullet \not \in \mathcal{C}$ and any $z<\infty$, there exists a $\cDD$-open set $\cU \ni u_\bullet$ such that \begin{equation}
			\label{eq: continuity} \limsup_N \frac{\chi_N}{N^d}\log \PP\left(\eta^N_\bullet \in \cU \right)\le -z.			\end{equation} 
		\end{enumerate} \end{lemma}  \begin{proof}[Proof of Lemma \ref{lemma: ET}] Let us define the sets \begin{equation}
			\label{eq: compact sets} K_{M}:=\left\{u\in L^1(\TTd,[0,\infty)): \cH(u)\le M\right\}.
		\end{equation} By the Dunford-Pettis Theorem \cite[Theorem 4.7.18]{bogachev2007measure}, $K_M$ are compact in the weak topology of $L^1(\TTd)$, and hence in the metric $d$, and we proved in Lemma \ref{lemma: QY1} that \begin{equation}
			\limsup_{M\to \infty} \limsup_N \frac{\chi_N}{N^d} \log \PP\left(\exists t\le \tf: \eta^N_t\not \in K_M\right)=-\infty.
		\end{equation} We will now prove, for every $\varphi \in C^2(\TTd)$, that  \begin{equation}
			\label{eq: verification of ET 2'} \limsup_{\delta\downarrow 0} \limsup_N \frac{\chi_N}{N^d}\log \PP\left(q^N_\varphi(\delta)>\e\right)=-\infty
		\end{equation} where $q^N_\varphi(\delta)$ is given by \begin{equation} q^N_\varphi(\delta):=\sup\left\{\langle \varphi, \eta^N_t-\eta^N_s\rangle: s, t\le \tf, |s-t|\le \delta\right\}.\end{equation} Point i) will then follow from general considerations, see Feng and Kurtz \cite[Theorems 4.1, 4.4]{feng2006large}. Using Lemma \ref{lemma: computation of generators}, there exists a finite constant $C=C(\varphi)$ such that, for all $N$, the functional $F^\varphi_N(\eta^N):=\langle \varphi, \eta^N\rangle$ satisfies  \begin{equation}
			\left|\cG_N F_N^\varphi\right|\le C(\varphi)\left\|\eta^N\right\|_{L^\alpha(\TTd)}^\alpha. 
		\end{equation} Let us fix $z<\infty$, $\e>0$ and $\varphi\in C^2(\TTd)$, and choose $\lambda>0$ such that $\frac{\lambda \e}{6}>z$. By Lemma \ref{lemma: QY1}, we can find $\beta>\alpha$ and $M\in [1,\infty)$ such that \begin{equation}
			\label{eq: choice of M} \limsup_N \frac{\chi_N}{N^d}\log\PP\left(\int_0^\tf \left\|\eta^N_s\right\|_{L^\beta(\TTd)}^\beta ds > M\right)\le -z. 
		\end{equation}On the complement of this event, it follows that for all $0\le s\le t\le \tf$, \begin{equation}
			\begin{split}
				\left|\int_s^t \cG_N F^{\lambda\varphi}_N(\eta^N_u) du\right| &\le C(\lambda \varphi)\int_s^t\left\|\eta^N_u\right\|_{L^\alpha(\TTd)}^\alpha du \\ & \le C(\lambda \varphi)\left(\int_s^t \left\|\eta^N_u\right\|_{L^\beta(\TTd)}^\beta\right)^{\alpha/\beta}(t-s)^\gamma \\ & \le C(\lambda \varphi)M(t-s)^\gamma,
			\end{split}
		\end{equation} for $\gamma=\frac{\beta-\alpha}{\beta}>0$. We now choose $\delta>0$ such that  $			C(\lambda \varphi)M\delta^\gamma < \frac{\lambda \e}{6}$ so that, on a fixed interval $[t_0, t_1]$ of length at most $\delta$, we bound \begin{equation}\begin{split}
			\PP\left(\sup_{s\in [t_0, t_1]} \langle \varphi, \eta^N_s-\eta^N_{t_0}\rangle > \frac{\e}{3}\right) & \le \PP\left( \sup_{s\in [t_0, t_1]} \left(\langle \lambda \varphi, \eta^N_s-\eta^N_{t_0}\rangle-\int_{t_0}^s \cG_N F^{\lambda \varphi}_N(\eta^N_u) du\right)> \frac{\lambda \e}{6}\right) \\& + \PP\left(\int_0^\tf \left\|\eta^N_t\right\|_{L^\beta(\TTd)}^\beta dt >M\right). \end{split} 
		\end{equation} For the first term, we observe that $$ \mathcal{Z}^{N, \lambda\varphi}_{t_0,s}:=\exp\left(\frac{N^d}{\chi_N} \left(\langle \lambda \varphi, \eta^N_s-\eta^N_{t_0}\rangle-\int_{t_0}^s \cG_N F^{\lambda \varphi}_N(\eta^N_u) du\right)\right), \hspace{0.5cm} s\ge t_0 $$ is a c{\`a}dl{\`a}g martingale with $\mathcal{Z}^{N, \lambda \varphi}_{t_0, t_0}=1$. By Doob's maximal inequality and a Markov inequality, we find \begin{equation}
			\begin{split} 
				& \PP\left( \sup_{s\in [t_0, t_1]} 
				\left(\langle \lambda \varphi, \eta^N_s-\eta^N_{t_0}\rangle-\int_{t_0}^s \cG_N F^{\lambda \varphi}_N(\eta^N_u) du\right)>\frac{\lambda \e}{6}\right) \\&\hspace{3cm} = \PP\left(\sup_{s\in [t_0,t_1]} \mathcal{Z}^{N,\lambda\varphi}_{t_0,s}>e^{N^d \lambda \epsilon/ 6\chi_N}\right)  \le e^{-N^d \lambda \epsilon/ 6\chi_N}.
			\end{split}
		\end{equation} By the choice of $\lambda$, the upper bound is at most $e^{-N^d z/\chi_N}$, and by the choice of $M$ (\ref{eq: choice of M}) to control the second term, we conclude that \begin{equation} \label{eq: equicontinuity conclusion}
			\limsup_N \frac{\chi_N}{N^d}\log \PP\left(\sup_{s\in [t_0, t_1]} \langle \varphi, \eta^N_s-\eta^N_{t_0}\rangle > \frac{\e}{3}\right) \le -z. 
		\end{equation} We now cover $[0,\tf]$ with $\lceil \tf /\delta\rceil$ such intervals $[t_i, t_{i+1}]$, $0\le i< \lceil \tf /\delta\rceil$, and (\ref{eq: equicontinuity conclusion}) applies to each interval, so\begin{equation}
			\limsup_N \frac{\chi_N}{N^d}\log \PP\left(\max_{i\le \lceil \tf/\delta\rceil} \sup_{s\in [t_i, t_{i+1}]} \langle \varphi, \eta^N_s-\eta^N_{t_0}\rangle > \frac{\e}{3}\right) \le -z. 
		\end{equation} On the compliment of this event, $q^N_\varphi(\delta)\le \epsilon$, and so we conclude (\ref{eq: verification of ET 2'}). The proof of item i). is therefore complete. \bigskip \\ The second item can be verified in the same way. If $u_\bullet \not \in \mathcal{C}$, then there exists $\varphi\in C^2(\TTd)$ and $\e>0$ such that, for all $\delta>0$, $$\sup\{|\langle \varphi, u_s-u_t\rangle|: s, t\le \tf, |s-t|<\delta\}>\e.$$ Using the definition of the Skorokhod topology, for all $\delta>0$ there is a $\cDD$-open set $\cU(\delta)\ni u_\bullet$ such that, for all $v_\bullet \in \cU(\delta)$, \begin{equation}
			\sup\left\{|\langle \varphi, v_s-v_t\rangle|: s,t\le \tf, |s-t|<3\delta\right\}>\frac{\e}{3}
		\end{equation} whence \begin{equation}\begin{split}
			&\limsup_N \frac{\chi_N}{N^d}\log \PP\left(\eta^N_\bullet \in \cU(\delta)\right) \\& \hspace{2cm} \le \limsup_N \frac{\chi_N}{N^d}\log \PP\left(
		\exists s,t\le \tf: |s-t|<3\delta, |\langle \varphi, \eta^N_s-\eta^N_t\rangle|>\frac{\e}{3} \right). \end{split}
		\end{equation} Given $z<\infty$, the argument of item i) shows that we may choose $\delta>0$ such that the right-hand side is at most $-z$, and the desired conclusion (\ref{eq: continuity}) holds for $\cU=\cU(\delta)$.  \end{proof} 

	 \section{Large Deviations Aubin-Lions-Simon Theorem} \label{sec: WTS}

We now assemble Lemma \ref{lemma: DWtS} and the probabilistic considerations in Section \ref{sec: a priori} to prove Theorem \ref{thrm: WtS}.

	\begin{proof}[Proof of Theorem \ref{thrm: WtS}] Let us begin with item i). We deal with the cases $u_\bullet \in \cR, u_\bullet \not \in \cR$ separately. \paragraph{\textbf{Case 1. $u_\bullet\in \cR$.}} Fix $u_\bullet$, $\cV$ and $z$ as in the theorem. Thanks to Lemma \ref{lemma: QY1}, there exists $\lambda<\infty$ such that \begin{equation}\label{eq: choice of lambda} \limsup_N \frac{\chi_N}{N^d} \log \PP\left(\cF_{\alpha, N}(\eta^N_\bullet)>\lambda\right)\le -z. \end{equation} 
	With this choice of $\lambda$, we choose $\cU\ni u_\bullet$ open in the topology of $\cDD$ according to Lemma \ref{lemma: DWtS}i), so that (\ref{eq: ALS claim 1}) gives, for all sufficiently large $N$, \begin{equation} \label{eq: ALS claim 1 repeat} \cU\cap \{\eta^N_\bullet \in \cDD_N: \cF_{\alpha,N}(\eta^N_\bullet)\le \lambda\} \subset \cV. \end{equation} Taking the complement yields, for sufficiently large $N$, \begin{equation}(\cU\setminus \cV)\cap \mathbb{D}_N\subset \left\{\eta^N_\bullet\in \mathbb{D}_N: \cF_{\alpha, N}(\eta^N_\bullet)> \lambda\right\}\end{equation} 
	and by the choice of $\lambda$ in (\ref{eq: choice of lambda}), \begin{equation} \limsup_N \frac{\chi_N}{N^d}\log \PP\left(\eta^N_\bullet \in \cU\setminus \cV\right)\le -z. \end{equation}\paragraph{\textbf{Case 2. $u_\bullet \not \in \cR$.}} For the case $u_\bullet \not \in \cR$, we must further subdivide into cases to find, for all $z<\infty$, $\cU\ni u_\bullet$, open in the topology of $\cDD$, such that \begin{equation} \limsup_N \frac{\chi_N}{N^d}\log \PP\left(\eta^N_\bullet \in \cU\right)\le -z. \end{equation} We observe that $$\cDD\setminus \cR \subset (\cDD\setminus \mathcal{C})\cup \{\cF_\alpha(u_\bullet)=\infty\}.$$ In the case $\cF_\alpha(u_\bullet)=\infty$, we argue as above by picking $\lambda$ according to (\ref{eq: choice of lambda}) and then $\cU$ according to Lemma \ref{lemma: DWtS}ii), so that (\ref{eq: ALS claim 2}) produces \begin{equation}
		\limsup_N \frac{\chi_N}{N^d}\log \PP\left(\eta^N_\bullet \in \cU\right)\le -z.
	\end{equation} If $u_\bullet \not \in \mathcal{C}$, the existence of such a $\cU$ was already established by the second part of Lemma \ref{lemma: ET} and the proof of item i) is complete. \bigskip \\ We now turn to item ii), which allows us to gain almost sure convergence in $L^\alpha([0,\tf]\times\TTd)$ from convergence in $\cDD$, after a change of measures. Using convex duality for functions $\theta(e^{x/\theta}-1)$ and its dual $\theta(y\log y-y+1)$ yields, for $A_N\in \sigma(\eta^N_\bullet)$ and any $z>0$, \begin{equation}
		\label{eq: orlicz ineq} \mathbb{Q}(A_N)\le \frac{\chi_N}{N^d z} H\left(\Law_{\QQ}[\eta^N_\bullet]\right|\left.\Law_{\PP}[\eta^N_\bullet]\right) + \left(e^{N^d z/\chi_N}-1\right)\mathbb{P}(A_N).
	\end{equation} Using the hypothesis (\ref{eq: entropy bound hypothesis}) to control the first term, we may thus transfer Lemma \ref{lemma: QY1} to $\mathbb{Q}$ to find  \begin{equation}
		\label{eq: superexponential estimate QN} \limsup_{M\to\infty} \limsup_{N\to\infty}  \QQ\left(\cF_{\alpha,N}(\eta^N_\bullet)>M\right)=0. 
	\end{equation} Recalling that $\eta_\bullet$ is the $\QQ$-almost sure limit in the topology of $\cDD$ on a subsequence $N_k\to \infty$, we now argue that $\QQ(\eta_\bullet \in \cR)=1$. To see this, it follows from (\ref{eq: superexponential estimate QN}) that there exists an increasing function $\vartheta: [0,\infty)\to [0,\infty)$, with $\vartheta(x)\to \infty$ as $x\to \infty$, such that $$\sup_N \EE_{\QQ}\left[\vartheta\left(\cF_{\alpha,N}(\eta^N_\bullet)\right)\right]<\infty. $$ Extending $\vartheta$ by $\vartheta(\infty):=\infty$, it follows from monotonicity and (\ref{eq: Gamma lower bound}) in Lemma \ref{lemma: DWtS} that, $\QQ$-almost surely, $$ \vartheta\left(\cF_\alpha(\eta_\bullet)\right)\le \liminf_k \vartheta\left(\cF_{\alpha,N_k}(\eta^{N_k}_\bullet)\right) $$ and the right-hand side has finite expectation by Fatou. It follows that the left-hand side is almost surely finite, whence $\QQ(\cF_\alpha(\eta_\bullet)<\infty)=1$ and the claim follows. \medskip \\ We now prove convergence in probability in $L^\alpha([0,\tf]\times\TTd)$ on the subsequence $N_k$.  Fix $\e, \d>0$, and choose $M,a$ such that \begin{equation}
		\limsup_N \QQ\left(\cF_{\alpha,N}(\eta^N_\bullet)>M\right)<\d; 
	\qquad	\limsup_N \QQ\left(\eta^N_0\not \in X_{N,a}\right)<\d. 
	\end{equation} Thanks to Lemma \ref{lemma: nonstandard block convolution via dirichlet}, we may now choose $r\in(0,1)$ small enough that, for the constant $C$ in (\ref{eq: nonstandard block convolution via dirichlet 1}), $C(1+a)r^\gamma(\tf+M)<\e^\alpha$, so that \begin{equation}
		\limsup_N \QQ\left(\left\|\eta^N_\bullet -\overline{\eta}^{N,r}_\bullet\right\|_{L^\alpha([0,\tf]\times\TTd)}> \e\right)<2\delta.
	\end{equation} By Remark \ref{rmk: continnum nonstandard convolution}, $\overline{\eta}^r \to \eta$ in $L^\alpha([0,\tf]\times\TTd)$, $\QQ$-almost surely, so by making $r>0$ smaller if necessary, we can further arrange that \begin{equation}
		\QQ(\|\overline{\eta}^r_\bullet-\eta_\bullet\|_{L^\alpha([0,\tf]\times\TTd)}>\e)<\delta. 
	\end{equation} By Lemma \ref{lemma: convergence of convolution}, it holds $\QQ$-almost surely that $\overline{\eta}^{N_k,r}_\bullet \to \overline{\eta}^r_\bullet$ in $L^\infty([0,\tf]\times\TTd)$, and hence also in $L^\alpha([0,\tf]\times\TTd)$. Combining everything, it follows that $$ \limsup_{N_k} \mathbb{Q}\left(\left\|\eta^{N_k}_\bullet-\eta_\bullet\right\|_{L^\alpha([0,\tf]\times\TTd)}>3\e\right)<4\delta$$ and we conclude the convergence $\eta^{N_k}_\bullet \to \eta_\bullet$ in $\QQ$-probability in the topology of $L^\alpha([0,\tf]\times\TTd)$. The conclusion of $\QQ$-almost sure convergence on a subsequence is elementary.\end{proof} 
\section{Hydrodynamic Limit}\label{sec: hydrodynamic}

We now prove Theorem \ref{th: hydrodynamic limit}, based on the general tools we have already developed. We fix $u_0 \in L^1_{\ge 0}(\TTd)$ and a probability measure $\PP$ on the underlying probability space as in the statement of the theorem, and let $u_\bullet$ be the unique solution to the PME (\ref{eq: PME}) starting at $u_0$.
\begin{proof}[Proof of Theorem \ref{th: hydrodynamic limit}] Let $\PP$ be as given, and for all $M, a$, let us define new measures $\PP^N_{M,a}$ by conditioning on the event \begin{equation}\label{eq: conditioning for hydrodynamic limit}
	A^N_{M,a}:=\left\{\cH(\eta^N_0)\le M, \langle 1, \eta^N_0\rangle \le a\right\}.
\end{equation} We will now prove that, under these conditioned measures, $\eta^N_\bullet\to u_\bullet$ in the topology of $\mathbb{D}$ in $\PP$-probability, using the well-known technique of proving compactness and identifying all possible limit points. Finally, we will remove the conditioning, to recover statements in terms of $\PP$. For now, let us fix $M,a$. \paragraph{\textbf{Step 1. Tightness.}} For each $M,a$, the measures $\PP^N_{M,a}$ trivially satisfy the hypothesis (\ref{eq: entropy finite hyp}), and so Lemma \ref{lemma: ET} shows that $\Law_{\PP^N_{M,a}}[\eta^N_\bullet]$ are exponentially tight, and hence tight, relative to the topology of $\mathbb{D}$. \bigskip \\ \paragraph{\textbf{Step 2. Identification of Limits.}} Let us now prove that, if a subsequence $\eta^{N_k}_\bullet$ converges in distribution under $\PP^N_{M,a}$ to $\eta_\bullet$ relative to the topology of $\mathbb{D}$, then $\eta_\bullet$ takes values in the set $\cS(u_0)$ of solutions to the PME (\ref{eq: PME}) starting at $u_0$. Since we know from Proposition \ref{prop: uniqueness FP} that $\cS(u_0)=\{u_\bullet\}$ is a singleton, we will then conclude that $$\Law_{\PP^N_{M,a}}[\eta^N_\bullet]\to \delta_{u_\bullet}$$ weakly as measures on $\mathbb{D}$. \bigskip \\ Thanks to Skorokhod's representation theorem, we may realise all $\eta^{N_k}_\bullet, \eta_\bullet$ on a common probability space $\QQ_{M,a}$ such that $\eta^N_\bullet\to \eta_\bullet$ in the topology of $\mathbb{D}$, $\QQ_{M,a}$-almost surely. By Theorem \ref{thrm: WtS}ii), $\QQ_{M,a}(\eta_\bullet \in \cR)=1$ and, passing to a further subsequence, $\eta^{N'_k}_\bullet \to \eta_\bullet$ in $L^\alpha([0,\tf]\times\TTd)$, $\QQ_{M,a}$-almost surely; to ease notation, we will continue to denote the subsequence by $\eta^N_\bullet$. For each $N$ in the subsequence, we write \begin{equation}\label{eq: prelimit equation}
\langle \varphi, \eta^{N}_t\rangle =\langle \varphi, \eta^N_0\rangle + \int_0^t \cL_{N}F^\varphi_{N}(\eta^N_s)ds+ M^{{N},\varphi}_t
\end{equation} where $M^{{N},\varphi}$ is a $\QQ_{M,a}$-martingale, starting at zero. We now deal with the terms one by one; the left-hand side, and the first term of the right-hand side, converge to $\langle \varphi, \eta_t\rangle, \langle \varphi, \eta_0\rangle$ respectively, uniformly in $t\le \tf$, thanks to convergence in the topology of $\mathbb{D}$ and the fact that $\eta_\bullet \in \mathcal{C}$ almost surely. The jumps of $M^{{N}, \varphi}$ are of size at most $\lesssim \frac{\chi_{N} \|\nabla \varphi\|_\infty}{{N}^{d+1}}$ and occur at rate at most $\lesssim \frac{{N}^2}{\chi_{N}}\sum_{x\in \TND} (\eta^{N}_{t}(x))^\alpha$, which gives the elementary bound on the quadratic variation \begin{equation}
	\label{eq: QV} \left[M^{N,\varphi}\right]_t \le C\frac{\chi_N N^2}{N^{2d+2}} \int_0^t \sum_{x\in \TND} (\eta^N_{s}(x))^\alpha ds = C\frac{\chi_N}{N^d}\int_0^t \left\|\eta^N_s\right\|^\alpha_{L^\alpha(\TTd)} ds 
\end{equation} for some $C=C(\varphi)$. Thanks to the second part of Lemma \ref{lemma: QY1}, it follows that, for all $\delta>0$, \begin{equation}
	\QQ_{M,a}\left(\left[M^{N,\varphi}\right]_{\tf}>\delta\right) \to 0
\end{equation} and hence $M^{N, \varphi}_\bullet$ converges in $\QQ_{M,a}$-probability to $0$. Concerning the integral term, we use Lemma \ref{lemma: computation of generators} and the second part of Lemma \ref{lemma: QY1} again to find the $\QQ_{M,a}$-almost sure convergence \begin{equation}\int_0^\tf \left|\cL_N F_N(\eta^N_s)-\frac{1}{2}\int_{\TTd} \Delta \varphi(x)(\eta^N_s(x))^\alpha  dx \right| ds \to 0
	.\end{equation} Meanwhile, by the almost sure convergence in $L^\alpha([0,\tf]\times\TTd)$ obtained from Theorem \ref{thrm: WtS}ii), we have $\QQ_{M,a}$-almost surely,\begin{equation}\int_0^\tf \left|\frac{1}{2}\int_{\TTd} \Delta \varphi(x)(\eta^N_s(x))^\alpha dx -\frac{1}{2}\int_{\TTd} \Delta\varphi(x)(\eta_s(x))^\alpha dx \right| ds \to 0.\end{equation} Combining the previous two displays and returning to (\ref{eq: prelimit equation}), we conclude that, $\QQ_{M,a}$-almost surely, for all $t\le \tf$, \begin{equation} 	\langle \varphi, \eta_t\rangle =\langle \varphi, \eta_0\rangle + \frac12\int_0^t \int_{\TTd} (\Delta \varphi)(x)(\eta_s(x))^\alpha dx ds. 
\end{equation} By taking an intersection over a countable dense set $\varphi_n\in C^2(\TTd)$, we may extend this to holding, $\QQ_{M,a}$-almost surely, for all $\varphi \in C^2(\TTd)$ simultaneously. We have now checked that $\eta_\bullet$ satisfies Definition \ref{def: solutions} with $g=0$, and hence conclude that, $\QQ_{M,a}$-almost surely, $\eta_\bullet \in \mathcal{S}(u_0)$, which completes the step. \bigskip \\ \paragraph{\textbf{Step 3. Removal of Conditioning}} From steps 1-2 it follows that, for any $\cU\ni u_\bullet$ open in the topology of $\mathbb{D}$ and $M,a$ sufficiently large,\begin{equation}
	\PP^N_{M,a}\left(\eta^N_\bullet \not \in \cU\right)\to  0. 
\end{equation} To finish the theorem we now need to replace $\PP^N_{M,a}$ with $\PP$ in the above statement. Let us fix $\e>0$. Thanks to the hypotheses (\ref{eq: entropy UI hyp}) and the observation that this implies (\ref{eq: bound mass hyp}), we can find $M,a$ sufficiently large that the conditioning in (\ref{eq: conditioning for hydrodynamic limit}) satisfies \begin{equation}
	\limsup_N \PP\left(\left(A^N_{M,a}\right)^\mathrm{c}\right)=\limsup_N \PP\left(\cH(\eta^N_0)>M \text{ or } \langle 1, \eta^N_0\rangle > a\right)<\e. 
\end{equation} It now follows that \begin{equation}\begin{split}
	\limsup_N \PP\left(\eta^N_\bullet \not \in \cU\right) &\le \limsup_N\left[\PP\left(\left(A^N_{M,a}\right)^\mathrm{c}\right) + \PP^N_{M,a}\left(\eta^N_\bullet \not \in \cU\right) \right] <\e \end{split}
\end{equation} and, since $\e>0$ was arbitrary, the proof is complete.  \end{proof}

	\section{Proof of the Upper Bound} \label{sec: UB} Equipped with the variational formulation Lemma \ref{lemma: variational form} and Theorem \ref{thrm: WtS}, we are now in a position to prove the upper bound (\ref{eq: UB statement}) in Theorem \ref{thrm: LDP}. We begin by proving a local version for paths in $\cR$ in Lemma \ref{lemma: local upper bound} below; since the upper bound is already proven for paths not belonging to $\cR$ by Theorem \ref{thrm: WtS}i), the full upper bound (\ref{eq: UB statement}) will follow by the usual covering argument. \begin{lemma}\label{lemma: local upper bound} Assume (\ref{eq: scaling hypothesis}), fix $\rho\in C(\TTd,(0,\infty))$ and $\mathbb{P}$ as in Theorem \ref{thrm: LDP}, and let $u_\bullet \in \cR$. For every $z<\cI_\rho(u_\bullet) \in [0,\infty]$, there exists a $\cDD$-open neighbourhood $\cU=\cU(u_\bullet, z)\ni u_\bullet$ such that  \begin{equation}\label{eq: local upper bound}
	\limsup_N \frac{\chi_N}{N^d}\log \PP\left(\eta^N_\bullet \in \cU \right) \le - z.
\end{equation} \end{lemma}
\begin{proof}Fix $u_\bullet \in \cR$ and $z<\cI_\rho(u_\bullet)$; choose $\e>0$ so that $z+6\e<\cI_\rho(u_\bullet)$. Thanks to Lemma \ref{lemma: variational form}, we may find $\psi\in C(\TTd), \varphi\in C^{1,2}([0,\tf]\times\TTd)$ such that \begin{equation}\label{eq: choose psi phi}
	\Xi_0(\psi, u_0, \rho)+\Xi_1(\varphi, u_\bullet)>z+6\e. 
\end{equation} With this $\psi, \varphi$ fixed, we consider the nonnegative martingale \begin{equation} \label{eq: defn Z UB}\begin{split}  \mathcal{Z}^N_t:& =\exp\bigg(\frac{N^d}{\chi_N}\bigg(\langle \psi, \eta^N_0\rangle - \alpha \int_{\TTd}\rho(x)(e^{\psi(x)/\alpha}-1)dx + \langle \varphi_t, \eta^N_t\rangle - \langle \varphi_0, \eta^N_0\rangle \\ & \hspace{4.5cm} - \int_0^t \langle \partial_s\varphi_s, \eta^N_s\rangle ds   -\int_0^t \cG_N F_N^{\varphi_s}(\eta^N_s) ds \bigg) \bigg) \end{split}\end{equation} where, for each $s$, $F_N^{\varphi_s}$ is the evaluation $\eta^N\mapsto \langle \varphi_s, \eta^N \rangle$. Thanks to Lemma \ref{lemma: computation of generators}, we have \begin{equation}
	\left|\cG_N F_N^{\varphi_t}(\eta^N_t)-\frac{1}{2}\int_{\TTd} (\Delta \varphi_t(x)+|\nabla \varphi_t(x)|^2)(\eta^N_t(x))^\alpha dx\right| \le \theta_N\left\|\eta^N_t\right\|_{L^\alpha(\TND)}^\alpha
\end{equation} for some deterministic $\theta_N\to 0$; applying Lemma \ref{lemma: QY1}, it follows that the events \begin{equation}\label{eq: first error event}
	A^1_{N, \e}=\left\{\int_0^\tf \left| \cG_N F_N^{\varphi_t}(\eta^N_t)-\frac{1}{2}\int_{\TTd} (\Delta \varphi_t(x)+|\nabla \varphi_t(x)|^2)({\eta}^{N}_t(x))^\alpha dx \right| dt >\e  \right\} 
\end{equation} have asymptotic probability \begin{equation} \label{eq: first error event probability} \limsup_N \frac{\chi_N}{N^d}\log \PP(A^1_{N, \e})\le -z.
\end{equation} Next, noting that the map $\cDD\to (L^1_{\ge 0}(\TTd), d), v_\bullet \mapsto v_0$ is continuous, we can find a $\mathbb{D}$-open neighbourhood $\cU_1$ of $u_\bullet$ contained in \begin{equation}\begin{split}
	& \left\{v_\bullet \in \mathbb{D}: \h \sup_{t\le \tf} |\langle \varphi, u_t-v_t\rangle|<\e, \h \sup_{t\le \tf} \left|\int_0^t \langle \partial_s \varphi_s, u_s-v_s\rangle ds\right|<\e \right\} \\ & \hspace{7cm} \cap \left\{v_\bullet \in \mathbb{D}: \h |\langle \psi, u_0-v_0\rangle|<\e \right\} \end{split}
\end{equation} and consider the $L^{\alpha}([0,\tf]\times\TTd)$-open neighbourhood \begin{equation}
	\cV_2=\left\{v_\bullet \in L^{\alpha}([0,\tf]\times\TTd): \int_0^\tf \int_{\TTd}|u^\alpha_t(x)-v_t^\alpha(x)|dx dt<\e' \right\}
\end{equation} where $\e'>0$ is defined by $$ \e':= \frac{\e}{\|\Delta \varphi + |\nabla \varphi|^2\hspace{0.1cm} \|_{L^\infty([0,\tf]\times\TTd)}}.$$By applying Theorem \ref{thrm: WtS}, we can find $\cU_2\ni u_\bullet$, open in the topology of $\mathbb{D}$, such that \begin{equation}\label{eq: second error event}
	\limsup_N \frac{\chi_N}{N^d}\log \PP\left(\eta^N_\bullet \in \cU_2\setminus \cV_2\right)\le -z. 
\end{equation} We now take $\cU:=\cU_1\cap \cU_2$, and consider the event $$ A^2_{N,z,\e}:=\{\eta^N_\bullet \in \cU\}\setminus (A^1_{N,\e}\cup \{\eta^N_\bullet \in \cU_2\setminus \cV_2\}) $$ so that $$A^2_{N,z,\e}=\{\eta^N_\bullet \in \cU_1\cap \cV_2\}\setminus A^1_{N,z,\e} $$ and \begin{equation} \label{eq: UB containments} \{\eta^N_\bullet \in \cU\}\subset A^1_{N,\e}\cup \{\eta^N_\bullet \in \cU_2\setminus \cV_2\}\cup A^2_{N,\e,z}.  \end{equation} We will now bound $\mathcal{Z}^N_\tf$ below on the event $A^2_{N,z,\e}$, which will imply an upper bound on the probability of this event. We first observe that, since $A^2_{N,z,\e}$ is contained in the complement of the event in (\ref{eq: first error event probability}), on $A^2_{N,z,\e}$ i \begin{equation}
	\begin{split}
		\frac{\chi_N}{N^d}\log \mathcal{Z}^N_{\tf} \ge &\langle \psi, \eta^N_0\rangle -\alpha \int_{\TTd}\rho(x)(e^{\psi(x)/\alpha}-1)dx\\& +\langle \varphi_\tf, \eta^N_{\tf}\rangle - \langle \varphi_0, \eta^N_0\rangle - \int_0^\tf \langle \partial_s \varphi_s, \eta^N_s\rangle ds \\& -\frac{1}{2} \int_0^\tf \int_{\TTd} (\Delta \varphi_s +|\nabla \varphi_s(x)|^2)(\eta^N_s(x))^\alpha dxds \\ &-\e. 
	\end{split}
\end{equation} Next, we observe that on the event $A^2_{N,z,\e}$, we have the $\cDD$-closeness of $\eta^N_\bullet$ to $u_\bullet$ from $\eta^N_\bullet \in \cU_1$ and the $L^\alpha([0,\tf]\times\TTd)$-closeness from $\eta^N_\bullet \in \cV_2$. We use the definition of $\cU_1$ to replace all instances of $\eta^N_\bullet$ with $u_\bullet$ in the first two lines of the right-hand side, incurring a further error of at most $4\e$, and use the definitions of $\cV_2$ and $\e'$ to make the same replacement in the third line, incurring another error of at most $\e$. All together, on the event $A^2_{N,z,\e}$, 
\begin{equation}
	\begin{split}
		\frac{\chi_N}{N^d}\log \mathcal{Z}^N_{\tf} \ge &\langle \psi, u_0\rangle -\alpha \int_{\TTd}\rho(x)(e^{\psi(x)/\alpha}-1)dx\\& +\langle \varphi_\tf, u_{\tf}\rangle - \langle \varphi_0, u_0\rangle - \int_0^\tf \langle \partial_s \varphi_s, \eta^N_s\rangle ds \\& -\frac{1}{2} \int_0^\tf \int_{\TTd} (\Delta \varphi_s +|\nabla \varphi_s(x)|^2)(u_s(x))^\alpha dxds \\[1ex] &-6\e. 
	\end{split}
\end{equation} The first line of the right-hand side is nothing other than $\Xi_0(\psi, u_0, \rho)$, and the second and third lines together are $\Xi_1(\varphi, u_\bullet)$. Recalling the choice of $\psi$ and $\varphi$ in (\ref{eq: choose psi phi}), the right-hand side of the previous display is at least $z$, and we conclude that

\begin{equation} \label{eq: good event for Z} \begin{split} \mathcal{Z}^N_{\tf}\indiq[A^2_{N,z,\e}] &  \ge \exp\left(\frac{N^dz}{\chi_N}\right)\indiq[A^2_{N,z,\e}] \end{split}	
 \end{equation} whence \begin{equation} \label{eq: good event for Z 2} \begin{split} \PP(A^2_{N,z,\e}) &  \le  \exp\left(-\frac{N^dz}{\chi_N}\right)\EE \mathcal{Z}^N_{\tf}.\end{split}	
 \end{equation} Returning to the definition (\ref{eq: defn Z UB}) of $\mathcal{Z}^N_0$, the only terms contributing at $t=0$ are those coming from $\Xi_0(\psi, \eta^N_0, \rho)$, and using  Lemma \ref{lemma: cgf}, $\frac{\chi_N}{N^d}\log \EE[\mathcal{Z}^N_0]\to 0$. Recalling that $\mathcal{Z}^N_t$ is a martingale, Doob's inequality yields \begin{equation} \label{eq: good event for Z 3}
 	\limsup_N \frac{\chi_N}{N^d}\log \PP\left(A^2_{N,z,\e}\right)\le -z +\limsup_N \frac{\chi_N}{N^d}\log \EE\left[\mathcal{Z}^N_0\right] =-z. 
 \end{equation}
Gathering (\ref{eq: first error event probability}, \ref{eq: second error event}, \ref{eq: good event for Z 3}), we conclude that $$ \limsup_N \frac{\chi_N}{N^d}\log \PP\left(\eta^N_\bullet \in \cU\right)\le -z$$ as desired. \end{proof} Once we have this, the global upper bound follows easily by standard arguments.
\begin{proof}[Proof of the global upper bound (\ref{eq: UB statement})]  Fix $z<\inf_\mathcal{E} \cI_\rho$. Thanks to Lemma \ref{lemma: ET}, there exists a compact $\mathcal{K}\subset \mathbb{D}$ such that \begin{equation}
	\limsup_N \frac{\chi_N}{N^d}\log \PP\left(\eta^N_\bullet \not \in \mathcal{K}\right)<-z.
\end{equation} For every $u_\bullet \in \mathcal{E}\cap \mathcal{K}$, the hypotheses of Lemma \ref{lemma: local upper bound} hold and we find $\cU(u_\bullet,z) \ni u_\bullet$ open in the topology of $\cDD$ with \begin{equation}
	\limsup_N \frac{\chi_N}{N^d}\log \PP\left(\eta^N_\bullet \in \cU(u_\bullet,z)\right)\le -z. 
\end{equation} Since $\mathcal{K}$ is compact and $\mathcal{E}$ is closed in the (Hausdorff) topology of $\cDD$, we can find $n<\infty$ and $u_\bullet^{(i)}, 1\le i\le n$ such that \begin{equation}
	\mathcal{E}\cap \mathcal{K} \subset \bigcup_{i=1}^n \cU(u_\bullet^{(i)},z)
\end{equation} whence we find \begin{equation} \begin{split}
	&\limsup_N \frac{\chi_N}{N^d}\log \PP\left(\eta^N_\bullet \in \mathcal{E} \right) \\ & \hs \le \limsup_N \frac{\chi_N}{N^d} \log\left( \PP(\eta^N_\bullet \not \in \mathcal{K}) + \sum_{i=1}^n \PP(\eta^N_\bullet \in \cU(u^{(i)}_\bullet,z)) \right) \\ & \hs \le \max\left(\limsup_N \frac{\chi_N}{N^d}\log \PP\left(\eta^N_\bullet \in \cV \right): \cV= \mathbb{D}\setminus \mathcal{K}, \h \cU(u^{(1)}_\bullet,z), \dots \cU(u^{(n)}_\bullet,z)\right) \\[1ex] & \hspace{3cm} \le -z \end{split} 
\end{equation} and since $z< \inf_\mathcal{E} \cI_\rho$ was arbitrary, we are done. \end{proof} 

\section{Proof of Lower Bound}\label{sec: LB} We now prove the lower bound (\ref{eq: LB statement}). In order to obtain the \emph{true} lower bound, i.e. without restriction to a class of regular paths, it is sufficient to prove a following \emph{restricted} lower bound in the following Lemma. The improvement to a true lower bound then follows thanks to the properties of the dynamic cost $\mathcal{J}$ proven in \cite{fehrman2019large} and recalled in Proposition \ref{prop: lsc envelope}.   \begin{lemma}\label{lemma: local lower bound} Let $\PP, \rho$ be as in Theorem \ref{thrm: LDP}, let $\cX$ be as in Proposition \ref{prop: lsc envelope}, and let $u_\bullet \in \cX$. Then for any $\cDD$-open set $\cU\ni u_\bullet$,  \begin{equation}
	\liminf_N \frac{\chi_N}{N^d}\log \PP\left(\eta^N_\bullet \in \cU\right)\ge -\cI_\rho(u_\bullet).
\end{equation}
	
\end{lemma} Once this is established, (\ref{eq: LB statement}) will follow. 
\begin{proof}[Proof of Lemma \ref{lemma: local lower bound}] Let us fix $u_\bullet \in \cX$, which we recall means that $u_0\in C^\infty(\TTd, (0,\infty))$, that $u_t$ is bounded and bounded away from $0$, and a weak solution in the sense discussed above Proposition \ref{prop: uniqueness FP} to the Fokker-Planck equation \begin{equation} \label{eq: sk for cX} \partial_t u_t=\frac12 \Delta(u_t^\alpha)- \nabla\cdot\left(u_t^\alpha \nabla h_t(x)\right)\end{equation} for some $h\in C^{1,3}([0,\tf]\times\TTd)$, which we recall means that we take $g:=u^{\alpha/2}\nabla h$ in Definition \ref{def: solutions}. Let us also fix a $\cDD$-open set $\cU\ni u_\bullet$. We follow the standard pattern of explicitly constructing a change of measure, proving convergence under the `tilted' measures, and estimating the Radon-Nikodym derivative. \paragraph{\textbf{Step 1. Construction of a change of measure.}} For the initial data, let us pick $M>\sup_x u_0(x)$ and set \begin{equation}
	Y^N_0:=\prod_{x\in \TND} \frac{\pi^N_{u_0(x), M}(\eta^N_0(x))}{\pi^N_{\rho(x)}(\eta^N_0(x))}
\end{equation} where $\pi^N_{u_0(x),M}$ is the probability measure defined in Remark \ref{rmk: slowly varying local equilibrium}. Since the distribution of $\eta^N_0$ under $\PP$ is $\Law_{\PP}[\eta^N_0]=\Pi^N_\rho$, it follows that $\EE[Y^N_0]=1$, and moreover \begin{equation} \label{eq: initial distribution under Q}
	\EE\left[\indiq[\eta^N_0=\eta^N]Y^N_0\right] = {\Pi}^N_{u_0,M}(\eta^N)
\end{equation} is the conditioned slowly-varying local equilibrium given in Remark \ref{rmk: slowly varying local equilibrium}. For $N$ large enough, the conditioning factor $\pi^N_{u_0(x)}([0,M])>\frac{1}{2}$ for all $x$, and using the form of the single-site distributions $\pi^N_{\theta}(\cdot)$, we find the uniform bound \begin{equation}
	Y^N_0 \le \left(\frac{\|u_0\|_{L^\infty(\TTd)}}{\inf \rho}+1\right)^{\alpha MN^d/\chi_N}\left(\frac{2Z_\alpha(\|\rho\|_\infty/\chi_N)}{Z_\alpha(\inf u_0/\chi_N)}\right)^{N^d}
\end{equation} whence, using Lemma \ref{lemma: useful asymptotics} to control the terms involving $Z_\alpha$, \begin{equation} \label{eq: change of measure at time 0 weak est} 
	\limsup_N \frac{\chi_N}{N^d}\log \left\|Y^N_0\right\|_{L^\infty(\PP)}<\infty.
\end{equation}
	In order to modify the dynamics, we need to introduce some auxiliary objects. For $N<\infty$ and $x,y\in \TND$, let $J^N_{x,y,t}$ be the empirical measure on times $s\le t$ at which a particle jumps from site $x$ to $y$, and let $J^N_t=(J^N_{x,y,t}: x,y\in \TND)$ be the collection of all such processes. We observe that the pair $(\eta^N_t, J^N_t)$ is a time-inhomogeneous Markov process; we will denote objects associated to this `two-level' process with a $\widehat\cdot$ in order to avoid confusion. Let us write $\widehat\cL_{N,t}$ for its generator and $\widehat\cG_{N,t}$ for the corresponding rescaled nonlinear generator, defined in the same way as (\ref{eq: nonlinear generator}). We now set $\widehat{F}_N$ to be the functional given by, for $\eta^N\in X_N$ and empirical measures $(J^N_{x,y})_{x,y\in \TND}$ on $[0,\tf]$, \begin{equation} 
		\widehat{F}_N(\eta^N, J^N):=\frac{\chi_N}{N^d}\sum_{x, y\in \TND} \int_0^\tf \left(\overline{h}^N_s(y)-\overline{h}^N_s(x)(x)\right)J^N_{x,y}(ds)
	\end{equation} with $\overline{h}^N_s(x)$ being the local averages over cubes $c^N_x, x\in \TND$ as in (\ref{eq: representation of evaluation'}). The process \begin{equation}\label{eq: define Z}
		\mathcal{Z}^N_t:=\exp\left(\frac{N^d}{\chi_N}\left(\widehat F_N(\eta^N_t, J^N_t)-\int_0^t \widehat\cG_{N,s}\widehat F_N(\eta^N_s, J^N_s) ds\right)\right)
	\end{equation} is a nonnegative local martingale, and hence so is $Y^N_0\mathcal{Z}^N_t$. On the event $\{Y^N_0\neq 0\}$, we have an $N$-uniform bound on the mass $\langle 1, \eta^N_0\rangle$, and observing that \begin{equation}\label{eq: reformulate Fhat}
		\langle h_t, \eta^N_t\rangle - \langle h_0, \eta^N_0\rangle = \int_0^t \left\langle \partial_s h_s, \eta^N_s\right \rangle ds + \widehat{F}_N(\eta^N_t, J^N_t)
	\end{equation} which implies the existance of a finite $C$ with \begin{equation}\label{eq: change of measure dynamic 1 weak estimate}\sup_N \left\|\indiq_{\{Y^N_0\neq 0\}}\sup_{t\le \tf}|\widehat{F}_N(\eta^N_t, J^N_t)|\right\|_{L^\infty(\PP)}\le C.\end{equation} Meanwhile, a calculation analagous to Lemma \ref{lemma: computation of generators} (c.f. (\ref{eq: Ghat generator 1}, \ref{eq: Ghat generator 1}) below) shows that $|\widehat\cG_{N,s}\widehat F_N|\le \lambda \|\eta^N\|_{L^\alpha(\TTd)}^\alpha$. Using the first moment estimate in Lemma \ref{lemma: QY1} and combining with (\ref{eq: change of measure at time 0 weak est}, \ref{eq: change of measure dynamic 1 weak estimate}) we finally conclude that \begin{equation} \label{eq: change of measure weak est}
		\limsup_N \frac{\chi_N}{N^d}\log \EE\left[(Y^N_0\h \sup_{t\le \tf} \mathcal{Z}^N_{t})^p\right] <\infty
	\end{equation} for all $1\le p<\infty$. In particular, $Y^N_0\mathcal{Z}^N_t, t\le \tf$ is a uniformly integrable, nonnegative and mean-1 martingale, so we may now define the probability measure, on the same underlying probability space \begin{equation}\label{eq: define com}
		\mathbb{Q}^N(A):=\EE\left[\indiq_A Y^N_0\mathcal{Z}^N_\tf\right].
	\end{equation} Thanks to (\ref{eq: change of measure weak est}), it follows that all superexponential estimates transfer to $\mathbb{Q}^N$ in the sense that, if $A_{N,z}, N\ge 1, z<\infty$ are a parametrised family of events such that \begin{equation}
		\limsup_{z\to \infty} \limsup_N \frac{\chi_N}{N^d}\log \PP(A_{N,z}) =-\infty
	\end{equation} then the same is true with $\mathbb{Q}^N$ in place of $\PP$, see \cite[Theorem 3.2]{kipnis1989hydrodynamics}, and in particular the conclusions of Lemmata \ref{lemma: QY1}, \ref{lemma: ET} still hold with $\QQ^N$ in place of $\PP$. Furthermore, by taking any $p>1$ in (\ref{eq: change of measure weak est}), it also follows that $$\limsup_N \frac{\chi_N}{N^d}H\left(\Law_{\QQ^N}[\eta^N_\bullet]|\Law_{\PP}[\eta^N_\bullet]\right)<\infty $$ for $H(\cdot|\cdot)$ denoting the relative entropy of probability measures on $\cDD$, so Theorem \ref{thrm: WtS}ii) applies.  By the general change-of-measure formula \cite[Appendix 1, Proposition 7.3]{kipnis1998scaling}, the process $(\eta^N_t, J^N_t)$ is a time-inhomogeneous Markov process under this measure, and on functions depending only on $\eta^N$, the generator is given by \begin{equation}\begin{split} \label{eq: h generator}
		\widetilde\cL_{N,t}F_N(\eta^N):=\frac{dN^2}{\chi_N}& \sum_{x,y\in \TND} p^N(x,y)(\eta^N(x))^\alpha(F_N(\eta^{N,x,y})-F_N(\eta^N))\\ & \hspace{4.5cm}\dots \times \exp\left(\overline{h}^N_t(y)-\overline{h}^N_t(x)\right).\end{split}
	\end{equation}
	\paragraph{\textbf{Step 2. Convergence under the measures $\mathbb{Q}^N$.}} We now argue that, for any $\cDD$-open $\cU\ni u_\bullet$, \begin{equation}\label{eq: wk convergence under QN}\liminf_N \mathbb{Q}^N\left(\eta^N_\bullet \in \cU\right)=1. 	\end{equation}  As remarked above, Lemma \ref{lemma: ET} shows that $\Law_{\QQ^N}[\eta^N_\bullet]$ are tight in $\mathbb{D}$, so it is sufficient to show that the all possible subsequential limits in law almost surely take values in the solution set $\mathcal{S}^\mathrm{FP}_{h}(u_0)$ to the Fokker-Planck equation (\ref{eq: sk for cX}). By Proposition \ref{prop: uniqueness FP}, this is the singleton $\{u_\bullet\}$. Since this step is very similar to the hydrodynamic limit argued in Section \ref{sec: hydrodynamic}, we will only describe the most important steps of the argument.  \bigskip \\ Let us fix an infinite subsequence $S\subset \mathbb{N}$ under which the laws of $\eta^N_\bullet$ converge to the law of a (potentially random) $\eta_\bullet$; as before, we may realise all $\eta^N_\bullet, \eta_\bullet$ on a common probability measure $\mathbb{Q}$ which makes the converge almost sure in the topology of $\mathbb{D}$. By Theorem \ref{thrm: WtS}ii), we pass to a subsequence which makes the convergence almost sure in the norm topology of $L^\alpha([0,\tf]\times\TTd)$, and we get $\mathbb{Q}(\eta_\bullet \not \in \cR)=0$. \bigskip \\ For the initial data, we saw at (\ref{eq: initial distribution under Q}) that under $\mathbb{Q}$, $\eta^N_0\sim {\Pi}^N_{u_0,M}$ and by Remark \ref{rmk: slowly varying local equilibrium}, it follows that $d(\eta^N_0, u_0)\to 0$ in $\mathbb{Q}$-probability, so we must have $\mathbb{Q}(\eta_0=u_0)=1.$    If we repeat the Taylor expansion of Lemma \ref{lemma: computation of generators} on both $$ \overline{\varphi}^N(x+l)-\overline{\varphi}^N(x), \qquad e^{\overline{h}_t^N(x+l)-\overline{h}_t^N(x)} $$ we find that for any $\varphi\in C^2(\TTd)$,  the evaluation function $F^\varphi_N(\eta^N):=\langle \varphi, \eta^N\rangle$ satisfies \begin{equation} \label{eq: approximate generator COM}
		\left|\widetilde\cL_{N,t} F^\varphi_N(\eta^N)-\frac1{2}\int_{\TTd}(\eta^N(x))^\alpha\left(2\nabla \varphi(x)\cdot \nabla h_t(x) +\Delta \varphi(x)\right) \right| \le \theta_N\|\eta^N\|_{L^\alpha(\TTd)}^\alpha
	\end{equation} where $\widetilde\cL_{N,t}$ is the generator (\ref{eq: h generator}) under $\mathbb{Q}$, and where $\theta_N \to 0$ depends only on $\varphi$, and in particular is independent of $t$. As in the proof of Theorem \ref{th: hydrodynamic limit} in Section \ref{sec: hydrodynamic}, we may therefore write \begin{equation}\label{eq: prelimit COM}
		\langle \varphi, \eta^N_t\rangle = \langle \varphi, \eta^N_0\rangle + M^{N,\varphi}_t+\int_0^t \widetilde{\cL}_{N,t}F^\varphi_N(\eta^N_s)ds 
	\end{equation} with $M^{N,\varphi}_\bullet$ a $\QQ$-martingale with the same quadratic variation bound (\ref{eq: QV}). Since the subsequence was chosen to have $\QQ$-almost sure convergence in $L^\alpha([0,\tf]\times\TTd)$, we use (\ref{eq: approximate generator COM}) and take the limit of each term of (\ref{eq: prelimit COM}), exactly as in Section \ref{sec: hydrodynamic}, to find \begin{equation}
		\mathbb{Q}\left(\sup_{t\le \tf} \left|\langle \varphi, \eta_t-u_0\rangle -\int_0^t \int_{\TTd}\left(\frac12\Delta \varphi(x)+ \nabla \varphi (x) \cdot \nabla h_t(x)\right) (\eta_t(x))^\alpha \right|=0 \right) =1. 
	\end{equation} Taking an intersection over a countable dense set of $\varphi_n \in  C^2(\TTd)$, we see that, $\QQ$-almost surely, $\eta_\bullet$ is a solution to the skeleton equation (\ref{eq: Sk}) with $g=\eta^{\alpha/2}\nabla h$ starting at $\eta_0=u_0$, which is the definition of weak solutions to (\ref{eq: sk for cX}). In particular, $\eta_\bullet \in \cS^{\mathrm{FP}}_{h}(u_0)$ almost surely, as claimed, and the step is complete. \paragraph{\textbf{Step 3. Estimate of the Radon-Nikodym Derivative}.} To conclude, we must estimate the Radon-Nikodym derivative $\frac{d\QQ^N}{d\PP}$. We will now show that, for all $\e>0$,  \begin{equation}\label{eq: estimate RN derivative}
		\mathbb{Q}^N\left(\left|\frac{\chi_N}{N^d}\log (Y^N_0\mathcal{Z}^N_{\tf})-\cI_\rho(u_\bullet)\right|>\e\right)\to 0
	\end{equation} and, in conjunction with the previous step, the conclusion of the lemma will follow quickly in step 4. \paragraph{\textbf{Step 3a. Control of $Y^N_0$.}}   We start by writing, $\QQ^N$-almost surely,  \begin{equation}\begin{split} \label{eq: log initial density}
		\frac{\chi_N}{N^d}\log Y^N_0&=\frac{1}{N^d}\sum_{x\in \TND} \biggr(\chi_N \log \left(\frac{\pi^N_{u_0(x)}(\eta^N_0(x))}{\pi^N_{\rho(x)}(\eta^N_0(x))}\right)  -\chi_N \log \pi^N_{u_0(x)}([0,M])\biggr) \end{split}
	\end{equation} where the last term results from the conditioning of $\pi^N_{u_0(x)}$ to $\{1, ..., M/\chi_N\}$. In the first two terms, we use the explicit form of the single-site distribution $\pi^N_{\rho}$ to find \begin{equation}\label{eq: first term of initial density}
		\chi_N \log \left(\frac{\pi^N_{u_0(x)}(\eta^N_0(x))}{\pi^N_{\rho(x)}(\eta^N_0(x))}\right) =\alpha\eta^N_0(x)\log \left(\frac{u_0(x)}{\rho}\right) -\chi_N  \log \frac{Z(u_0(x)/\chi_N)}{Z(\rho(x)/\chi_N)}.
	\end{equation} Since $\log(u_0(x)/\rho(x))$ is bounded uniformly in $x\in \TND, N\ge 1$ and since $\eta^N_0(x)\sim \pi^N_{u_0(x),M}$ under $\QQ^N$,  Remark \ref{rmk: slowly varying local equilibrium} shows that \begin{equation}\label{eq: initial entropy 1}
		\sup_{x\in \TND} \EE_{\QQ^N}\left[ \left|\alpha\eta^N_0(x)\log\left(\frac{u_0(x)}{\rho(x)}\right) -\alpha u_0(x)\log \left(\frac{u_0(x)}{\rho(x)}\right)\right|^2\right]\to 0
	\end{equation} where we recall that we chose $M>\sup_x u_0(x)$. The second term of (\ref{eq: first term of initial density}) does not depend on $\eta^N_0$, and we apply Lemma \ref{lemma: useful asymptotics} and boundedness and positivity of $\rho(x), u_0(x)$ to see that \begin{equation}\label{eq: initial entropy 2}
		\sup_{x\in \TND}\left|\chi_N \log \frac{Z_\alpha(u_0(x)/\chi_N)}{Z_\alpha(\rho(x)/\chi_N)} - \alpha \rho(x)\left(\left(\frac{u_0(x)}{\rho(x)}\right)-1\right)\right|\to 0. 
	\end{equation} For the final term in (\ref{eq: log initial density}), Remark \ref{rmk: slowly varying local equilibrium} shows that the conditioning factor is at least $\frac{1}{2}$ for all $x$ as soon as $N$ is large enough, and so the contribution to the sum is at most $\chi_N \log 2\to 0$. Gathering (\ref{eq: initial entropy 1}, \ref{eq: initial entropy 2}) and returning to (\ref{eq: log initial density}), we have proven that \begin{equation}
		\mathbb{E}_{\QQ^N}\left[\left|\frac{\chi_N}{N^d} \log Y^N_0-\frac1{N^d}\sum_{x\in \TND} \alpha \rho(x)\left(\frac{u_0(x)}{\rho(x)}\log \frac{u_0(x)}{\rho(x)} - \frac{u_0(x)}{\rho(x)} +1\right)\right|\right]\to 0
	\end{equation} Meanwhile, by continuity and positivity of $\frac{u_0(x)}{\rho(x)}$, the sum converges to the integral of the same quantity, which is exactly $\alpha \cH_\rho(u_0)$. To conclude, \begin{equation}
		\label{eq: initial com conclusion} \EE_{\QQ^N}\left[\left|\frac{\chi_N}{N^d}\log Y^N_0 - \alpha \cH_\rho(u_0)\right|\right]\to 0. 
	\end{equation}\paragraph{\textbf{Step 3b. Control of $\mathcal{Z}^N_{\tf}$.}} We next turn to the dynamical change of measure induced by $\mathcal{Z}^N_{\tf}$. Using the definition (\ref{eq: define Z}) and the observation (\ref{eq: reformulate Fhat}), we write \begin{equation}
		\begin{split} \label{eq: break up dynamic cost}
			\frac{\chi_N}{N^d}\log \mathcal{Z}^N_{\tf} =\langle h_{\tf}, \eta^N_{\tf}\rangle - \langle h_0, \eta^N_0\rangle - \int_0^\tf \langle \partial_s h_s, \eta^N_s\rangle ds -\int_0^\tf \widehat{\cG}_{N,s}\widehat{F}_N(\eta^N_s, J^N_s)ds. 
		\end{split}
	\end{equation} The first three terms define a continuous function on $\mathbb{D}$, which therefore converges to its value at $u_\bullet$ in $\mathbb{Q}^N$-probability, by step 2, so for any $\e>0$, \begin{equation} \label{eq: dynamic change of ms 1}
		\QQ^N\left(\left|\langle h_{\tf}, \eta^N_{\tf}-u_{\tf} \rangle - \langle h_0, \eta^N_0-u_0 \rangle - \int_0^\tf \langle \partial_s h_s, \eta^N_s-u_s \rangle ds\right|>\e\right)\to 0. 
	\end{equation} We must now compute the nonlinear generator in order to deal with the final term in (\ref{eq: break up dynamic cost}). Using the definition of $\widehat{F}_N$, we write \begin{equation}
		\label{eq: Ghat generator 1} \begin{split}
			\widehat{\cG}_{N,t}\widehat{F}_N(\eta^N, J^N) & = \frac{d}{N^{d-2}} \sum_{x,y\in \TND} (\eta^N(x))^\alpha \left(e^{\overline{h}^N_t(y)-\overline{h}^N_t(x)}-1\right)p^N(x,y) \\ &=\cG_N F^{h_t}_N(\eta^N)
		\end{split}
	\end{equation} where $F^{h_t}_N: X_N\to \RR$ is the evaluation $\langle h_t, \eta^N\rangle$. This is exactly one of the quantities computed in Lemma \ref{lemma: computation of generators}, and in particular \begin{equation}\label{eq: Ghat generator 2}
		\left|\widehat{\cG}_{N,t}\widehat{F}_N(\eta^N, J^N) - \frac{1}{2}\int_{\TTd}(\eta^N(x))^\alpha \left(\Delta h_t(x)+|\nabla h_t(x)|^2\right)\right|\le \theta_N\left\|\eta^N\right\|_{L^\alpha(\TTd)}^\alpha
	\end{equation} for some nonrandom $\theta_N\to 0$. Using the convergence in $L^\alpha([0,\tf]\times\TTd)$ obtained in step 2, for any $\e>0$,  \begin{equation}\label{eq: dynamic change of ms 2}
		\mathbb{Q}^N\left(\left|\int_0^\tf \widehat{\cG}_{N,t}\widehat{F}_N(\eta^N_t, J^N_t)dt-\frac12\int_0^\tf \int_{\TTd} u_t(x)^\alpha(\Delta h_t(x)+|\nabla h_t(x)|^2) dx dt\right|>\e\right)\to 0. 
	\end{equation} Combining (\ref{eq: dynamic change of ms 1}, \ref{eq: dynamic change of ms 2}), returning to (\ref{eq: break up dynamic cost}) and recalling the functionals $\Xi_1$ defined at (\ref{eq: def xi1}), we have proven that \begin{equation}
		\mathbb{Q}^N\left(\left|\frac{\chi_N}{N^d}\log \mathcal{Z}^N_\tf - \Xi_1(h, u_\bullet) \right|>\e\right)\to 0. 
	\end{equation} Using (\ref{eq: sk for cX}) and taking the test function $h$ in Definition \ref{def: solutions}, a simple computation\footnote{See also the proof of Lemma \ref{lemma: variational form} in Appendix \ref{sec: misc proofs}.}  shows that \begin{equation}
		\Xi_1(h, u_\bullet)=\frac12\int_0^\tf \int_{\TTd} u_t(x)^\alpha |\nabla h_t(x)|^2 dtdx = \frac{1}{2}\left\|u^{\alpha/2}\nabla h\right\|_{{L^2_{t,x}}}^2. 
	\end{equation} Using the final part of Lemma \ref{lemma: variational form}, it is clear that $g=u^{\alpha/2}\nabla h$ has the property (\ref{eq: characterisation}) characterising the unique optimiser for $\mathcal{J}$, and is in particular the unique $g$ minimising the cost subject to $u_\bullet$ solving the skeleton equation (\ref{eq: Sk}). The right-hand side of the last display is therefore the dynamical cost $\mathcal{J}(u_\bullet)$ and we conclude that \begin{equation}
		\mathbb{Q}^N\left(\left|\frac{\chi_N}{N^d}\log \mathcal{Z}^N_\tf - \mathcal{J}(u_\bullet) \right|>\e\right)\to 0. 
	\end{equation} Combining with (\ref{eq: initial com conclusion}), we have proven (\ref{eq: estimate RN derivative}) as claimed.\paragraph{\textbf{Step 4. Conclusion.}} We now show how the previous steps imply the conclusion of Lemma \ref{lemma: local lower bound}.  Fixing $u_\bullet\in \cX$ and a $\cDD$-open set $\cU\ni u_\bullet$,  let $\mathbb{Q}^N$ be the changes of measure constructed in step 1. Thanks to step 2, $\QQ^N(\eta^N_\bullet\in \cU)\to 1$ and by step 3, for all $\e>0$, $$\QQ^N\left(\frac{\chi_N}{N^d}\log \frac{d\QQ^N}{d\PP}\le \cI_\rho(u_\bullet)+\e \right) \to 1. $$ In particular, for $N$ large enough, the event \begin{equation}
		\label{eq: Q prob final} A^N_{\cU,\e}:=\left\{\eta^N_\bullet\in \cU, \frac{\chi_N}{N^d}\log \frac{d\QQ^N}{d\PP}\le \cI_\rho(u_\bullet)+\e\right\}
	\end{equation} has $\QQ^N$-probability at least $\frac12$. We now estimate, for all such $N$, \begin{equation}
		\begin{split}
			\PP\left(\eta^N_\bullet \in \cU\right) &\ge \PP\left(A^N_{\cU,\e}\right)=\EE_{\QQ^N}\left[\left(\frac{d\QQ^N}{d\PP}\right)^{-1}\indiq[A^N_{\cU,\e}]\right] \\& \ge \exp\left(-\frac{N^d(\cI_\rho(u_\bullet)+\e)}{\chi_N}\right)\QQ^N\left(A^N_{\cU,\e}\right) \\ & \ge \frac{1}{2} \exp\left(-\frac{N^d(\cI_\rho(u_\bullet)+\e)}{\chi_N}\right).
		\end{split}
	\end{equation} Taking the logarithm and extracting the speed $N^d/\chi_N$, we find the conclusion of the lemma, now with $\cI_\rho(u_\bullet)$ replaced by $\cI_\rho(u_\bullet)+\e$ on the right-hand side. Since $\e>0$ was arbitrary, we now take $\e\to 0$ to conclude the proof.  \end{proof}
	
	We now conclude the lower bound (\ref{eq: LB statement}), which is the final part of Theorem \ref{thrm: LDP}. \begin{proof}[Proof of (\ref{eq: LB statement})] Let us fix $\PP, \rho$ as in Theorem \ref{thrm: LDP} and consider any $\cDD$-open set $\cU$. For any $u_\bullet \in \cU$ and $\e>0$, Proposition \ref{prop: lsc envelope} shows that there exists $v_\bullet \in \cU\cap \cX$ with $\cI_\rho(v_\bullet)<\cI_\rho(u_\bullet)+\e$, and Lemma \ref{lemma: local lower bound} shows that \begin{equation}
	\liminf_N \frac{\chi_N}{N^d}\log \PP\left(\eta^N_\bullet\in \cU\right)\ge -\cI_\rho(v_\bullet)> -\cI_\rho\left(u_\bullet\right) -\e. 
\end{equation} Since $u_\bullet\in \cU$ and $\e>0$ were arbitrary, taking $\e\downarrow 0$ and optimising over $u_\bullet \in \cU$ proves that $$ \liminf_N \frac{\chi_N}{N^d}\log \PP \left(\eta^N_\bullet \in \cU\right)\ge -\inf\left\{\cI_\rho(u_\bullet): u_\bullet \in \cU\right\} $$which is the claim (\ref{eq: LB statement}), and the proof is complete. \end{proof} 

\section{From Large Deviations to Gradient Flow} \label{sec: gradient flow} We now prove Theorem \ref{thrm: gradient flow} as  a consequence of the large deviations principle and the microscopic reversibility. This proof was inspired by a similar proof of the Boltzmann $H$-Theorem in kinetic theory by the second author \cite[Section 6.5]{heydecker2021kac}. We recall the notation $L^2_{t,x}:=L^2([0,\tf]\times\TTd,\RRd)$ and $\|\cdot\|_{L^2_{t,x}}$ from the introduction.   \begin{proof} We divide into steps. Firstly, the equality is trivial if $u_\bullet \not \in \cR$, since both sides are infinite. On the left-hand side, the property that $\mathcal{J}=\infty$ outside $\cR$ follows from the definition of $\mathcal{J}$ as already remarked in the discussion under (\ref{eq: energy class}-\ref{eq: energy class'}). On the right-hand side, if $\cA(u_\bullet)<\infty$, then $u_\bullet \in \mathcal{C}$, so the finiteness of the right-hand side implies that $u_\bullet \in \cR$ according to the definition (\ref{eq: energy class}). Next, assuming that $u_\bullet \in \cR$, we observe that $\cA(u_\bullet)<\infty$ if and only if $\mathcal{J}(u_\bullet)<\infty$. Indeed, the assumption $u_\bullet \in \cR$ implies $\nabla u^{\alpha/2} \in L^2_{t,x}$, and it is straightforward to verify that the map $$ \theta\mapsto g=\nabla u^{\alpha/2}+\frac\theta2 $$ is a bijection of $L^2_{t,x}$ onto itself, and from the set of competitors $\theta$ for $\cA$ onto the set of competitors $g$ for $\mathcal{J}$. In particular, the existence of a competitor in $L^2_{t,x}$ for either problem implies the same for the other problem, and hence the finiteness of the corresponding infimum. We may therefore restrict attention to the case $$u_\bullet\in \cR, \qquad \cA(u_\bullet)<\infty, \qquad \mathcal{J}(u_\bullet)<\infty.$$  \paragraph{\textbf{Step 1. Definition of Time-Reversal}} Let us define $\mathcal{T}:\mathbb{D}\to \mathbb{D}$ by the time reversal \begin{equation}
	(\mathcal{T}\eta)_t=\eta_{(\tf-t)-}.
\end{equation} The detailed balance condition for $\Pi^N_\rho, \rho\in (0,\infty)$ implies that \begin{equation}
	\Law_{\Pi^N_\rho}[\eta^N_\bullet]=\Law_{\Pi^N_\rho}[\mathcal{T}\eta^N_\bullet]
\end{equation} and hence that $\mathcal{T}\eta^N_\bullet$ satisfy the same large deviation principle as $\eta^N_\bullet$. On the other hand, $\mathcal{T}$ is a continuous involution, and so the contraction principle (see Theorem 4.2.1 in the book of Dembo and Zeioutini \cite{dembo2009large} or \cite[Lemma 3.11]{feng2006large}) implies that $\mathcal{T}\eta^N_\bullet$ have a large deviation principle with rate function \begin{equation}\begin{split}
	\mathcal{I}^\mathcal{T}_\rho(u_\bullet)&=\inf\left\{\mathcal{I}_\rho(v_\bullet): \mathcal{T}v_\bullet = u_\bullet \right\} =\mathcal{I}_\rho(\mathcal{T}u_\bullet).  \end{split}
\end{equation} Since rate functions are unique \cite[Remark 3.1]{feng2006large}, we conclude that, for all $u_\bullet \in \mathbb{D}$,  \begin{equation}\label{eq: TIME REVERSE RATE FN} \mathcal{I}_\rho(u_\bullet)=\mathcal{I}_\rho(\mathcal{T}u_\bullet).\end{equation} Let us stress that this is the key equality on which the proof rests; although the objects involved in the statement of the theorem do not involve probability, this equality has a natural probabilistic meaning, and indeed we are not aware of any proof of it which does not involve the large deviations.  \bigskip \\  To ease notation, it will be helpful to also have a time reversal operation $\mathcal{T}$ on $L^2_{t,x}$; given $g\in L^2_{t,x}$, let us define $$ (\mathcal{T}g)_t(x):=g_{\tf-t}(x) $$ so that $\mathcal{T}$ is readily seen to be an isometry of $L^2_{t,x}$. \medskip \\ \paragraph{\textbf{Step 2. Orthogonal Projection.}} The other key ingredient which we will use in step 3 below is the orthogonal projection onto a $u_\bullet$-dependent linear subspace of $L^2_{t,x}$. For $u_\bullet \in \cR$, let $\Lambda_{u_\bullet}$ be the linear subspace of $\Lambda$ defined in (\ref{eq: characterisation}), which characterises the minimisers for $\cA, \mathcal{J}$, and let $\Pi[u_\bullet]:L^2_{t,x}\to L^2_{t,x}$ be the orthogonal projection onto $\Lambda_{u_\bullet}$. \medskip \\ Let us make a few remarks which follow easily from the definition, and will be used later. Firstly, given $u_\bullet$ solving (\ref{eq: Sk}) for a control $g$, it also holds that (\ref{eq: Sk}) holds with the new control $\Pi[u_\bullet]g$. Indeed, whenever we test a function $\varphi \in C^{1,2}([0,\tf]\times\TTd)$ in Definition \ref{def: solutions}, the difference $$ \int_0^\tf \int_{\TTd}u_t^{\alpha/2}\nabla \varphi_t(x)\cdot\left(g-\Pi[u_\bullet]g\right) dt dx = 0$$ by definition of $\Pi[u_\bullet]$. Secondly, let us give some interactions between time reversal and projection. For $u_\bullet \in \cR$,  \begin{equation} \label{eq: time reverse tangent} \mathcal{T}\Lambda_{u_\bullet} =\{\mathcal{T}g: g\in \Lambda_{u_\bullet}\} = \Lambda_{\mathcal{T}u_\bullet} \end{equation}which follows by reversing $\varphi_t(x)\mapsto \varphi_{\tf-t}(x)$ in an approximating sequence. Since $\mathcal{T}$ is an isometry on $L^2_{t,x}$,  $$ \Pi[\mathcal{T}u_\bullet]\mathcal{T}g=\mathcal{T}\left(\Pi[u_\bullet]g\right), \qquad g\in L^2_{t,x}.$$ \paragraph{\textbf{Step 3. Equality for Projected Entropy Dissipation and Upper Bound.}}  We will now prove that, for $u_\bullet \in \cR$, we have the equality \begin{equation}
	\label{eq: LD to GF w Projection} \mathcal{J}(u_\bullet)=\frac12\left(\alpha\cH(u_\tf)-\alpha\cH(u_0)+\left\|\Pi[u_\bullet]\nabla u^{\alpha/2}\right\|^2_{L^2_{t,x}} + \cA(u_\bullet)\right)
\end{equation}  where $\Pi[u_\bullet]$ is the orthogonal projection defined in step 2. It follows from the definition (\ref{eq: entropy dissipation}) and the fact that $\Pi[u_\bullet]$ is nonexpansive in $\Lambda$ that \begin{equation} \label{eq: projection decreases entropy dissipation} 
	\left\|\Pi[u_\bullet]\nabla u^{\alpha/2}\right\|^2_{L^2_{t,x}} \le \left\|\nabla u^{\alpha/2}\right\|^2_{L^2_{t,x}}=\frac{\alpha}{2}\int_0^\tf \cD_\alpha(u_s)ds 
\end{equation} and so it follows from (\ref{eq: LD to GF w Projection}) that the claim (\ref{eq: conclusion of gf}) in the theorem holds with an inequality:\begin{equation}
	\label{eq: conclusion of gf with inequality} \mathcal{J}(u_\bullet)\le \frac12\left(\alpha\cH(u_\tf)-\alpha\cH(u_0)+\frac{\alpha}{2}\int_0^\tf \cD_\alpha(u_s)ds + \cA(u_\bullet)\right).
\end{equation} Let us fix $u_\bullet \in \cR$ and set $v_\bullet:=\mathcal{T}u_\bullet$. As remarked in Lemma \ref{lemma: variational form}, \ref{lemma: variational A}, membership in $\Lambda_{u_\bullet}$ characterises the unique optimisers in the definitions (\ref{eq: dynamic cost}, \ref{eq: action 2}) of $\mathcal{J}(u_\bullet), \cA(u_\bullet)$. Letting $\theta$ be the optimal choice attaining $\mathcal{A}(u_\bullet)$, we remarked above that (\ref{eq: Sk}) holds for the control $\frac\theta2-\nabla u^{\alpha/2}$, and by the first observation in step 2, the same is true of \begin{equation}\label{eq: find g from theta}g:=\Pi[u_\bullet]\left(\frac{\theta}{2}-\nabla u^{\alpha/2}\right)=\frac\theta2-\Pi[u_\bullet]\nabla u^{\alpha/2}\end{equation}  where we have recalled from Lemma \ref{lemma: variational A} that $\theta\in \Lambda_{u_\bullet}$. Since $\Pi[u_\bullet]$ maps into $\Lambda_{u_\bullet}$ by definition, this construction produces a candidate control $g$ with $g\in \Lambda_{u_\bullet}$, which is therefore the optimiser by Lemma \ref{lemma: variational form}.  \bigskip \\ Using the time-reversal, we next observe that $v_\bullet$ solves, in the sense of Definition \ref{def: solutions}, \begin{equation}
	\begin{split} \partial_t v_t& =-\frac{1}{2}\Delta(v_t^\alpha)+\nabla\cdot(v^{\alpha/2}_t (\mathcal{T}{g})_t)\\& = \frac{1}{2}\Delta(v_t^\alpha)-\nabla\cdot(2v^{\alpha/2}_t\nabla v^{\alpha/2}_t-v^{\alpha/2}_t(\mathcal{T}g)_t).\end{split} 
\end{equation} Using the first observation in step 2, a candidate control for the reversed path is given by \begin{equation}
	g_\mathrm{r}:=2\Pi[v_\bullet]\nabla v^{\alpha/2}-\mathcal{T}{g}.
\end{equation} We know from the second observation of step 2 that $\mathcal{T}{g}\in \Lambda_{v_\bullet}$, so $g_\mathrm{r}\in \Lambda_{v_\bullet}$ and hence attains the dynamic cost for $v_\bullet$. It follows that $$\cI_\rho(v_\bullet)= \alpha \cH_\rho(u_\tf)+\frac{1}{2}\left\|g_\mathrm{r}\right\|_{L^2_{t,x}}^2.$$ We now observe that \begin{equation}
	\|g_\mathrm{r}+\mathcal{T}{g}\|_{L^2_{t,x}}^2=4\left\|\Pi[v_\bullet]\nabla v^{\alpha/2}\right\|^2_{L^2_{t,x}} = 4\left\|\Pi[u_\bullet]\nabla u^{\alpha/2}\right\|^2_{L^2_{t,x}}
\end{equation} while \begin{equation}\begin{split}
	\|g_\mathrm{r}-\mathcal{T}{g}\|_{L^2_{t,x}}^2&=4\left\|\Pi[v_\bullet]\nabla v^{\alpha/2}-\mathcal{T}{g}\right\|^2_{L^2_{t,x}} \\ &=4\left\|\Pi[u_\bullet]\nabla u^{\alpha/2}-g\right\|^2_{\Lambda} =\|\theta\|^2_{L^2_{t,x}}= 2\cA(u_\bullet)\end{split} 
\end{equation} where in the second equality we use the final observation of step 2, the third equality follows from the definition (\ref{eq: find g from theta}) of $g$ from $\theta$ and the final equality follows because $\theta$ was chosen to be the unique minimiser of the action.  Recalling the key equality (\ref{eq: TIME REVERSE RATE FN}), the finiteness of $\mathcal{J}(u_\bullet)$ and the hypothesis that $\cH_\rho(u_0)<\infty$ implies $\cI_\rho(u_\bullet)<\infty$, and hence that $$\alpha\cH(u_\tf)\le \cI_\rho(v_\bullet)=\cI_\rho(u_\bullet)<\infty.$$  Using (\ref{eq: TIME REVERSE RATE FN}), we now write the chain of equalities \begin{equation}
	\begin{split}
		0& =\cI_\rho(v_\bullet)-\cI_\rho(u_\bullet) \\ & = \alpha\cH_\rho(u_\tf)-\alpha\cH_\rho(u_0) + \frac{1}{2}\left(\|g_\mathrm{r}\|_{L^2_{t,x}}^2+\|\mathcal{T}g\|_{L^2_{t,x}}^2\right)- \|g\|_{L^2_{t,x}}^2 \\ & = \alpha\cH_\rho(u_\tf)-\alpha\cH_\rho(u_0) + \frac{1}{4}\left(\|g_\mathrm{r}+\mathcal{T}g\|_{L^2_{t,x}}^2+\|g_\mathrm{r}-\mathcal{T}g\|_{L^2_{t,x}}^2\right)- \|g\|_{L^2_{t,x}}^2  \\ & =\alpha\cH_\rho(u_\tf)-\alpha \cH_\rho(u_0)+\left\|\Pi[u_\bullet]\nabla u^{\alpha/2}\right\|^2_{L^2_{t,x}} + \frac{1}{2}\cA(u_\bullet) - \|g\|_{L^2_{t,x}}^2
	\end{split}
\end{equation} whence \begin{equation}\label{eq: GF step 2}
	\alpha \cH_\rho(u_\tf)-\alpha \cH_\rho(u_0)+\left\|\Pi[u_\bullet]\nabla u^{\alpha/2}\right\|^2_{L^2_{t,x}}+ \frac{1}{2}\cA(u_\bullet) = \|g\|_{L^2_{t,x}}^2
\end{equation} and the claim (\ref{eq: LD to GF w Projection}) is proven. \paragraph{\textbf{Step 4. Lower Bound by approximation.}} We now prove a matching lower bound for $\mathcal{J}$ to complement the previous step. For any given $u_\bullet \in \cR$, we use Proposition \ref{prop: lsc envelope} to find $u^{(n)}_\bullet\in \mathcal{X}$, such that each $u^{(n)}_\bullet \in C^{1,2}([0,\tf]\times\TTd)$ is bounded away from $0$, such that $u^{(n)}_\bullet\to u_\bullet$ in the topology of $\mathbb{D}$, with \begin{equation} \label{eq: control cost gf approximation} \mathcal{J}(u^{(n)}_\bullet)\to \mathcal{J}(u_\bullet); \qquad \cH_\rho(u^{(n)}_0)\to \cH_\rho(u_0). \end{equation} For each $n$, $\inf_{t,x}u^{(n)}_t(x)>0$, whence $\log u^{(n)}\in C^{1,2}([0,\tf]\times\TTd)$ and $$\nabla (u^{(n)})^{\alpha/2}=\frac\alpha 2 (u^{(n)})^{\alpha/2}\nabla \log u^{(n)} \in \Lambda_{u^{(n)}_\bullet}.$$ It follows that $\Pi[u^{(n)}_\bullet]$ has no effect on $\nabla(u^{(n)})^{\alpha/2}$, so $$ \left\|\Pi[u^{(n)}_\bullet]\nabla (u^{(n)}_\bullet)^{\alpha/2}\right\|^2_{L^2_{t,x}}=\left\|\nabla (u^{(n)}_\bullet)^{\alpha/2}\right\|^2_{L^2_{t,x}}=:\frac{\alpha}{2}\int_0^\tf\cD_\alpha(u^{(n)}_t)dt. $$ For each $n$, the previous step now yields \begin{equation}\label{eq: GF equality prelimit}
	\mathcal{J}(u^{(n)}_\bullet) = \frac{1}{2}\left(\alpha \cH_\rho(u^{(n)}_\tf)-\alpha \cH_\rho(u^{(n)}_0)+\frac{\alpha}{2}\int_0^\tf \cD_\alpha(u^{(n)}_s)ds +\cA(u^{(n)}_\bullet)\right).
\end{equation} We now take the limit of (\ref{eq: GF equality prelimit}). Thanks to (\ref{eq: control cost gf approximation}) and the previous estimate, it follows that $\limsup_n \int_0^\tf \cD_\alpha(u^{(n)}_s) ds <\infty$ and so may apply the approximation results in Remark \ref{rmk: continnum nonstandard convolution} and Lemma \ref{lemma: variational A}. It is standard that \begin{equation}\label{eq: limit of entropy}
	\liminf_n \cH_\rho(u^{(n)}_\tf) \ge \cH_\rho(u_\tf) 
\end{equation} while by the final part of Remark \ref{rmk: continnum nonstandard convolution}, \begin{equation}\label{eq: limit of entropy diss} \liminf_n \int_0^\tf \cD_\alpha(u^{(n)}_s)ds\ge \int_0^\tf \cD_\alpha(u_s)ds
\end{equation} and by Lemma \ref{lemma: variational A}, \begin{equation}
	\label{eq: limit action} \liminf_n \cA(u^{(n)}_\bullet) \ge \cA(u_\bullet).
\end{equation} Meanwhile, the only negative term $-\cH_\rho(u^{(n)}_0) \to -\cH_\rho(u_0)$ is explicitly controlled by (\ref{eq: control cost gf approximation}). Taking the limit inferior of both sides of (\ref{eq: GF equality prelimit}) and recalling the first part of (\ref{eq: control cost gf approximation}), \begin{equation}
	\begin{split}
		\mathcal{J}(u_\bullet)& =\liminf_n \mathcal{J}(u^{(n)}_\bullet) \\&  \ge \frac12 \left(\alpha \cH_\rho(u_\tf)-\alpha \cH_\rho(u_0)+\frac{\alpha}{2}\int_0^\tf \cD_\alpha(u_s)ds + \frac12 \cA(u_\bullet)\right). 
	\end{split}
\end{equation} In light of the matching upper bound (\ref{eq: conclusion of gf with inequality}), we have the equality (\ref{eq: conclusion of gf}) claimed in the theorem. \medskip \\ \paragraph{\textbf{Step 5. Geometric Consequences.}} Thanks to the equality proven in steps 3-4, the inequality (\ref{eq: projection decreases entropy dissipation}) must also be an equality as soon as $\mathcal{J}(u_\bullet)<\infty$, which implies that $\nabla u^{\alpha/2} \in \Lambda_{u_\bullet}$. Using the definition (\ref{eq: characterisation}), it follows that, for almost all $t\ge 0$, $\nabla u_t^{\alpha/2}$ is the $L^2(\TTd,\RRd)$-limit of functions $u_t^{\alpha/2}\nabla \varphi^{(n)}_t, n\to \infty$, with $\varphi^{(n)}\in C^2(\TTd)$. This is exactly the definition (\ref{eq: tangent space useful def}) of $T_{u_t}\cM_{\text{ac},\lambda}(\TTd)$, and the claim is proven for $\nabla u_t^{\alpha/2}$. The equivalent claim for $g$ was already noted in Remark \ref{rmk: g tangent space}, which proves the claims of the theorem.  \bigskip \\ \paragraph{\textbf{Step 6. $H$-Theorem for the Porous Medium Equation.}} Finally, let us show how the previous steps imply the equality (\ref{eq: entropy dissipation equality}) for solutions $u_\bullet \in \cR$ to the PME (\ref{eq: PME}), continuing in the notation of step 3. Since $g=0$ for $u_\bullet$, the candidate control $g_\mathrm{r}$ for the reversed path $v_\bullet$ is given by \begin{equation}
	\label{eq: H thrm control} g_\mathrm{r}=2\Pi[v_\bullet]\nabla v^{\alpha/2} = 2\nabla v^{\alpha/2}
\end{equation} where the second equality holds thanks to step 5 above, and this is indeed the optimal control thanks to Lemma \ref{lemma: variational form}. Substituting this into the definition of $\mathcal{I}_\rho(v_\bullet)$ in (\ref{eq: TIME REVERSE RATE FN}) produces $\cI_\rho(u_\bullet)=\alpha \cH_\rho(u_0)$ and \begin{equation}
	\cI_\rho(v_\bullet)=\alpha \cH_\rho(u_\tf)+\frac{1}{2}\|g_\mathrm{r}\|_{L^2_{t,x}}^2 = \alpha \cH_\rho(u_\tf)+\alpha\int_0^\tf \cD_\alpha(u_s)ds
\end{equation} and the equality result follows from (\ref{eq: TIME REVERSE RATE FN}). This proves the equality of the $H$-theorem at $\tf$: \begin{equation}\label{eq: entropy dissipation equality tf} \cH_\rho(u_\tf)+\int_0^\tf \cD_\alpha(u_s)ds = \cH_\rho(u_0) \end{equation} which is the special case $t=\tf$ of the claim (\ref{eq: entropy dissipation equality}) in the theorem. To extend this to all intermediate times, note that (\ref{eq: basic entropy estimate}) shows that, for all $t\in (0, \tf)$, \begin{equation} \label{eq: entropy dissipation inequality} \cH_\rho(u_0)\ge \cH_\rho(u_t)+\int_0^t \cD_\alpha(u_s)ds\end{equation} and, applying the same on the time interval $[t, \tf]$,  \begin{equation} \cH_\rho(u_t)\ge \cH_\rho(u_t)+\int_t^\tf \cD_\alpha(u_s)ds.\end{equation} If we assume a strict inequality in (\ref{eq: entropy dissipation inequality}) for any $t\in (0,\tf)$, we add the previous two displays together to find a strict inequality at $t=\tf$. This contradicts (\ref{eq: entropy dissipation equality tf}), and so (\ref{eq: entropy dissipation inequality}) must hold with equality for all $0\le t\le \tf$, as claimed.  \end{proof}

\section*{Acknowledgments} 

We express our gratitude to the anonymous reviewers, whose feedback and suggestions have helped to improve the presentation. The first author acknowledges support by the Max Planck Society through the Research Group ``Stochastic Analysis in the Sciences (SAiS)''. This work was co-funded by the European Union (ERC, FluCo, grant agreement No. 101088488). Views and opinions expressed are however those of the author(s) only and do not necessarily reflect those of the European Union or of the European Research Council. Neither the European Union nor the granting authority can be held responsible for them. The second author is supported by the Royal Commission for the Exhibition of 1851.

\appendix
\section{Wasserstein Geometry of the Gradient Flow Formulation}\label{sec: formal GF}

In order to give context to Theorem \ref{thrm: gradient flow}, we now give a formal description of a Riemannian geometry on the space $\cM_{\text{ac},\lambda}(\TTd)$ of all absolutely continuous measures on $\TTd$ with a prescribed mass $\lambda$. The Wasserstein-type geometry goes back to the work of Jordan, Kinderlehrer and Otto \cite{jordan1998variational} and Otto \cite{otto2001geometry}; the metric which arises naturally from the large deviations was introduced in \cite{dirr2016entropic} and called the \emph{thermodynamic metric}. \bigskip \\ The tangent space ${T}_u\cM_{\mathrm{ac},\lambda}(\TTd)$ is given as the set of all derivatives $\partial_t u_t|_{t=0}$ for smooth curves in $u_t$ taking values in $\cM_{\mathrm{ac},\lambda}(\TTd)$ with $u_0=u$. Given $\zeta \in {T}_u\cM_{\mathrm{ac},\lambda}(\TTd)$, we associate a vector field $v$ via the continuity equation $\zeta+ \nabla\cdot(uv)=0$. \bigskip \\ We now argue formally for $u\in L^\alpha(\TTd)\cap \cM_{\text{ac},\lambda}(\TTd)$. Following \cite{dirr2016entropic}, the nonlinearity of the diffusion leads one to seek such a $v$ with the minimal \emph{weighted} kinetic energy $\int_{\TTd} 2u^{2-\alpha}|v|^2$, which takes the form $v=\frac{1}{2}u^{\alpha-1}\nabla\xi$ for a function $\xi$, and the continuity equation becomes \begin{equation} \label{eq: CE0}
		\zeta+\nabla\cdot(\frac{1}{2}u^{\alpha}\nabla \xi)=0.
	\end{equation} The metric tensor is then given by, for $\zeta_1, \zeta_2$ and $\xi_1, \xi_2$ the associated functions, \begin{equation} \label{eq: metric tensor}\begin{split} g_u(\zeta_1, \zeta_2) & =:\frac{1}{2}\int_{\TTd} u(x)^\alpha\nabla \xi_1\cdot \nabla \xi_2 dx \end{split}\end{equation} which modifies the Wasserstein gradient flow structure of Otto \cite{otto2001geometry} by replacing the factor $u\mapsto u^\alpha$. This allows the immersion of the tangent space into $L^2(\TTd,\RRd)$, and we identify, see also Ambrosio \cite[Section 8.4]{ambrosio2005gradient} \begin{equation} \label{eq: tangent space useful def'}
		{T}_u \cM_{\text{ac},\lambda}(\TTd):=\overline{\left\{u^{\alpha/2}\nabla \varphi: \varphi\in C^2(\TTd)\right\}}^{L^2(\TTd,\RRd)}
	\end{equation} which is the definition (\ref{eq: tangent space useful def}) used in the text. The continuity equation (\ref{eq: CE}) becomes \begin{equation}\label{eq: CE'} \zeta+\frac12 \nabla\cdot(u^{\alpha/2}\theta), \qquad \theta\in T_u\mathcal{M}_{\text{ac},\lambda}(\TTd). \end{equation} This suggests how we see the term $\frac12\Delta u^\alpha$ appearing as a time-derivative in the PME (\ref{eq: PME}) or skeleton equation (\ref{eq: Sk}) in this framework. The functional derivative of the entropy $\cH$ along the directions tangential to $\mathcal{M}_{\text{ac},\lambda}$ is $\frac{\delta \cH}{\delta u}= \log u$, and taking $$\xi=-\alpha \log u =- \alpha \frac{\delta \cH}{\delta u}$$ exactly produces $\zeta=\alpha \nabla \cH(u)=-\frac12\Delta u^\alpha$. This therefore suggests that (\ref{eq: PME}) should be the gradient flow in this geometry of the functional $\alpha\cH(u)$, and using the definition (\ref{eq: metric tensor}), norm squared of the gradient is \begin{equation}\label{eq: gradient term}
		g_u(\alpha \nabla \cH, \alpha \nabla \cH)=\frac12 \int_{\TTd} \left|u^{\alpha/2} (\alpha \nabla \log u)\right|^2=2\int_{\TTd}|\nabla u^{\alpha/2}|^2 =: \alpha \cD_\alpha(u)
	\end{equation} while for curves $u_\bullet$ in $\mathcal{M}_{\text{ac},\lambda}(\TTd)$, the action (\ref{eq: action 2}) coincides with $\int_0^\tf g_{u_t}(\dot u_t, \dot u_t) dt$.\bigskip \\   Let us highlight one immediate issue in making the previous discussion rigorous. In order to identify $\frac{1}{2}\Delta u^\alpha$ as a time derivative, we need apply (\ref{eq: CE'}) to $\theta=\nabla u^{\alpha/2}$, while it is not \emph{a priori} clear that this is an element of the tangent space in the definition (\ref{eq: tangent space useful def}, \ref{eq: tangent space useful def'}); this is not obviously implied even by the finiteness of the entropy dissipation. Indeed, the situation is even worse, because we wish to express a condition between the large deviation rate functional $\cI_\rho$ and the gradient flow, so we deal with the case where $u_\bullet$ is, in general, a solution to (\ref{eq: Sk}), and we cannot appeal to any properties of (\ref{eq: PME}) to restrict attention to `better behaved'\footnote{For example, the inclusion $\nabla u^{\alpha/2}\in T_u\mathcal{M}_{\text{ac},\lambda}$ is not difficult to prove if one assumes that $u\in L^\infty(\TTd,[0,\infty))\cap \mathcal{M}_{\text{ac},\lambda}(\TTd))$ and $\cD_\alpha(u)<\infty$. However, the boundedness is much too strong a property for typical large deviations paths, as remarked in Section \ref{sec: prelim}. This is related to the supercriticality of (\ref{eq: Sk}) in $L^p(\TTd)$ for any $p>1$, see \cite[Section 2.2]{fehrman2019large}.} $u$. In particular, we must address the question of whether \begin{equation} 
		\nabla u_t^{\alpha/2} \in T_u\cM_{\text{ac},\lambda}(\TTd)\quad \text{ for almost all } t\ge 0 \label{eq: dubious tangent inclusion} 
	\end{equation} assuming only that $u_\bullet \in \mathcal{R}$ is a solution to (\ref{eq: Sk}) for some $g$. A direct approach to this problem from the definition (\ref{eq: tangent space useful def}, \ref{eq: tangent space useful def'}) appears to be highly nontrivial. Indeed, this same issue appears in the proof of in Theorem \ref{thrm: gradient flow} in Section \ref{sec: gradient flow}: if it were clear, in the notation of Section \ref{sec: gradient flow}, that $\nabla u^{\alpha/2} \in \Lambda_{u_\bullet}$, then step 2 would already establish the desired equality (\ref{eq: conclusion of gf}), and step 3 would be unnecessary. We obtain the inclusion, which implies (\ref{eq: dubious tangent inclusion}), in step 5 of the cited proof, only thanks to the approximation argument, benefiting from the previous knowledge of (\ref{eq: Sk}) thanks to \cite{fehrman2019large}. 
	
	\section{Deferred Proofs}\label{sec: misc proofs} \begin{proof}[Proof of Lemma \ref{lemma: computation of generators}] For $x\in \TND$, let us set $\overline{\varphi}^N(x)$ to be the average of $\varphi$ over the cube $c^N_x$, so that $F^\varphi_N$ has the representation \begin{equation}\label{eq: representation of evaluation}
	F^\varphi_N(\eta^N)=\frac{1}{N^d}\sum_{x\in \TND} \eta^N(x)\overline{\varphi}^N(x).
\end{equation}
	We compute \begin{equation}\begin{split}
			\label{eq: linear generator on evaluation} \cL_N F^\varphi_N(\eta^N)&=\frac{dN^2}{N^d}\sum_{x,y\in \TND}\frac{(\eta^N(x))^\alpha}{\chi_N}p^N(x,y)\left(\chi_N\overline{\varphi}^N(y)-\chi_N\overline{\varphi}^N(x)\right) \\ &=\frac{d}{N^{d-2}}\sum_{x,y\in \TND}(\eta^N(x))^\alpha p^N(x,y)\left(\overline{\varphi}^N(y)-\overline{\varphi}^N(x)\right)\\& =\frac{1}{2N^{d-2}}\sum_{x\in \TND, \h l\in \{\pm e_i/N, i\le d\}}(\eta^N(x))^\alpha  \left(\overline{\varphi}^N(x+l)-\overline{\varphi}^N(x)\right) \end{split} 
		\end{equation}and \begin{equation}\begin{split}
			\label{eq: nonlinear generator on evaluation} \cG_N F^\varphi_N(\eta^N)&=\frac{d\chi_N N^2}{N^d}\sum_{x,y\in \TND}\frac{(\eta^N(x))^\alpha }{\chi_N}p^N(x,y)\left(e^{\overline{\varphi}^N(y)-\overline{\varphi}^N(x)}-1\right) \\ &=\frac{d}{N^{d-2}}\sum_{x,y\in \TND}(\eta^N(x))^\alpha  p^N(x,y)\left(e^{\overline{\varphi}^N(y)-\overline{\varphi}^N(x)}-1\right) \\ &=\frac{1}{2N^{d-2}}\sum_{x\in \TND, \h l\in \{\pm e_i/N, i\le d\}}(\eta^N(x))^\alpha  \left(e^{\overline{\varphi}^N(x+l)-\overline{\varphi}^N(x)}-1\right)\end{split} 
		\end{equation} where in both of the final lines we have recalled the fact that $p^N$ is the nearest neighbours transition matrix to turn the sum over $y$ with respect to $dp^N(x,\cdot)$ into a sum over tangent vectors. We observe that\begin{equation}
			\overline{\varphi}^N(x+l)-\overline{\varphi}^N(x)=N^d\int_{c^N_0} (\varphi(x+l+z)-\varphi(x+z))dz
		\end{equation} and recalling that $\varphi\in C^2(\TTd)$, we Taylor expand to second order. We write $\theta_{N}$ for errors, potentially with additional arguments $x, (x,l), (x,l,z)$, which are $\mathfrak{o}(1)$ (uniformly in all arguments) and which may change from line to line. We find \begin{equation}\begin{split}
			\overline{\varphi}^N(x+l)-\overline{\varphi}^N(x) &= N^d\int_{c^N_0}\left(\nabla \varphi(x+z)\cdot l +\frac{1}{2}l\cdot \nabla^2 \varphi(x+z) l+ \theta_N(x,l,z)N^{-2}\right)dz \\[1ex] & =\nabla \varphi(x)\cdot l+ \frac{1}{2} l\cdot \nabla^2 \varphi(x) l + \theta_N(x,l)N^{-2}\end{split}
		\end{equation}  whence  \begin{equation}\begin{split}
			e^{\overline{\varphi}^N(x+l)-\overline{\varphi}^N(x)}-1 &= \varphi(x+l)-\varphi(x)+\frac{1}{2}(\varphi(x+l)-\varphi(x))^2 +\theta_{N}(x,l)N^{-2} \\[1ex] & = \nabla \varphi(x)\cdot l +\frac{1}{2}l\cdot \nabla^2 \varphi(x)l + \frac{1}{2}\left(\nabla \varphi(x)\cdot l\right)^2 + \theta_N(x,l)N^{-2}. \end{split}
		\end{equation} For fixed $x$, we sum over $l\in \{\pm e_i/N, 1\le i\le d\}$. By symmetry, $\sum_l l=0$ and the first-order terms vanish, and the remaining terms produce \begin{equation}\label{eq: taylor expansion of linear generator}
			\begin{split}
				\sum_{l\in \{\pm e_i/N\}} (\overline{\varphi}^N(x+l)-\overline{\varphi}^N(x))=N^{-2}\left(\Delta \varphi(x) + \theta_N(x)\right).
			\end{split} \end{equation} and  \begin{equation}\label{eq: taylor expansion of nonlinear generator}
			\begin{split}
				\sum_{l\in \{\pm e_i/N\}}(e^{\overline{\varphi}^N(x+l)-\overline{\varphi}^N(x)}-1)=N^{-2}\left(\Delta \varphi(x)+|\nabla \varphi(x)|^2 + \theta_N(x)\right).
			\end{split}
		\end{equation} Substituting into (\ref{eq: linear generator on evaluation}) and (\ref{eq: nonlinear generator on evaluation}) respectively, we find the claimed result, with $\Delta \varphi, \nabla \varphi$ replaced by versions which are piecewise constant on each box $c^N_x$. Using continuity to control the difference, both claims follow.
\end{proof}

				\begin{proof}[Sketch Proof of Proposition \ref{prop: uniqueness FP}] We will sketch the main ideas of the proof. The first step of the argument is to define a renormalised solution, in analogy to \cite[Definition 2.4]{fehrman2019large}, based on the \emph{kinetic function} $\chi:[0,\tf]\times \TTd\times\RR\to \RR$ of $u_\bullet$ given by \begin{equation}
					 \chi(t,x,\xi):=1(0<\xi<u_t(x))-1(u_t(x)<\xi<0)
				\end{equation} and imposing the relation between $g$ and $u$ in \cite[Equation (2.12)]{fehrman2019large}. The additional relation between $g$ and $u$ does not change the argument of \cite[Theorem 4.3]{fehrman2019large}, which establishes that the weak formulation is equivalent to the renormalised formulation, and hence it is sufficient to prove the uniqueness in the renormalised formulation, for which we follow \cite[Theorem 3.3]{fehrman2019large}. Following the same argument, we find that for two solutions $u_\bullet, v_\bullet$ to (\ref{eq: FP}), the $L^1(\TTd)$ norm can be written \begin{equation}\begin{split}
					\|u_t-v_t\|_{L^1(\TTd)}& =\int_{\RR\times\TTd}\left|\chi^u_t-\chi^v_t\right|^2 dy d\xi\\& =\lim_{M\to \infty}\lim_{\e, \delta\downarrow 0} \int_{\RR\times\TTd}\left|\chi^{u,\e,\delta}_t-\chi^{v,\e,\delta}_t\right|^2 \zeta^M(\xi) dy d\xi
				\end{split}\end{equation} where $\zeta^M$ is a cutoff in the velocity, supported on $\{M^{-1}\le \xi\le M+1\}$, and $\chi^{\{u,v\},\e,\delta}$ are given by convolving the kinetic functions of $u, v$ respectively with an approximation to the identity $\kappa^{\e,\delta}:(\TTd)^2\times(\RR)^2\to [0,\infty)$ of scale $\e$ in the spatial variables and $\delta$ in the velocity, such that the original function $\kappa$ has support contained in $\{|x-y|\le 1, |\xi-\xi'|\le 1\}$. It is now sufficient to show that \begin{equation}\label{eq: conclusion uniqueness FP}
					\limsup_{M\to \infty}\limsup_{\delta\downarrow 0}\limsup_{\e\downarrow 0} \partial_t \left(\int_{\RR\times\TTd}\left|\chi^{u,\e,\delta}_t-\chi^{v,\e,\delta}_t\right|^2 \zeta^M(\xi) dy d\xi\right) \le 0
				\end{equation} for almost every $t\le \tf$. Identical arguments to those of \cite{fehrman2019large} decompose the derivative into the same four terms \cite[Equations (3.10 - 3.14)]{fehrman2019large}. To make the appropriate modifications, note that the use of distinct variables $x,x'$ for the convolutions defining $\chi^{u,\e,\delta}, \chi^{v,\e,\delta}$ ensures that the control  $g^u_t(x)=u^{\alpha/2}_t(x)\nabla_x h$ associated to $u$ always appears with the variable $x$, while $g^v_t(x')=v^{\alpha/2}_t(x)\nabla_{x'}h$ always appears with the variable $x'$. Of these four terms, the only one where the difference of the controls is important is the term  \cite[Equation (3.13)]{fehrman2019large}
					\begin{equation}\label{eq: control term}
						\begin{split}
							\partial_t\mathbb{I}^{\e,\delta,M}_{t,\mathrm{con}} &= \frac2\alpha\int_{\RR\times(\TTd)^3} \Theta_t(x,x')\cdot u_t(x)^{1-\alpha/2}v_t(x')^{\alpha/2} \nabla_x u_t^{\alpha/2}(x)\overline{\kappa}^{\e,\delta}_{1,t}\overline{\kappa}^{\e,\delta}_{2,t}\zeta^M \h dx dx' dy d\xi
\\ & + \frac2\alpha\int_{\RR\times(\TTd)^3} \Theta_t(x,x')\cdot u_t^{\alpha/2}(x)v_t(x')^{1-\alpha/2} \nabla_{x'} v_t^{\alpha/2}(x')\overline{\kappa}^{\e,\delta}_{1,t}\overline{\kappa}^{\e,\delta}_{2,t}\zeta^M \h dx dx' dy d\xi
						\end{split}
					\end{equation} with the shorthand $$\overline{\kappa}^{\e,\d}_{1,t}:=\overline{\kappa}^{\e,\d}_{1,t}(x,y,\xi)= \kappa^{\e,\d}(x,y,u_t(x),\xi);  $$ $$ \overline{\kappa}^{\e,\d}_{2,t}:=\overline{\kappa}^{\e,\d}_{2,t}(x',y,\xi)= \kappa^{\e,\d}(x',y,v_t(x'),\xi)$$ and $$ \Theta_t(x,x'):=g^u_t(x)-g^v_t(x').$$ We now note that, on the support of $\overline{\kappa}^{\e,\delta}_{1,t}\overline{\kappa}^{\e,\delta}_{2,t}\zeta^M $, we have the bounds $$|x-y|\le \e, |x'-y|\le \e, \quad |u_t(x)-\xi|, |v_t(x')-\xi|\le \delta,$$ $$ M^{-1}\le |\xi|\le M$$ and as soon as $\delta<\frac{1}{2M}$, these imply that $u_t(x), v_t(x')\in ((2M)^{-1}, 2M)$, allowing us to bound the powers of $u_t(x), v_t(x')$. On this region, we also bound the difference of the controls by \begin{equation}
						\begin{split}
							\left|\Theta_t(x,x')\right| & \le |u_t^{\alpha/2}(x)-\xi^{\alpha/2}||\nabla_x h_t(x)|+|\xi^{\alpha/2}||\nabla_x h_t(x)-\nabla_{x'}h_t(x')| \\ & \hspace{3cm} + |v_t^{\alpha/2}(x')-\xi^{\alpha/2}|\nabla_{x'}h_t(x')| \\[1ex]&\hspace{-2cm} \le CM^{|1-\alpha/2|}\delta \|\nabla h\|_{L^\infty([0,\tf]\times\TTd)}+(M+1)^{\alpha/2}\sup_{t,x,x': |x-x'|\le 2\e}|\nabla h_t(x)-\nabla h_t(x')|
						\end{split}
					\end{equation} where the factor of $CM^{|1-\alpha/2|}$ is the Lipschitz constant of $\xi\mapsto \xi^{\alpha/2}$ on the interval $((2M)^{-1}, (2M))$. We thus conclude that, on the support of $\overline{\kappa}^{\e,\delta}_{1,t}\overline{\kappa}^{\e,\delta}_{2,t}\zeta^M $, we have the uniform bound \begin{equation} |g^u_t(x)-g^v_t(x')|\le \vartheta(\e,\d,M) \end{equation} for some $\vartheta(\e,\d,M)$ with $\limsup_{\e,\d\downarrow 0} \vartheta(\e,\d,M) = 0$ for all $M<\infty$. We now return to (\ref{eq: control term}) and use Young's inequality and Cauchy-Schwarz to find the distributional inequality \begin{equation}
						|\partial_t \mathbb{I}^{\e,\d,M}_{t,\mathrm{con}}|\le c(M)\vartheta(\e,\delta,M)\|\nabla u^{\alpha/2}_t\|_{L^2(\TTd)}\|\nabla v^{\alpha/2}_t\|_{L^2(\TTd)}
					\end{equation} for some $c=c(M)$ depending only on the cutoff $M$ in velocity. By hypothesis on $u_\bullet, v_\bullet \in \cR$ and the Cauchy-Schwarz inequality, the right-hand side belongs to $L^1([0,\tf])$ and we conclude that, as distributions, \begin{equation}
						\limsup_{M\to \infty} \left(\limsup_{\e,\d\downarrow 0} \partial_t\mathbb{I}^{\e,\d,M}_{t,\mathrm{con}}\right)\le 0.
					\end{equation} The rest of the proof proceeds exactly as in the proof of \cite[Theorem 3.3]{fehrman2019large} to obtain (\ref{eq: conclusion uniqueness FP}) and we conclude that, for any two solutions $u_\bullet, v_\bullet$ to (\ref{eq: FP}) for the same $h$, \begin{equation}
						\sup_{t\le \tf} \|u_t-v_t\|_{L^1(\TTd)} \le \|u_0-v_0\|_{L^1(\TTd)}.
					\end{equation} It immediately follows that if $u_0=v_0$ then $u_\bullet=v_\bullet$, and we are done. 
				\end{proof}
				
				Before giving the proof of Lemmata \ref{lemma: variational form} and \ref{lemma: variational A}, let us give without proof the following reformulation of the skeleton equation (\ref{eq: Sk}).  \begin{lemma}[Alternative Formulation of weak solutions] \label{lemma: reformulation} Fix $g\in {L^2_{t,x}}$ and $u_\bullet \in \cR$. Then $u_\bullet$ is a weak solution to the skeleton equation (\ref{eq: Sk}) if, and only if, for all $\varphi\in C^{1,2}([0,\tf]\times\TTd)$, \begin{equation}\label{eq: reformulation} \begin{split} \langle \varphi_\tf, u_{\tf}\rangle =\langle \varphi_0, u_0\rangle & +\int_0^{\tf} \langle \partial_t \varphi_t, u_t\rangle + \frac12\int_0^\tf \int_{\TTd} u_t(x)^\alpha  \Delta \varphi_t(x)dx dt \\& +\int_0^\tf \int_{\TTd} g_t(x)u^{\alpha/2}_t(x)\cdot \nabla \varphi_t(x) dt dx.  \end{split}
				\end{equation}
					
				\end{lemma}
				
				Let us observe that the hypothesis $u_\bullet \in \cR$ implies that $u_\bullet \in L^\alpha([0,\tf]\times \TTd)$ as in (\ref{eq: energy class'}), and in particular all terms appearing in (\ref{eq: reformulation}) are well-defined.
\begin{proof}[Proof of Lemmata \ref{lemma: variational form} and \ref{lemma: variational A}] We divide the proof into steps for clarity. \paragraph{\textbf{Step 1. Lower Bound of (\ref{eq: variational}).}} We will first show that \begin{equation}
	\begin{split}
			\label{eq: variational LB} & \frac{1}{2}\inf\left\{\left\|g\right\|_{{L^2_{t,x}}}^2: u_\bullet \text{ solves (\ref{eq: Sk})}\right\}\\&\hspace{4cm} \ge \sup_{\varphi \in C^{1,2}([0,\tf]\times \TTd)}\Xi_1(\varphi, u_\bullet).\end{split} \end{equation} If the set of $g\in {L^2_{t,x}}$ such that $u_\bullet$ solves (\ref{eq: Sk}) is empty, the infimum is infinite, and there is nothing to prove. Otherwise, let us fix such a $g$. Using the reformulation in Lemma \ref{lemma: reformulation}, it holds for all $\varphi \in C^{1,2}([0,\tf]\times\TTd)$ that \begin{equation}\begin{split}
				\langle \varphi_{\tf},u_{\tf}\rangle =\langle \varphi_0, u_0\rangle + \int_0^\tf \langle \partial_t\varphi_t, u_t\rangle & + \frac12\int_0^\tf \int_{\TTd} \Delta \varphi_t(x) u_t(x)^\alpha \\& + \int_0^\tf \int_{\TTd} g_t(x)u^{\alpha/2}_t(x)\cdot \nabla \varphi_t(x) \end{split}
			\end{equation} so that $\Xi_1$ simplifies to  \begin{equation}\begin{split}
				\Xi_1(\varphi, u_\bullet) =&\int_0^\tf \int_{\TTd} \left(g_t(x)u^{\alpha/2}_t(x)\cdot \nabla \varphi_t(x) -\frac12 u_t(x)^\alpha |\nabla \varphi_t(x)|^2\right)dt dx \\ &\qquad \qquad  \le \frac{1}{2}\int_0^\tf \int_{\TTd} |g_t(x)|^2 dt dx:=\frac{1}{2}\|g\|_{L^2_{t,x}}^2  \end{split}
			\end{equation} by Cauchy-Schwarz. Optimising over $\varphi \in C^{1,2}([0,\tf]\times\TTd)$ and minimising over $g$ subject to $u_\bullet \in \cS_g(u_0)$ proves (\ref{eq: variational LB}) as desired. \bigskip \\ \paragraph{\textbf{Step 2. Upper Bound of (\ref{eq: variational}) by Hilbert-space analysis.}} To complement the previous step, we now fix $u\in \cR$. If the right-hand side of (\ref{eq: variational LB}) is infinite then both expressions in (\ref{eq: variational LB}) are infinite and there is nothing to prove. We may therefore assume that it is finite; in this case, we will construct $g\in {L^2_{t,x}}$ such that (\ref{eq: Sk}) holds and with \begin{equation}
				\label{eq: variational UB'} \frac12\left\|g\right\|_{{L^2_{t,x}}}^2 \le \sup_{\varphi\in C^{1,2}([0,\tf]\times\TTd)} \Xi_1(\varphi, u)
			\end{equation} which proves that the bound in (\ref{eq: variational LB}) is an equality. We define a nonnegative-semidefinite bilinear form $\langle \langle \cdot, \cdot \rangle \rangle _{u_\bullet}$ on $C^{1,2}([0,\tf]\times\TTd)$ by\begin{equation}
				\label{eq: inner product} \lllangle \varphi, \psi \rrrangle_{u_\bullet} :=\int_0^\tf \int_\TTd u_t(x)^\alpha \nabla \varphi_t(x)\cdot \nabla \psi_t(x) dtdx. 
			\end{equation}Let us briefly emphasise that this is well-defined, since we assumed $u_\bullet \in \cR$ and $\cR \subset L^\alpha([0,\tf]\times\TTd)$ by (\ref{eq: energy class'}). We now define a Hilbert space $H^1_{u_\bullet}$ as the completion of $C^{1,2}([0,\tf]\times \TTd)$ with respect to this bilinear form, quotiented by the relation $$ \lllangle \varphi-\psi, \varphi-\psi\rrrangle_{u_\bullet}=0$$ in order to restore positive definiteness. We now observe that for $\varphi\in C^{1,2}([0,\tf]\times\TTd)$, we can write $\Xi_1(\varphi, u_\bullet)$ as \begin{equation}\label{eq: hilbert form}
				\Xi_1(\varphi, u_\bullet)=L_{u_\bullet} \varphi - \frac{1}{2}\lllangle \varphi, \varphi\rrrangle_{u_\bullet}
			\end{equation} with $L_{u_\bullet}$ collecting the terms that are linear in $\varphi$: \begin{equation}\begin{split}
				L_{u_\bullet}\varphi:=\langle \varphi_{\tf}, u_{\tf}\rangle & - \langle \varphi_0, u_0\rangle - \int_0^\tf \langle \partial_t \varphi_t, u_t\rangle \\& \hspace{1cm} - \frac12\int_0^\tf \int_{\TTd} \Delta \varphi_t(x)	u_t(x)^\alpha dtdx.	\end{split}	\end{equation} Since we assumed the finiteness of the supremum, it follows that $$ \sup\left\{L_{u_\bullet}\varphi:\varphi \in C^{1,2}([0,\tf]\times\TTd), \lllangle \varphi, \varphi\rrrangle_{u_\bullet} \le 1 \right\}<\infty$$ which implies that $L_{u_\bullet}$ extends to a bounded linear map on $H^1_{u_\bullet}$; the extension is unique because the equivalence classes of $C^{1,2}([0,\tf]\times\TTd)$ functions are dense in $H^1_{u_\bullet}$ by construction. Since $H^1_{u_\bullet}$ is a Hilbert space, it follows by the Riesz representation theorem that there exists $\psi \in H^1_{u_\bullet}$ such that $$ L_{u_\bullet} \varphi=\lllangle \varphi, \psi\rrrangle_{u_\bullet}$$ for all $\varphi\in H^1_{u_\bullet}$. It follows from (\ref{eq: hilbert form}) that \begin{equation}\begin{split}
					\sup_{\varphi\in C^{1,2}([0,\tf]\times\TTd)}\Xi_1(\varphi,u_\bullet) & =\sup_{\varphi \in C^{1,2}([0,\tf]\times\TTd)} \left\{L_{u_\bullet} \varphi-\frac{1}{2}\lllangle \varphi, \varphi\rrrangle_{u_\bullet}\right\} \\ &= \sup_{\varphi \in H^1_{u_\bullet}} \left\{\lllangle \varphi, \psi\rrrangle_{u_\bullet} -\frac{1}{2}\lllangle \varphi, \varphi\rrrangle_{u_\bullet}\right\} \end{split}
				\end{equation} which is readily computed to be $\frac{1}{2} \|\psi\|^2_{H^1_{u_\bullet}}$. By construction of $H^1_{u_\bullet}$, there exist $\psi_n\in C^{1,2}([0,\tf]\times\TTd)$ such that $\psi_n\to \psi$ in $H^1_{u_\bullet}$, and in particular, for all $\varphi\in C^{1,2}([0,\tf]\times\TTd)$, \begin{equation}\label{eq: wk formulation 1} \begin{split} L_{u_\bullet}\varphi  &= \lim_n \int_0^\tf \int_\TTd u_t(x)^\alpha \nabla \varphi_t(x)\cdot \nabla \psi_{n,t}(x) dt dx \\& =\lim_n \int_0^\tf \int_\TTd u^{\alpha/2}_t(x)\nabla \varphi_t(x)\cdot g_{n,t}(x)\end{split}				\end{equation} with $g_n(t,x):=u^{\alpha/2}_t(x)\nabla \psi_{n,t}(x)$, and \begin{equation} \label{eq: boundedness}
					 \lllangle \psi_n, \psi_n\rrrangle_{u_\bullet} \to \lllangle \psi, \psi\rrrangle_{u_\bullet}. 
				\end{equation} Using the definition, \begin{equation}
					\lllangle \psi_n, \psi_n\rrrangle_{u_\bullet}= \int_0^\tf \int_\TTd u_t(x)^\alpha|\nabla \psi_{n,t}(x)|^2 dt dx =  \left\|g_n\right\|_{{L^2_{t,x}}}^2.
				\end{equation} We can therefore find $g\in {L^2_{t,x}}$ and a subsequence $g_n\to g$ in the weak topology of ${L^2_{t,x}}$. Using the lower semicontinuity of the norm, we have that \begin{equation} \label{eq: norm of g} \begin{split}\left\|g\right\|_{{L^2_{t,x}}}^2 &\le \liminf_n \left\|g_n\right\|_{{L^2_{t,x}}}^2  \\[1ex] & =  \liminf_n \lllangle \psi_n, \psi_n\rrrangle_{u_\bullet} = \lllangle \psi, \psi\rrrangle_{u_\bullet} \\[1ex]  &= 2\sup_{\varphi \in C^{1,2}([0,\tf]\times\TTd)} \Xi_1(\varphi, u_\bullet). \end{split} \end{equation} As already remarked, $u\in  L^\alpha([0,\tf]\times\TTd)$, and in particular the test function $u^{\alpha/2}_t(x)\nabla \varphi_t(x)$ appearing in (\ref{eq: wk formulation 1}) belongs to ${L^2_{t,x}}$ for any $\varphi\in C^{1,2}([0,\tf]\times\TTd)$. Taking the limit, for all $\varphi\in C^{1,2}([0,\tf]\times\TTd)$, \begin{equation}
					L_{u_\bullet} \varphi = \int_0^\tf \int_{\TTd} u^{\alpha/2}_t(x)\nabla \varphi_t(x)\cdot g_t(x). 
				\end{equation} Using the definition of $L_{u_\bullet}\varphi$, this is exactly the statement that (\ref{eq: reformulation}) holds, and since $u_\bullet \in \cR$ by hypothesis, Lemma \ref{lemma: reformulation} shows that (\ref{eq: Sk}) holds for this choice of $g$. Returning to (\ref{eq: norm of g}), we have proven (\ref{eq: variational UB'}) and the step is complete; in conjunction with the previous step we have proven the assertion (\ref{eq: variational}) in Lemma \ref{lemma: variational form}.\bigskip \\ \paragraph{\textbf{Step 3. Characterisation of the minimiser.}}In light of (\ref{eq: variational LB}), and (\ref{eq: norm of g}), the second step in fact constructs $g$ attaining this minimum, and belonging to the set $\Lambda_{u_\bullet}$ defined in (\ref{eq: characterisation}). Uniqueness follows because  the feasible set  $\{g\in {L^2_{t,x}}: u_\bullet \in \cS_g(u_0)\}$ is an intersection of the affine, and hence convex, sets defined by (\ref{eq: reformulation}) for $\varphi$ running over $C^{1,2}([0,\tf]\times\TTd)$ and the function $g\mapsto \|g\|^2_{L^2_{t,x}}$ is strictly convex. Finally, let $g$ be the unique minimiser, and let $g'$ be any other control with $u_\bullet \in \cS_{g'}(u_0)$ and $g'\in \Lambda_{u_\bullet}$. It follows that, for any $\varphi\in C^{1,2}([0,\tf]\times\TTd)$, $$ \int_{[0,\tf]\times\TTd} \nabla \varphi_t(x)u_t(x)^{\alpha/2}\cdot (g'_t(x)-g_t(x)) dt dx=0$$ which implies that $g'-g$ belongs to the orthogonal complement $\Lambda_{u_\bullet}^\bot$ of $\Lambda_{u_\bullet}$ in ${L^2_{t,x}}$. On the other hand, $\Lambda_{u_\bullet}$ is a linear subspace, and so $g'-g\in \Lambda_{u_\bullet}$, and hence $g'-g=0$ in ${L^2_{t,x}}=L^2([0,\tf]\times \TTd, \RRd)$ as claimed.   \paragraph{\textbf{Step 4. Characterisation of the whole rate function.}} The final statement (\ref{eq: whole rate function variational}) follows from (\ref{eq: variational}) using the well-known Fenchel-Legendre transform \begin{equation}\begin{split}
					\alpha \cH_\rho(u_0)&=\sup\left\{\langle \psi, u_0\rangle - \alpha \int\rho(x)(e^{\psi(x)/\alpha}-1) dx: \psi \in C(\TTd)\right\} \\ & =\sup\left\{\Xi_0(\psi, u_0, \rho): \psi\in C(\TTd)\right\} \end{split}
				\end{equation} which can be verified in the same way as \cite[Lemma 6.2.13]{dembo2009large}.\bigskip \\ \paragraph{\textbf{Step 5. Properties of $\cA$.}}
				We now turn to some properties of the functional $\cA$ given in (\ref{eq: action 2}).	Let us remark that nowhere in the above does the diffusive term $\Delta u^{\alpha}$ appearing in (\ref{eq: Sk}) play a special role, and so we can repeat steps 1-2 to see that, for $u_\bullet \in \cR$,\begin{equation}
					\label{eq: dual A}\begin{split}\cA(u_\bullet)&=\frac{1}{2}\inf\left\{\left\|\theta\right\|_{{L^2_{t,x}}}^2: \text{(\ref{eq: CE}) holds}\right\}=\sup_{\varphi \in C^{1,2}([0,\tf]\times\TTd)} \left\{\widetilde{\Xi}_1(\varphi, u_\bullet) \right\}\end{split}
				\end{equation} for the functionals \begin{equation} \label{eq: def txi1}\begin{split}\widetilde{\Xi}_1(\varphi, u_\bullet):= \langle \varphi_{\tf}, u_{\tf}\rangle &- \langle \varphi_0, u_0\rangle -\int_0^\tf \langle \partial_t\varphi_t, u_t\rangle dt \\ &-\frac12\int_0^\tf \int_{\TTd} |\nabla \varphi_t(x)|^2 u_t(x)^\alpha dt dx. \end{split} \end{equation} Moreover, the same argument as in step 2 also constructs $\theta$ attaining the minimum and satisfying (\ref{eq: characterisation}). The argument for uniqueness of the minimiser and unique characterisation by membership of (\ref{eq: characterisation}) follow exactly as in step 3. \bigskip \\ We now turn to the assertion that $\cA$ is lower semicontinuous on sets of bounded $\cI_\rho$, $\rho\in (0,\infty)$, which will follow from the dual formulation (\ref{eq: dual A}). If we take a convergent sequence $u^{(n)}_\bullet \to u_\bullet$ in $\cDD$ with $\sup_n \cI_\rho(u^{(n)}_\bullet)<\infty$, then we can apply the \emph{a priori} estimate (\ref{eq: basic entropy estimate}) to see that $$\sup_n \int_0^\tf \cD_\alpha(u^{(n)}_t) dt<\infty.$$ Since we also assume the convergence $u^{(n)}_\bullet \to u_\bullet$ in the topology of $\cDD$, we are in the setting of Remark \ref{rmk: continnum nonstandard convolution} and it follows that $\|u^{(n)}_\bullet\to u_\bullet\|_{L^\alpha([0,\tf]\times\TTd)}\to 0$. It follows that $\widetilde{\Xi}_1(\varphi, u^{(n)}_\bullet)\to \widetilde\Xi_1(\varphi, u_\bullet)$ for all $\varphi\in C^{1,2}([0,\tf]\times\TTd)$, which immediately implies that \begin{equation} \begin{split} \cA(u_\bullet)&=\sup_{\varphi \in C^{1,2}([0,\tf]\times\TTd)} \widetilde{\Xi}_1(\varphi, u_\bullet) \\& \le \liminf_n \sup_{\varphi \in C^{1,2}([0,\tf]\times\TTd)} \widetilde{\Xi}_1(\varphi, u^{(n)}_\bullet)\\& =\liminf_n \cA(u^{(n)})  \end{split} \end{equation} which was the claim (\ref{eq: LSC of A}). 
				\end{proof}
  \fi
\fi

\end{document}